\documentclass[reqno]{amsart}%
\usepackage{palatino, mathpazo}
\usepackage{amsfonts}
\usepackage{amsmath}
\usepackage{amssymb,latexsym,xcolor}
\usepackage{graphicx}
\usepackage{harpoon}
\usepackage[mathscr]{eucal}
\usepackage{amssymb}%
\usepackage[
linktocpage=true,colorlinks,citecolor=magenta,linkcolor=blue,urlcolor=magenta]{hyperref}
\providecommand{\U}[1]{\protect \rule{.1in}{.1in}}

\newtheorem{theorem}{Theorem}[section]

\newtheorem{corollary}[theorem]{Corollary}

\newtheorem{lemma}[theorem]{Lemma}

\newtheorem{Theorem}{Theorem}

\theoremstyle{remark}
\newtheorem{remark}[theorem]{Remark}

\numberwithin{equation}{section}

\begin{document}
\title[Green functions on tori II: Rectangular torus]{Green functions, Hitchin's formula and curvature equations on tori II: Rectangular torus}
\author{Zhijie Chen}
\address{Department of Mathematical Sciences, Yau Mathematical Sciences Center,
Tsinghua University, Beijing, 100084, China }
\email{zjchen2016@tsinghua.edu.cn}
\author{Erjuan Fu}
\address{Beijing Institute of Mathematical Sciences and Applications, Beijing, 101408, China}
\email{fej.2010@tsinghua.org.cn, ejfu@bimsa.cn}
\author{Chang-Shou Lin}
\address{Department of Mathematics, National Taiwan University, Taipei 10617, Taiwan }
\email{cslin@math.ntu.edu.tw}

\begin{abstract}
Let $G(z)$ be the Green function on the flat torus $E_{\tau}=\mathbb{C}/(\mathbb{Z}+\mathbb{Z}\tau)$ with the singularity at $0$.  Lin and Wang (Ann. Math. 2010) proved that $G(z)$ has either $3$ or $5$ critical points (depending on the choice of $\tau$). Here
we study the sum of two Green functions which can be reduced to $G_p(z):=\frac12(G(z+p)+G(z-p))$. In Part I \cite{CFL}, we proved that for any $p$ satisfying $p\neq -p$ in $E_{\tau}$, the number of critical points of $G_p(z)$ belongs to $\{4,6,8,10\}$ (depending on the choice of $(\tau, p)$) and each number really occurs. 

In the Part II of this series, we study the important case $\tau=ib$ with $b>0$, i.e. $E_{\tau}$ is a rectangular torus. By developing a completely different approach from Part I, we show the existence of $8$ real values $d_1<d_2<\cdots<d_7<d_8$ such that if $$\wp(p)\in (-\infty, d_1]\cup [d_2, d_3]\cup [d_4, d_5]\cup [d_6, d_7]\cup [d_8,+\infty),$$
then $G_p(z)$ has no nontrivial critical points; if 
$$\wp(p)\in (d_1, d_2)\cup (d_3, d_4)\cup (d_5, d_6)\cup (d_7, d_8),$$
then $G_p(z)$ has a unique pair of nontrivial critical points that are always non-degenerate saddle points. This allows us to study the possible distribution of the numbers of critical points of $G_p(z)$ for generic $p$. Applications to the Painlev\'{e} VI equation and the curvature equation are also given.
\end{abstract}


\maketitle

\section{Introduction}

Let $\tau \in \mathbb{H}=\left \{  \tau\in\mathbb C|\operatorname{Im}\tau>0\right \}$, $\Lambda_{\tau}=\mathbb{Z}+\mathbb{Z}\tau$, and denote
$$\omega_{0}=0,\quad\omega_{1}=1,\quad\omega_{2}=\tau,\quad\omega_{3}=1+\tau.$$Let $E_{\tau}:=\mathbb{C}/\Lambda_{\tau}$ be a flat torus in the
plane and $E_{\tau}[2]:=\{ \frac{\omega_{k}}{2}|k=0,1,2,3\}+\Lambda
_{\tau}$ be the set consisting of the lattice points and half periods
in $E_{\tau}$.  

The Green function $G(z,w)=G(z,w;\tau)$ of the flat torus $E_{\tau}$ is the unique function that satisfies
\[
-\Delta_z G(z, w)=\delta_{w}-\frac{1}{\left \vert E_{\tau}\right \vert }\text{
\ on }E_{\tau},\quad
\int_{E_{\tau}}G(z,w)dxdy=0,
\]
where we use the complex variable $z=x+iy$, $\Delta_z=\frac{\partial^2}{\partial x^2}+\frac{\partial^2}{\partial y^2}=4\partial^2_{\bar z z}$ is the Laplace operator, $\delta_{w}$ is the Dirac measure at $w$ and $\vert E_{\tau}\vert$ is
the area of the torus $E_{\tau}$. By the translation invariance, we have $G(z,w)=G(z-w,0)$ and it is enough to consider the Green function
$G(z;\tau)=G(z):=G(z,0)$.
Clearly $G(z)$ is an even function on $E_{\tau}$ with the only
singularity at $0$, so $\frac{\omega_k}{2}$, $k\in \{1,2,3\}$, are always critical points of $G(z)$. 

\medskip

\noindent{\bf Definition.} {\it A critical point $a\in E_{\tau}$ of $G$ (resp. $G_p$; see below) is called trivial if $a=-a$ in $E_{\tau}$, i.e. $a\in E_{\tau}[2]$.  A critical point $a\in E_{\tau}$ is called nontrivial if $a\neq-a$ in $E_{\tau}$, i.e. $a\notin E_{\tau}[2]$.}
\medskip

Therefore, nontrivial critical points, if exist, must appear in pairs.
 Lin and Wang \cite{LW} studied critical points of $G(z)$ and proved the following remarkable result.
 
\begin{Theorem} \cite{LW} \label{thm-LW} $G(z)$ has at most one pair of nontrivial critical points, or equivalently, $G(z)$ has either $3$ or $5$ critical points (depends on the choice of $\tau$).
 For example, 
\begin{itemize}
\item[(1)] When $\tau=ib$ with $b>0$, i.e. $E_{\tau}$ is a rectangular torus,  $G(z)$ has exactly $3$ critical points  $\frac{\omega_k}{2}$, $k\in \{1,2,3\}$. Furthermore, $\frac{\omega_1}{2}$ and $\frac{\omega_2}{2}$ are both non-degenerate saddle points, while $\frac{\omega_3}{2}$ is a non-degenerate minimal point.
\item[(2)] When $\tau=\frac{1}{2}+ib$ with $b>0$, i.e. $E_{\tau}$ is a rhombus torus, there are $0<b_0<\frac12<b_1<\sqrt{3}/2$ such that $G(z)$ has exactly $5$ critical points if and only if $b\in (0, b_0)\cup (b_1,+\infty)$. 
\end{itemize}
 \end{Theorem}
 
See also Bergweiler and Eremenko \cite{BE} for a new proof of the first statement of Theorem \ref{thm-LW}. Later, Lin-Wang \cite{LW4} proved that nontrivial critical points of $G(z)$ must be minimal points.
Recently, we proved in \cite{CFL} that nontrivial critical points of $G(z)$ are always non-degenerate. We summarize these results together as follows.

\begin{Theorem}\cite{LW4,CFL}\label{thm-A0}
Suppose that the pair of nontrivial critical points $\pm q$ of $G(z)$ exists, then
\begin{itemize}
\item[(1)] all the three trivial critical points $\frac{\omega_k}{2}$ are non-degenerate saddle points of $G(z)$, that is, the Hessian $\det D^2G(\frac{\omega_k}{2})<0$ for $k=1,2,3$.
\item[(2)]  the $\pm q$ are non-degenerate minimal points of $G(z)$, i.e. $\det D^2G(\pm q)>0$.
 \end{itemize}
\end{Theorem}

It is well known that the non-degeneracy of critical points of Green functions has many important applications, such as in constructing bubbling solutions via the reduction method and proving the local uniqueness of bubbling solutions for elliptic PDEs; see e.g. \cite{BKLY, CLW, CL-2, LW4, LY} and references therein. Recently, we learned from Hao Chen that critical points of Green functions have also interesting applications in constructing minimal surfaces; see \cite{CW-TAMS,CT-SIAM}.

In view of Theorems \ref{thm-LW} and \ref{thm-A0}, naturally, we may ask the same question for the sum of two Green fuctions $G(z-p_1)+G(z-p_2)$. By changing variable $z\mapsto z+\frac{p_1+p_2}{2}$, we can always assume $p_2=-p_1$, so it is enough to study the Green function
 $G_p(z)=G_{-p}(z)$ defined by
\begin{equation}G_p(z;\tau)=G_p(z):=\frac12\big(G(z-p)+G(z+p)\big).\end{equation}
Note that if $p\in E_{\tau}[2]$, i.e. $p=\frac{\omega_k}{2}$ for some $k$, then $G_p(z)=G(z-p)$ and so Theorem \ref{thm-LW} implies that $G_p(z)$ has either $3$ or $5$ critical points. 
Thus, we only consider the case $p\in E_{\tau}\setminus E_{\tau}[2]$. Clearly $G_p(z)$ is also even, so $\frac{\omega_k}{2}$, $k=0,1,2,3$, are all trivial critical points of $G_p(z)$, i.e. the number of critical points of $G_p(z)$ is an even number of at least $4$. 
We generalized the first statement of Theorem \ref{thm-LW} to $G_p(z)$ in Part I \cite{CFL}.

\begin{Theorem}\cite{CFL}\label{main-thm-1} For any $p\in E_{\tau}\setminus E_{\tau}[2]$, $G_p(z)$ has at most $3$ pairs of nontrivial critical points, or equivalently, the number of critical points of $G_p(z)$ belongs to $\{4,6,8,10\}$, and each number in $\{4,6,8,10\}$ really occurs for different $(\tau, p)$'s. 

Moreover,  fix any $\tau$, then for almost all $p\in E_{\tau}\setminus E_{\tau}[2]$, all critical points of $G_p(z)$ are non-degenerate.
\end{Theorem}

Besides the motivation to generalize Theorems \ref{thm-LW} and \ref{thm-A0}, there are other two motivations from Painlev\'{e} VI equations and curvature equations for studying critical points of $G_p(z)$; see Section 1.2 below.

\subsection{New results}
In Part II of this series, we study the important case $\tau=ib$ with $b>0$, i.e. $E_{\tau}$ is a rectangular torus. By developing a completely different approach from Part I \cite{CFL}, we can give a complete answer for $\wp(p)\in\mathbb{R}$ as follows. Here $\wp(z)$ is the Weierstrass $\wp$-function that will be recalled soon.

\begin{theorem}\label{main-thm-b-2}
Let $\tau=ib$ with $b>0$. Then there are $8$ real values, denoted by
$$d_1<d_2<\cdots<d_7<d_8,$$
such that the following statements hold.
\begin{itemize}
\item[(1)] $G_p(z)$ has no nontrivial critical points for
$$\wp(p)\in (-\infty, d_1]\cup [d_2, d_3]\cup [d_4, d_5]\cup [d_6, d_7]\cup [d_8,+\infty).$$
\item[(2)] $G_p(z)$ has a unique pair of nontrivial critical points $\pm a_0$ that are always non-degenerate saddle points for
$$\wp(p)\in (d_1, d_2)\cup (d_3, d_4)\cup (d_5, d_6)\cup (d_7, d_8).$$
Furthermore, by writing $a_0=r_0+s_0\tau$ with $(r_0,s_0)\in\mathbb{R}^2\setminus\frac12\mathbb{Z}^2$, then
\begin{itemize}
\item[(2-1)] if $\wp(p)\in (d_1, d_2)$, then $a_0=\frac12+s_0\tau$.
\item[(2-2)] if $\wp(p)\in (d_3, d_4)$, then $a_0=r_0$.
\item[(2-3)] if $\wp(p)\in (d_5, d_6)$, then $a_0=s_0\tau$.
\item[(2-4)] if $\wp(p)\in (d_7, d_8)$, then $a_0=r_0+\frac12\tau$.
\end{itemize}
\end{itemize}
\end{theorem}

Here $\wp(z)=\wp( z;\tau)$ is the Weierstrass $\wp$-function with periods
$\Lambda_{\tau}$, defined by%
\[\wp(z):=\frac{1}{z^{2}}+\sum_{\omega \in \Lambda_{\tau
}\backslash\{0\}  }\left(  \frac{1}{(z-\omega)^{2}}-\frac
{1}{\omega^{2}}\right) ,
\]
which satisfies the well-known cubic equation
\[\wp^{\prime}(z)^{2}=4\wp(z)^{3}-g_{2}\wp
(z)-g_{3}=4\prod_{k=1}^3(\wp(z)-e_k),
\]
where $g_2=g_2(\tau), g_3=g_3(\tau)$ are known as invariants of the elliptic curve, and $e_k(\tau)=e_k:=\wp(\frac{\omega_k}{2})$ for $k=1,2,3$. It is well known that $\wp(\cdot): E_{\tau}\to \mathbb{C}\cup\{\infty\}$ is a double cover with branch points at $\{\frac{\omega_k}{2}\}_{k=0}^3$, i.e. for any $c\in \mathbb{C}\cup\{\infty\}$, there is a unique pair $\pm z_c\in E_{\tau}$ such that $\wp(\pm z_{c})=c$. Thus $p\in E_{\tau}\setminus E_{\tau}[2]$ is equivalent to $\wp(p)\in\mathbb C\setminus\{e_1, e_2, e_3\}$. We will prove in Theorem \ref{main-thm-b-22} that $e_2\in (d_2, d_3)$, $e_3\in (d_4, d_5)$ and $e_1\in (d_6, d_7)$ for $\tau=ib$ with $b>0$.

\begin{remark}\label{remk1-2}
Theorem \ref{main-thm-b-2} shows that $G_p(z)$ has at most one pair of nontrivial critical points as long as $\wp(p)\in\mathbb R$. This result can not be obtained from the general result Theorem \ref{main-thm-1}. Notice that this result does not necessarily hold for $\wp(p)\notin\mathbb{R}$, see Theorem \ref{thm-section5-5} below for example.  We will also prove in Lemma \ref{lemma3-10} that for $\tau=ib$ with $b\in (0, 2b_0)\cup(2b_1,+\infty)$ (where $b_0, b_1$ are given by Theorem \ref{thm-LW}), there exists $\varepsilon>0$ small such that for any $|p-\frac{1+\tau}{4}|<\varepsilon$, $G_{p}(z)$ has exactly $3$ pairs of nontrivial critical points, or equivalently has exactly $10$ critical points. 

It is also interesting to compare Theorem \ref{main-thm-b-2}-(2) with Theorem \ref{thm-A0}. Theorem \ref{thm-A0} says that the unique pair of nontrivial critical points of $G(z)$ are non-degenerate minimal points, while  Theorem \ref{main-thm-b-2}-(2) says that the unique pair of nontrivial critical points of $G_p(z)$ are non-degenerate saddle points. 
\end{remark}

Theorem \ref{main-thm-b-2} can be applied to obtain the following result.

\begin{theorem}\label{main-thm-b-23-0}
Let $\tau=ib$ with $b>0$ and $p\in E_{\tau}\setminus E_{\tau}[2]$. Suppose $G_p(z)$ has a nontrivial critical point $a$. Then \begin{align*}
\operatorname{Im}\wp(a)=0\quad\text{if and only if}\quad \operatorname{Im}\wp(p)=0,\\
\operatorname{Im}\wp(a)<0\quad\text{if and only if}\quad \operatorname{Im}\wp(p)>0,\\
\operatorname{Im}\wp(a)>0\quad\text{if and only if}\quad \operatorname{Im}\wp(p)<0.
\end{align*}
\end{theorem}

The approach of proving Theorem \ref{main-thm-b-2} is completely different from Part I \cite{CFL}. We will establish a deep connection between conditional stability sets of a second order linear ODE (see \eqref{GLE1} below)
and nontrivial critical points of $G_p(z)$. The advantage of this approach is that, when $\tau=ib$ with $b>0$ and $\wp(p)\in\mathbb{R}$, the conditional stability sets are easy to study because they admit certain symmetries.
Moreover, this approach allows us to write down $d_j$'s explicitly in terms of classical special functions. 
Let $\zeta(z)=\zeta(z;\tau):=-\int^{z}\wp(\xi)d\xi$ be the Weierstrass
zeta function with two quasi-periods
\[
\eta_{k}(\tau)=\eta_k:=2\zeta\Big(\frac{\omega_{k}}{2}\Big)=\zeta(z+\omega_{k}%
)-\zeta(z),\quad k=1,2.
\]
This $\zeta(z)$ is an odd meromorphic function with simple poles at $\Lambda_{\tau}$.
As in Part I \cite{CFL}, we define
\begin{equation}\label{B00}
\mathcal{B}_0:=\Big\{z\in\mathbb{C}\; :\; \Big|z-\Big(\frac{\pi}{\operatorname{Im}\tau}-\eta_1\Big)\Big|<\frac{\pi}{\operatorname{Im}\tau}\Big\},
\end{equation}
and for $k\in\{1,2,3\}$,
\begin{align}\label{alphak0}
\alpha_k:=\frac{\frac{\pi}{\operatorname{Im}\tau}-(\eta_1+e_k)}{3e_k^2-\frac{g_2}{4}},\quad \beta_k:=\frac{\pi}{|3e_k^2-\frac{g_2}{4}|\operatorname{Im}\tau}>0,
\end{align}
\begin{equation}\label{alphak1}
\mathcal{B}_k:=\begin{cases}\bigg\{z\in\mathbb{C}\; :\; \bigg|z-e_k-\frac{\overline{\alpha_k}}{|\alpha_k|^2-\beta_k^2}\bigg|<\frac{\beta_k}{\left||\alpha_k|^2-\beta_k^2\right|}\bigg\}\quad\text{if }|\alpha_k|\neq \beta_k,\\
\Big\{z\in\mathbb{C}\; :\; \operatorname{Re}(\alpha_k (z-e_k))>\frac12\Big\}\quad\text{if }|\alpha_k|=\beta_k,\end{cases}
\end{equation}
that is, $\mathcal{B}_k$ is either an open disk or an open half plane. It was proved in Part I \cite{CFL} that $|\alpha_k|=\beta_k$ is equivalent to that $\frac{\omega_k}{2}$ is a degenerate critical point of $G(z)$. We recall the following result about the degeneracy of trivial critical points.
\begin{Theorem}\cite{CFL}\label{main-thm-01} Let $k\in\{0,1,2,3\}$ and $p\neq \frac{\omega_k}{2}$. Then $\frac{\omega_k}{2}$ is a degenerate critical point of $G_p(z)$ if and only if $\wp(p-\frac{\omega_k}{2})\in \partial\mathcal{B}_0$, if and only if $\wp(p)\in \partial\mathcal{B}_k$, where $\mathcal{B}_k$ is defined by \eqref{B00}-\eqref{alphak1}.
\end{Theorem}

Let $\tau=ib$ with $b>0$. Then Theorem \ref{thm-LW} says that $\frac{\omega_k}{2}$ is a non-degenerate critical point of $G(z)$ for any $k\in \{1,2,3\}$, so $\mathcal{B}_k$'s are all open disks.
For example, the figures of $\mathcal{B}_k$'s for $\tau=i$ are presented numerically in Figure 1.
\begin{figure}[ht]
\label{O5-1}\includegraphics[width=2in]{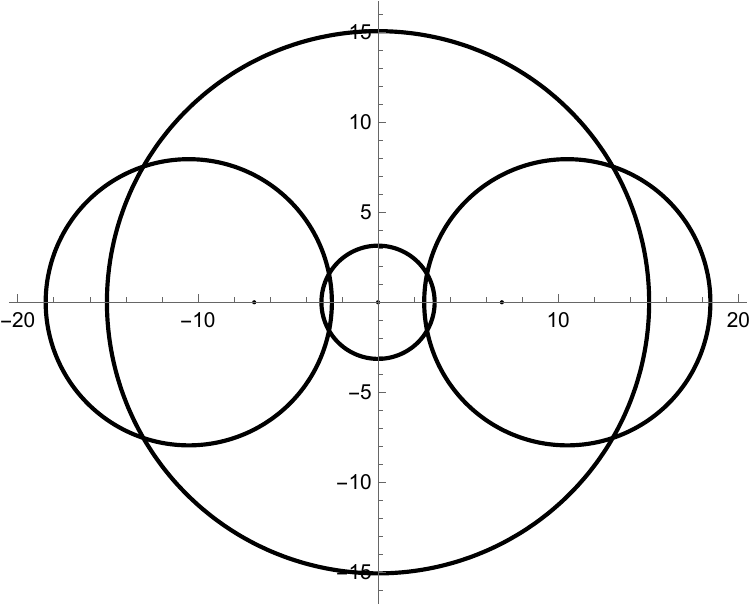}\caption{The four circles for $\tau=i$: the smallest circle for $\partial\mathcal{B}_0$, biggest for $\partial\mathcal{B}_3$, left for $\partial\mathcal{B}_1$ and right for $\partial\mathcal{B}_2$. We will prove in Theorem \ref{main-thm-b-22} that the relative positions of the four circles $\partial\mathcal{B}_k$'s are of this form for all $b>0$.}%
\end{figure}

Besides, since $\tau=ib$ with $b>0$, it is well known that $e_1, e_2, e_3, g_2, g_3,  \eta_1\in\mathbb{R}$, so the centers of $\mathcal{B}_k$'s all lie on $\mathbb{R}$, and then $\partial\mathcal{B}_k\cap\mathbb{R}$ contains two points for each $k$, which implies that $(\cup_k\partial\mathcal{B}_k)\cap\mathbb{R}$ contains at most $8$ points. 
Note from Figure 1 that $(\cup_k\partial\mathcal{B}_k)\cap\mathbb{R}$ contains exactly $8$ points for $b=1$. Our next result shows that $(\cup_k\partial\mathcal{B}_k)\cap\mathbb{R}$ contains exactly $8$ points that are precisely those $d_j$'s in Theorem \ref{main-thm-b-2} for all $b>0$.

\begin{theorem}\label{main-thm-b-22}  Let $\tau=ib$ with $b>0$. Let $d_1<\cdots<d_8$ be given by Theorem \ref{main-thm-b-2}.
Then $$(\cup_k\partial\mathcal{B}_k)\cap\mathbb{R}=\{d_1,\cdots,d_8\},$$and
\begin{align}\label{eqfc-21}
d_1&<\wp\Big(\frac\tau4\Big)<d_2<e_2<d_3<\wp\Big(\frac{1}{4}+\frac{\tau}2\Big)<d_4<e_3\nonumber\\
&<d_5<\wp\Big(\frac{1}{2}+\frac{\tau}{4}\Big)<d_6<e_1<d_7<\wp\Big(\frac14\Big)<d_8.
\end{align}
Furthermore,
{\allowdisplaybreaks
\begin{align}\label{fineq-1}
&d_1=e_1+\frac{3e_1^2-\frac{g_2}{4}}{\frac{2\pi}{b}-(e_1+\eta_1)},\quad
d_2=e_3+\frac{3e_3^2-\frac{g_2}{4}}{\frac{2\pi}{b}-(e_3+\eta_1)},\\
\label{fineq-2}&d_3=-\eta_1,\quad\quad\quad\quad\quad\quad\quad\;\;\;
d_4=e_1-\frac{3e_1^2-\frac{g_2}{4}}{e_1+\eta_1},\\
\label{fineq-3}&d_5=e_2+\frac{3e_2^2-\frac{g_2}{4}}{\frac{2\pi}{b}-(e_2+\eta_1)},\quad
d_6=\frac{2\pi}{b}-\eta_1,\\
\label{fineq-4}&d_7=e_3-\frac{3e_3^2-\frac{g_2}{4}}{e_3+\eta_1},\quad\quad\quad\;\;
d_8=e_2-\frac{3e_2^2-\frac{g_2}{4}}{e_2+\eta_1}.
\end{align}
}%
Consequently, the relative positions of the four circles $\partial\mathcal{B}_k$'s are the same as Figure 1 for all $b>0$.
\end{theorem}

In particular, by inserting \eqref{fineq-1}-\eqref{fineq-4} into \eqref{eqfc-21}, we obtain various inequalities about the classical objects $e_k$'s, $g_2$ and $\eta_1$ that hold for all  $\tau=ib$ with $b>0$. We believe that these inequalities have potential applications in future.

Thanks to Theorem \ref{main-thm-b-22} that holds for all $b>0$, we can determine more accurately the possible numbers of critical points of $G_p(z)$ for almost all $p$ as follows.

\begin{theorem}\label{main-thm-8}
Let $\tau=ib$ with $b>0$. Then the following statements hold.
\begin{itemize}

\item[(1)] Denote $$\Xi_1:=\mathcal{B}_1\setminus\overline{\mathcal{B}_3}\neq \emptyset,\quad \Xi_2:=\mathcal{B}_0\cap\mathcal{B}_1\neq \emptyset,$$
$$\Xi_3:=\mathcal{B}_0\cap\mathcal{B}_2\neq \emptyset,\quad{\Xi}_4:=\mathcal{B}_2\setminus \overline{\mathcal{B}_3}\neq \emptyset,$$
then $G_p(z)$ has exactly $6$ critical points for any $\wp(p)\in \Xi_1\cup\Xi_2\cup\Xi_3\cup\Xi_4$.

\item[(2)] Denote$$\Xi_5:=\mathbb{C}\setminus\cup_{k=0}^3 \overline{\mathcal{B}_k}\neq \emptyset,\quad \Xi_6:=\mathcal{B}_0\setminus(\{e_3\}\cup\cup_{k=1,2} \overline{\mathcal{B}_k})\neq \emptyset,$$
$$\Xi_7:=\mathcal{B}_1\cap\mathcal{B}_3\setminus(\{e_2\}\cup\overline{\mathcal{B}_0})\neq \emptyset,\quad{\Xi}_8:=\mathcal{B}_2\cap\mathcal{B}_3\setminus(\{e_1\}\cup\overline{\mathcal{B}_0})\neq\emptyset,$$
then $G_p(z)$ has exactly either $4$ or $8$ critical points for almost all $\wp(p)\in \Xi_5\cup \Xi_6\cup \Xi_7\cup{\Xi}_8$.

\item[(3)]  Denote $$\Xi_9:=\mathcal{B}_3\setminus\cup_{k=0,1,2} \overline{\mathcal{B}_k}\neq \emptyset,$$ then $G_p(z)$ has at least $6$ critical points for any $\wp(p)\in \Xi_9$, and $G_p(z)$ has exactly either $6$ or $10$ critical points for almost all $\wp(p)\in \Xi_9$.  
\end{itemize}
\end{theorem}

Theorem \ref{main-thm-8} for the special case $b=1$ was proved in Part I \cite{CFL}. The novelty of this Part II is that we can prove Theorem \ref{main-thm-8} for all $b>0$ by applying Theorem \ref{main-thm-b-22}. 
Note from Figure 1 that
$$(d_1,d_2)\subset\Xi_1,\quad (d_3,d_4)\subset\Xi_2, \quad (d_5,d_6)\subset\Xi_3,\quad (d_7,d_8)\subset\Xi_4,$$
$$(d_2,d_3)\setminus\{e_2\}\subset\Xi_7,\quad (d_4, d_5)\setminus\{e_3\}\subset\Xi_6,$$
$$(d_6,d_7)\setminus\{e_1\}\subset\Xi_8,\quad (-\infty,d_1)\cup(d_8, +\infty)\subset\Xi_5.$$
Theorem \ref{main-thm-b-2} shows that for each $5\leq j\leq 8$, $G_p(z)$ has exactly $4$ critical points at least for $\wp(p)\in \Xi_j\cap\mathbb R$.
Our next result shows that the number $8$ also occurs for some $\Xi_j$ and so Theorem \ref{main-thm-8}-(2) is sharp.

\begin{theorem}
\label{thm-section5-5} Let $\tau=ib$ with $b>0$. Then there exists $p$ satisfying $\wp(p)\notin\mathbb R$ such that $G_p(z)$ has degenerate nontrivial critical points.

More precisely, there exists $\Xi$ being one of $\Xi_5, \Xi_7$ such that the following holds: Define 
\[\Omega_N:=\bigg\{\wp(p)\in\Xi\,:\begin{array}{l}\text{$G_p(z)$ has exactly $N$ critical points}\\\text{that are all non-degenerate}\end{array}\bigg\}\;\text{for }N=4,8,\]
then $\Omega_4\neq \emptyset$ and $\Omega_{8}\neq \emptyset$ are both open subsets of $\Xi$. Furthermore, $\Xi\subset\overline{\Omega_4}\cup\overline{\Omega_{8}}$,  $\partial\Omega_4\cap\partial\Omega_{8}\cap\Xi\neq\emptyset$, and $G_p(z)$ has degenerate nontrivial critical points for any $\wp(p)\in(\partial\Omega_4\cup\partial\Omega_{8})\cap\Xi$. 
\end{theorem}

Notice from Theorem \ref{main-thm-1} that the four trivial critical points $\frac{\omega_k}{2}$, $k=0,1,2,3$, are always non-degenerate for any $\wp(p)\in \Xi$.
Theorem \ref{thm-section5-5} shows that the number of nontrivial critical points changes when $\wp(p)\in\Xi$ crosses $\partial\Omega_4\cap\partial\Omega_{8}$, because of the degeneracy of nontrivial critical points. 
Theorem \ref{thm-section5-5} also indicates that in Figure 1, becides the four circles $\partial \mathcal{B}_k$ consisting of those $\wp(p)$ such that $\frac{\omega_k}{2}$ is a degenerate trivial critical point of $G_p(z)$, there should be also some curves consisting of those $\wp(p)$ such that $G_p(z)$ has degenerate nontrivial critical points. This indicates that the precise distribution of the number of critical points of $G_p(z)$ could be very complicated when $\wp(p)\in\mathbb{C}$ varies. As a consequence, we obtain the following result.

\begin{corollary}\label{coro-main}
Let $\tau=ib$ with $b>0$, and $b_0, b_1$ be given in Theorem \ref{thm-LW}.
\begin{itemize}
\item[(1)] If $b\in (0, 2b_0)\cup (2b_1,+\infty)$, then for every $N\in\{4,6,8,10\}$, there exists $p\in E_{\tau}\setminus E_{\tau}[2]$ such that $G_p(z)$ has exactly $N$ critical points.
 \item[(2)] If $b\in [2b_0, 2b_1]$, then for every $N\in\{4,6,8\}$, there exists $p\in E_{\tau}\setminus E_{\tau}[2]$ such that $G_p(z)$ has exactly $N$ critical points.
 \end{itemize}
\end{corollary}

\begin{proof} This result follows directly from Theorems \ref{main-thm-b-2}, \ref{thm-section5-5} and Remark \ref{remk1-2}.
\end{proof}

Remark that Theorem \ref{main-thm-1} did not prove the existence of $\tau$ satisfying that, for such fixed $\tau$, each number of $\{4,6,8,10\}$ occurs for different $p$'s. Thus Corollary \ref{coro-main} can not follow from Theorem \ref{main-thm-1}. Corollary \ref{coro-main} motivates us to propose the following conjecture.

\medskip
\noindent{\bf Conjecture A.}  {\it Fix any $\tau$. Then for every $N\in\{4,6,8,10\}$, there exists $p\in E_{\tau}$ such that $G_p(z)$ has exactly $N$ critical points.}
\medskip

Conjecture A seems challenging and Corollary \ref{coro-main} proves this conjecture only for $\tau=ib$ with $b\in (0, 2b_0)\cup (2b_1,+\infty)$.

\subsection{Applications to Painlev\'{e} VI equation and curvature equation} 

As mentioned in Part I \cite{CFL}, our study of critical points of $G_p(z)$ has interesting applications to certain Painlev\'{e} VI equation and curvature equation. 

First,
let us consider the following elliptic form of the Painlev\'{e} VI equation with special parameters
\begin{equation}\label{PVI}
\frac{d^{2}p(\tau)}{d\tau^{2}}=\frac{-1}{32\pi^{2}}\sum_{k=0}^{3}
\wp^{\prime}\left( p(\tau)+\frac{\omega_{k}}{2};
\tau \right),\end{equation}
which was first studied by Hitchin \cite{Hit1}.
It is well known (see e.g. \cite{IKSY}) that the general Painlev\'{e} VI equation governs the isomonodromic deformation of a second order linear Fuchsian ODE on $\mathbb{CP}^1$ with five regular singular points. Recently, Chen, Kuo and Lin \cite{CKL-JMPA2016} proved that the elliptic form of the Painlev\'{e} VI equation also governs the isomonodromic deformation of a generalized Lam\'{e} equation (GLE for short). 
Then a natural question arises: \emph{How does the monodromy of the associated GLE effects the property of the solution of the elliptic form of the Painlev\'{e} VI equation}?

In this paper, we study this question for \eqref{PVI} by applying Theorem \ref{main-thm-b-2}.
Note that the associated GLE for \eqref{PVI} is
\begin{align}\label{GLE1}
y''(z)=\Big[&\frac{3}{4}
\big(\wp(z+p)+\wp(z-p)-\wp(2p)\big)\nonumber\\
&+A\big(\zeta(z+p)-\zeta(z-p)-\zeta(2p)\big)+A^{2}\Big]y(z),\;z\in\mathbb{C}.
\end{align}
Let us explain the relation between \eqref{PVI} and critical points of $G_p(z)$.
By using the complex variable $z\in\mathbb C$ and write $z=r+s\tau$ with $r,s\in \mathbb{R}$, it was proved in \cite{LW} that
\begin{equation}
\label{G_z}-4\pi \frac{\partial G}{\partial z}(z)=\zeta(z)-r\eta_{1}
-s\eta_{2}.
\end{equation}
Note that $z\notin E_{\tau}[2]$ is equivalent to $(r,s)\in\mathbb{R}^2\setminus\frac12\mathbb{Z}^2$, where $\frac12\mathbb{Z}^2:=\{(r,s)\,:\, 2r,2s\in\mathbb Z\}.$
Then $a=r+s\tau$ with $(r,s)\in\mathbb{R}^2\setminus\frac12\mathbb{Z}^2$ is a nontrivial critical point of $G_p$ if and only if 
\begin{equation}\label{a+001p}\zeta(a+p)+\zeta(a-p)-2(r\eta_1+s\eta_2)=0.\end{equation}
Using the additional formula of elliptic functions
\begin{equation}\label{eqfc-add}\zeta(z+w)+\zeta(z-w)-2\zeta(z)=\frac{\wp'(z)}{\wp(z)-\wp(w)},\end{equation}
it is easy to see that \eqref{a+001p} is equivalent to
\begin{align}\label{513-1} 
\wp(p)=\wp (r+s\tau)+\frac{\wp ^{\prime }(r+s\tau)}{%
2(\zeta(r+s\tau)-r\eta_1-s\eta_2)}.
\end{align}
The \eqref{513-1} is well known as Hitchin's formula \cite{Hit1}, where Hitchin studied Einstein metrics and proved that for any fixed $(r,s)\in \mathbb{C}^{2}\setminus \frac{1}{2}\mathbb{Z}^{2}$, the $p_{r,s}(\tau)$ defined by $\wp(p_{r,s}(\tau))=\text{RHS of \eqref{513-1}}$, as a function of $\tau\in \mathbb{H}$, is a solution of the elliptic form \eqref{PVI}.
Furthermore, those solutions $p_{r,s}(\tau)$ with one of $(r,s)$ real and the other one purely imaginary have important applications to Einstein metrics, where the case $r\in\mathbb R$ and $s\in i\mathbb R$ (resp. $r\in i\mathbb R$ and $s\in \mathbb R$) is related to Einstein metrics with positive (resp. negative) scalar curvature. See \cite[Theorem 6]{Hit1} for details.

Remark that for the solution $p_{r,s}(\tau)$ of \eqref{PVI},  the monodromy group of the associated GLE \eqref{GLE1} with $p=p_{r,s}(\tau)$ and 
$$A=\frac{1}{2}\left[  \zeta(p_{r,s}(\tau)+r+s\tau)+\zeta(p_{r,s}(\tau)-r-s\tau)-\zeta(2p_{r,s}(\tau))\right] $$
is generated by $-I_2=-\bigl(\begin{smallmatrix}1 & 0\\
0& 1\end{smallmatrix}\bigr)$, $N_1=
\bigl(\begin{smallmatrix}
e^{-2\pi is} & 0\\
0 & e^{2\pi is}%
\end{smallmatrix}\bigr)
$ and $
N_2=\bigl(\begin{smallmatrix}
e^{2\pi ir} & 0\\
0 & e^{-2\pi ir}%
\end{smallmatrix}\bigr).$ See \cite{CKL-PAMQ} and also Section \ref{section2} for a brielf review.
In particular, $(r,s)\in\mathbb R^2\setminus\frac12\mathbb{Z}^2$ implies that the monodromy group of \eqref{GLE1} is a subgroup of the unitary group $SU(2)$, or we simply say that the monodromy of \eqref{GLE1} is unitary. 
Note that the RHS of \eqref{513-1} is invariant if we replace $(r,s)$ with any element of $\pm(r,s)+\mathbb{Z}^2$. Thus, for $(r,s)\in\mathbb R^2\setminus\frac12\mathbb{Z}^2$, we only need to consider $(r,s)\in [-\frac12, \frac12]\times [0,\frac12]\setminus\frac12\mathbb{Z}^2= I\cup II\setminus\frac12\mathbb{Z}^2$, where
$$I:=\Big[0,\frac12\Big]^2,\qquad II:=\Big[-\frac12,0\Big]\times\Big[0,\frac12\Big].$$
We will prove that Theorem \ref{main-thm-b-23-0} is equivalent to the following result. 

\begin{theorem}\label{main-thm-b-23} 
Let
 $p_{r,s}(\tau)$ be a solution of the elliptic form \eqref{PVI} such that the monodromy of the associated GLE \eqref{GLE1} is unitary, i.e. $(r,s)\in [-\frac12, \frac12]\times [0,\frac12]\setminus\frac12\mathbb{Z}^2$. Then the following statements hold.
\begin{itemize}
\item[(1)] If $(r,s)\in \partial I\cup \partial II\setminus\frac{1}{2}\mathbb{Z}^2$, then $\wp(p_{r,s}(\tau))\in\mathbb R$ for any $\tau=ib$ with $b>0$.
\item[(2)] If $(r,s)\in I^\circ=(0,\frac12)^2$, then $\operatorname{Im}\wp(p_{r,s}(\tau))>0$ for any $\tau=ib$ with $b>0$.
\item[(3)] If $(r,s)\in II^\circ=(-\frac12,0)\times(0,\frac12)$, then $\operatorname{Im}\wp(p_{r,s}(\tau))<0$ for any $\tau=ib$ with $b>0$.
\end{itemize}
\end{theorem}

Secondly, we apply Theorem \ref{main-thm-b-2} to obtain a sharp existence and non-existence result for the curvature equation.

\begin{theorem}
Let $\tau=ib$ with $b>0$. Let $d_1<\cdots<d_8$ be given by Theorem \ref{main-thm-b-22}. Then the following statements hold.
\begin{itemize}
\item[(1)]If $$\wp(p)\in (-\infty, d_1]\cup [d_2, d_3]\cup [d_4, d_5]\cup [d_6, d_7]\cup [d_8,+\infty),$$ then the curvature equation
\begin{equation}\label{mfe}
\Delta u+e^u=4\pi (\delta_p+\delta_{-p})\quad\text{on }\;E_{\tau}
\end{equation}
 has no solutions.

\item[(2)] If
$$\wp(p)\in (d_1, d_2)\cup (d_3, d_4)\cup (d_5, d_6)\cup (d_7, d_8),$$
then \eqref{mfe} has a unique one-parameter scaling family of solutions $u_{\beta}(z)$, where $\beta>0$ is arbitrary. Furthermore, $u_1(z)=u_1(-z)=u_1(\bar z)$.
\end{itemize}
\end{theorem}

\begin{proof}
We proved in Part I \cite[Theorem 1.13]{CFL} that there is a one-to-one correspondence between pairs of nontrivial critical points of $G_p(z)$ and one-parameter scaling families of solutions of \eqref{mfe}. Furthermore, among each one-parameter scaling family of solutions $u_{\beta}(z)$ with $\beta>0$, $u_{\beta}(z)$ blows up as $\beta\to+\infty$ and the blowup point is a nontrivial critical point of $G_p(z)$. Besides, $u_{\beta}(z)$ is even (i.e. $u_{\beta}(z)=u_{\beta}(-z)$) if and only if $\beta=1$. Thus this theorem follows from Theorem \ref{main-thm-b-2}. Note that in the statement (2), we have that $u_1(z)$ is the unique even solution of  \eqref{mfe}. Since $\tau=ib$ and $\wp(p)\in\mathbb{R}$, we will see in Section \ref{section4} that $\bar p=\pm p$ in $E_{\tau}$, so $u_1(\bar z)$ is also an even solution of \eqref{mfe}, which implies $u_1(z)=u_1(\bar z)$ by the uniqueness of even solutions. 
\end{proof}

The rest of this paper is organized as follows. In Section \ref{section2}, we briefly review the basic theory of GLE \eqref{GLE1}, and establish its connection with $G_p(z)$. In Section \ref{section3}, we study the conditional stability sets of GLE \eqref{GLE1}, and establish its deep connection with nontrivial critical points of $G_p(z)$. We consider $\tau=ib$ with $b>0$ and $\wp(p)\in\mathbb R$ from Section \ref{section4}. In Section \ref{section4}, we give a precise characterization of the conditional stability sets, which are applied to prove Theorems \ref{main-thm-b-2} and \ref{main-thm-b-22} in Section \ref{section5}. In Section \ref{section7}, we apply Theorems \ref{main-thm-b-2} and \ref{main-thm-b-22} to prove Theorems \ref{main-thm-b-23-0}, \ref{main-thm-8}, \ref{thm-section5-5} and \ref{main-thm-b-23}. Finally, we give another application of the conditional stability sets in Section \ref{section6}.

\section{Preliminary: Generalized Lam\'{e} equation}
\label{section2}

In this section, we briefly review the basic theory of 
 the following second order generalized Lam\'{e} equation (GLE for short)
\begin{align}\label{GLE}
y''(z)=I(z; p, A, \tau)y(z),\quad z\in\mathbb{C},
\end{align}
with the potential given by
\begin{align}\label{potential-I}
I(z; p, A, \tau):=&\frac{3}{4}
\big(\wp(z+p)+\wp(z-p)-\wp(2p)\big)\nonumber\\
&+A\big(\zeta(z+p)-\zeta(z-p)-\zeta(2p)\big)+A^{2},
\end{align}
where
$p\in E_{\tau}\setminus E_{\tau}[2]$ and $A\in \mathbb{C}$.
Then a deep connection between GLE \eqref{GLE} and the Green function $G_p(z)$ will be established; see Theorem \ref{rmk2-5}. This connection will be used to study nontrivial critical points of $G_p(z)$ in subsequent sections.

Let $y_{1}(z),y_{2}(z)$ be any two solutions of \eqref{GLE} and set
$\Phi(z)=y_{1}(z)y_{2}(z).$ Then $\Phi(z)$ satisfies the second
symmetric product equation for \eqref{GLE}:%
\begin{equation}
\Phi^{\prime \prime \prime}(z)-4I(z; p, A, \tau)\Phi^{\prime}(z)-2I^{\prime}(z; p, A, \tau)\Phi(z)=0.
\label{303-1}%
\end{equation}

GLE \eqref{GLE}  is of Fuchsian type with regular singularities at $\pm p$.
The local exponents of GLE (\ref{GLE}) at $\pm p$ are $-\frac12$, $\frac32$, so solutions might have logarithmic singularities at $\pm p$.
It was proved in \cite{CKL-JMPA2016} that $\pm p$ are always apparent singularities, i.e. all solutions of \eqref{GLE} are free of
logarithmic singularities at $\pm p$. 
Consequently, the local monodromy matrix at $\pm p$ is $-I_2$, where $I_2$ is the identity matrix. Denote by $L_p\subset\{t_1+t_2\tau : t_1,t_2\in[-\frac12,\frac12]\}$ the straight segment crossing $0$ and connecting $\pm p$. Then by analytic continuation, any solution $y(z)$ of GLE (\ref{GLE}) can be viewed as a single-valued
meromorphic function in $\mathbb{C}\backslash(L_p+\Lambda_{\tau})$, and in this
region $y(-z)$ and $y(z+\omega_j)$ are well-defined. Here
$L_p+\Lambda_{\tau}=\cup_{\omega\in\Lambda_{\tau}} (\omega+L_p)$ is the preimage of $L_p$ under the projection $\pi: \mathbb{C}\to E_{\tau}$.

Fix any base point $q_{0}\in E_{\tau}\backslash\{ \pm 
p\}$ such that any of $q_0+\mathbb{R}, q_0+\tau\mathbb{R}, q_0+(2\tau-1)\mathbb{R}$ has no intersection with $\pm p+\Lambda_{\tau}$. The monodromy representation of GLE \eqref{GLE} is a group homomorphism $\rho:\pi_{1}(  E_{\tau}\backslash\{\pm p\},q_{0})  \rightarrow
SL(2,\mathbb{C})$ defined as follows. Take any basis of solutions $(y_1(z), y_2(z))$ of GLE  \eqref{GLE}. For any loop $\gamma\in \pi_{1}(  E_{\tau}\backslash
\{\pm p\},q_{0})$, let $\gamma^*y(z)$ denote the analytic continuation of $y(z)$ along $\gamma$. Then $\gamma^*(y_1(z), y_2(z))$ is also a basis of solutions, so there is a matrix $\rho(\gamma)\in SL(2,\mathbb{C})$ such that
$$\gamma^*\begin{pmatrix}
y_{1}(z)\\
y_{2}(z)
\end{pmatrix}
=\rho(\gamma)
\begin{pmatrix}
y_{1}(z)\\
y_{2}(z)
\end{pmatrix}.$$
Here $\rho(\gamma)\in SL(2,\mathbb{C})$ (i.e. $\det \rho(\gamma)=1$) follows from the fact that the Wronskian $y_{1}(z)y_2'(z)-y_1'(z)y_2(z)$ is a nonzero constant. The image of $\rho:\pi_{1}(  E_{\tau}\backslash\{\pm p\},q_{0})  \rightarrow
SL(2,\mathbb{C})$ is called the monodromy group of \eqref{GLE}, which is a subgroup of $SL(2,\mathbb{C})$.

Let
$\gamma_{\pm}\in \pi_{1}(  E_{\tau}%
\backslash\{  \pm p\}  ,q_{0})$ be a
simple loop encircling $\pm p$ counterclockwise respectively, and $\ell_{j}%
\in \pi_{1}(  E_{\tau}%
\backslash\{\pm p\}  ,q_{0})  $, $j=1,2$, be two
fundamental cycles of $E_{\tau}$ connecting $q_{0}$ with $q_{0}+\omega_{j}$
such that $\ell_{j}\cap (L_p+\Lambda_{\tau})=\emptyset$ and satisfies%
\begin{equation}
\gamma_{-}\gamma_{+}=\ell_{1}\ell_{2}\ell_{1}^{-1}\ell_{2}^{-1}\text{ \  \ in
}\pi_{1}\left(  E_{\tau}\backslash \left \{\pm p\right \}
,q_{0}\right)  . \label{II-iv1}%
\end{equation}
Denote by $$N_j=N_j(A):=\rho(\ell_j),\quad j=1,2.$$
Since
$
\rho(\gamma_{\pm})=-I_{2}$, 
it follows from \eqref{II-iv1} that $N_1N_2=N_2N_1$, 
and the monodromy group of (\ref{GLE}) is generated by $\{-I_{2},N_1,N_2\}$ and is abelian. So there is a common eigenfunction $y_{1}(z)$ of all
monodromy matrices. Let
$\varepsilon_{j}$ be the eigenvalue of $N_j$: $$\ell_{j}^{\ast}y_{1}(z)=\varepsilon
_{j}y_{1}(z),\quad j=1,2.$$As mentioned before, $y_{1}(z)$ can be viewed as a single-valued
meromorphic function in $\mathbb{C}\backslash(L_p+\Lambda_{\tau})$, and in this
region, $y_{1}(-z)$ and $y_1(z+\omega_j)$ are well-defined. Then it follows from $\ell_{j}\cap (L_p+\Lambda_{\tau})=\emptyset$ that
\begin{equation}
y_{1}(z+\omega_{j})=\ell_{j}^{\ast}y_{1}(z)=\varepsilon_{j}y_{1}(z),\quad j=1,2. \label{304-111}%
\end{equation}

Let $y_{2}(z)=y_{1}(-z)$ in $\mathbb{C}\backslash(L_p+\Lambda_{\tau})$. Clearly
$y_{2}(z)$ is also a solution of (\ref{GLE}) and (\ref{304-111}) implies
\begin{equation}
y_{2}(z+\omega_{j})=\ell_{j}^{\ast}y_{2}(z)=\varepsilon_{j}^{-1}%
y_{2}(z),\quad j=1,2, \label{304-12}%
\end{equation}
i.e. $y_{2}(z)$ is also an eigenfunction with eigenvalue $\varepsilon_{j}^{-1}$.
Define
\[
\Phi_{e}(z)=\Phi_{e}(z; A):=y_{1}(z)y_{2}(z)=y_{1}(z)y_{1}(-z).
\]
Obviously, $\pm p$ are no longer branch points of $\Phi
_{e}(z)$, so $\Phi_{e}(z)$ is single-valued meromorphic in
$\mathbb{C}$. By (\ref{304-111})-(\ref{304-12}), $\Phi_{e}(z)$ is an even
elliptic function. Therefore, $\Phi_{e}(z)$ is an even
elliptic solution of the third order ODE \eqref{303-1}. It was proved in \cite[Proposition 2.9]{Takemura} that this $\Phi_{e}(z)$ is the unique even elliptic solution of \eqref{303-1} up to multiplying a nonzero constant. Here we need to compute the expression of  $\Phi_{e}(z)$ in terms of elliptic functions.

\begin{lemma}
Up to multiplying a nonzero constant, 
\begin{equation}\label{phie}\Phi_{e}(z)=\Phi_{e}(z; A)=\zeta(z+p)-\zeta(z-p)-2A-\zeta(2p).\end{equation}
\end{lemma}

\begin{proof}
Note that the local exponents of $\Phi_{e}(z)$ at $\pm p$ belong to $\{-1, 1, 3\}$. Since $\Phi_e$ must have poles and is even, so the local exponents of $\Phi_{e}(z)$ at $\pm p$ are both $-1$, and then up to multiplying a nonzero constant, we may assume
$$\Phi_{e}(z)=\zeta(z+p)-\zeta(z-p)+d_1$$
for some constant $d_1\in\mathbb{C}$. Inserting this formula into \eqref{303-1} and expanding the left hand side of \eqref{303-1} as power series of $(z-p)$, we see from the coefficient of $(z-p)^{-3}$ that $d_1=-2A-\zeta(2p)$. 
\end{proof}

\begin{lemma}
\label{lemma2-2} Let $\Phi_e(z)=y_1(z)y_2(z)=y_{1}(z)y_{1}(-z)$ be the even elliptic solution of %
\eqref{303-1} given by \eqref{phie}, and define $W:=y_1(z)y_2'(z)-y_1'(z)y_2(z)$ to be the Wronkian of $y_1$ and $y_2$. Then
\begin{equation}  \label{eq218}
Q(A)=Q_{p,\tau}(A):=\Phi_e'(z)^{ 2}-2\Phi_e^{\prime \prime }(z)\Phi_e(z)+4I(z;
p,A,\tau)\Phi_e(z)^2
\end{equation}
is a polynomial of $A$ with the leading term $16A^4$, and $W^2=Q(A)$. Furthermore,
\begin{equation}  \label{eq219}
y_1(z)=\sqrt{\Phi_e(z)}\exp\int^z\frac{-W}{2\Phi_e(\xi)}d\xi.
\end{equation}
\end{lemma}

\begin{proof} First, it follows from \eqref{303-1} that $Q(A)$ is independent of $z$. Using \eqref{potential-I} and \eqref{phie}, we see that $Q(A)$ is a polynomial of $A$ with the leading term $16A^4$. 

It suffices to prove $W^2=Q(A)$ and \eqref{eq219}, which has been proved in \cite{CKL1, Takemura}. Indeed,
since $\Phi_e(z)=y_1(z)y_2(z)$ and $W=y_1y_2^{\prime }-y_1^{\prime }y_2$, we
have
\begin{equation*}
\frac{y_1^{\prime }}{y_1}=\frac{\Phi_e^{\prime }-W}{2\Phi_e},\quad\frac{%
y_2^{\prime }}{y_2}=\frac{\Phi_e^{\prime }+W}{2\Phi_e},
\end{equation*}
which implies \eqref{eq219} and
\begin{equation*}
\frac{\Phi_e^{\prime \prime }}{2\Phi_e}-\frac{\Phi_e^{\prime }-W}{2\Phi_e^2}%
\Phi_e^{\prime }=\left(\frac{y_1^{\prime }}{y_1}\right)^{\prime }=\frac{%
y_1^{\prime \prime }}{y_1}-\left(\frac{y_1^{\prime }}{y_1}\right)^2=I(z; p, A, \tau)-\left(%
\frac{\Phi_e^{\prime }-W}{2\Phi_e}\right)^2,
\end{equation*}
\begin{equation*}
\frac{\Phi_e^{\prime \prime }}{2\Phi_e}-\frac{\Phi_e^{\prime }+W}{2\Phi_e^2}%
\Phi_e^{\prime }=I(z; p, A, \tau)-\left(\frac{\Phi_e^{\prime }+W}{2\Phi_e}\right)^2.
\end{equation*}
Adding these two formulas together, we obtain $W^2=Q(A)$.
\end{proof}

Let $\sigma(z;\tau)=\sigma(z):=\exp(\int^{z}\zeta(\xi)d\xi)$ be the Weierstrass sigma function, which is an odd
entire function with simple zeros at $\Lambda_{\tau}$, and satisfies the following transformation law
\begin{equation}  \label{123123}
\sigma(z+\omega_j)=-e^{\eta_j(z+\frac{1}{2}\omega_j)}\sigma(z),\quad j=1,2.
\end{equation}
Then the explicit expression of the common eigenfunction $y_1(z)$ can be written down by $\sigma(z)$, which has been studied in \cite{CKL1}.

\begin{lemma}\cite[Section 3]{CKL1}\label{lemma2-3} Let $y_{1}(z)$ be
the common eigenfunction mentioned above. Then there is $a=a(A)\neq \pm p$ in $E_{\tau}$ such that up to multiplying a nonzero
constant, 
\begin{equation}\label{exp1}
y_{1}(z)=y_{a}(z):=\frac{e^{c(a)z}\sigma(z-a)}%
{\sqrt{\sigma(z-p)\sigma(z+p)}},
\end{equation}
with
\begin{equation}
c(a)=\frac{1}{2}[\zeta
(a+p)+\zeta(a-p)].\label{61-38}
\end{equation}
Furthermore, $A$ can be expressed by $a$ as 
\begin{equation}
A=\frac{1}{2}\left[  \zeta(p+a)+\zeta(p-a)-\zeta(2p)\right]  . \label{Aap}%
\end{equation}
\end{lemma}

\begin{proof}
We sketch the proof for completeness. Since the local exponents of $y_1(z)^2$ at $\pm p$ belong to $\{-1,3\}$, $y_1(z)^2$ has no branch points and so is meromorphic in $\mathbb{C}$. By \eqref{304-111} we have $y_1(z+\omega_j)^2=\varepsilon_j^2y_1(z)^2$, so $y_1(z)^2$ is an elliptic function of the second kind. Then the theory of elliptic functions implies that up to multiplying a nonzero constant, there are $a\neq \pm p$ in $E_{\tau}$ and $c\in\mathbb{C}$ such that $$y_1(z)^2=\frac{e^{2cz}\sigma(z-a)^2}{\sigma(z-p)\sigma(z+p)}.$$
This proves \eqref{exp1}. Inserting \eqref{exp1} into the linear ODE \eqref{GLE} and computing the Laurent expansions at $z=\pm p$, we easily obtain \eqref{61-38} and \eqref{Aap}.
\end{proof}

Observe that in $\mathbb{C}\backslash(L_p+\Lambda_{\tau})$,
\begin{equation}
y_2(z)=y_{{a}}(-z)=y_{-{a}}(z)\text{ \ up to multiplying a nonzero constant,} \label{exp11}%
\end{equation}
which infers that
\begin{align}
\Phi_{e}(z)=y_{a}(z)y_{-a}(z) \text{ \ up to multiplying a nonzero constant.}\label{a-a1}%
\end{align}
By the uniqueness of $\Phi_{e}(z)$, we see that 
\begin{equation} \label{eq-data1}
  \parbox{\dimexpr\linewidth-5em}{$\pm a
\in E_{\tau}$ is uniquely determined for given GLE \eqref{GLE} (i.e. uniquely determined by $A$ for fixed $p, \tau$),
  }
\end{equation}
and for different representatives $a, \tilde{a}\in a+\Lambda_{\tau}$ of the same $a\in E_{\tau}$, we have
\[
y_{a}(z)=y_{\tilde{a}}(z)\text{ \ up to multiplying a nonzero
constant}
\]
by using the transformation law \eqref{123123}.

\begin{lemma} We have
\[Q(A)=16\prod_{k=0}^3(A-A_k),\]
where
\begin{align}\label{Ak}
A_k:=&\frac{1}{2}\left[  \zeta\Big(p+\frac{\omega_k}{2}\Big)+\zeta\Big(p-\frac{\omega_k}{2}\Big)-\zeta(2p)\right] \nonumber\\
=&\begin{cases}
-\frac{\wp''(p)}{4\wp'(p)} & k=0,\\
A_0+\frac{\wp'(p)}{2(\wp(p)-e_k)}&k=1,2,3,
\end{cases}
\end{align}
satisfies $\sum_{k=0}^{3}A_k=0$.
\end{lemma}

\begin{proof}
Let $Q(A)=0$. Then $y_{a}(z)$ and $y_{-a}(z)$ are linearly dependent, so $a=-a$ in $E_{\tau}$, namely $a=\frac{\omega_k}{2}$ and then \eqref{Aap} implies
$$A=\frac{1}{2}\left[  \zeta\Big(p+\frac{\omega_k}{2}\Big)+\zeta\Big(p-\frac{\omega_k}{2}\Big)-\zeta(2p)\right]=:A_k.$$
Conversely, for each $k\in\{0,1,2,3\}$, a direct computation shows that $y_{\frac{\omega_k}{2}}(z)$ is a solution of GLE \eqref{GLE}  with $A=A_k$, and $y_{\pm\frac{\omega_k}{2}}(z)$ are linearly dependent, so $Q(A_k)=0$. Together with $\deg Q(A)=4$ with leading coefficient $16$, we conclude that $Q(A)=16\prod_k(A-A_k)$. Using the additional formula of elliptic functions
$$2\zeta(z)-\zeta(2z)=-\frac{\wp''(z)}{2\wp'(z)}$$ and \eqref{eqfc-add},
we obtain \eqref{Ak}. Finally, by
\begin{align*}
&\wp'(z)^2 = 4\wp(z)^3 - g_2\wp(z) - g_3=4\prod_{k=1}^3(\wp(z)-e_k), \\  & \wp''(z) = 2\sum_{k=1}^3\prod_{j\in\{1,2,3\}\setminus\{k\}}(\wp(z)-e_j), \end{align*}
we have
$$\sum_{k=0}^3A_k=\sum_{k=1}^3\frac{\wp'(p)}{2(\wp(p)-e_k)}-\frac{\wp''(p)}{\wp'(p)}=0.$$
The proof is complete.
\end{proof}

To establish the connection with the Green function $G_p(z)$, we need to compute the monodromy matrices $N_j(A)$ explicitly.
To avoid the branch points $\pm p$ of $y_{\pm a}(z)$, as in \cite{CKL1}, we define
\begin{equation*}
\Psi_{p}(z):=\frac{\sigma(z) }{\sqrt{\sigma(z-p) \sigma (z+p)}}.
\end{equation*}
By using the transformation law \eqref{123123},
we see that $\Psi_p(z)^2$ is an {elliptic function}. Since $\ell_{j}, j=1,2,$ are the two fixed fundamental
cycles of $E_{\tau}$ which do not intersect with $L_p+\Lambda_{\tau}$, it was proved in  \cite{CKL1} that
 the analytic continuation of $\Psi_p(z)$ along $\ell_{j}$ satisfies%
\begin{equation}
\ell_{j}^{\ast}\Psi_{p}(z)=\Psi_{p}(z),\quad j=1,2.   \label{304-1}
\end{equation}
Write
\begin{equation}
y_{\pm a}(z)=\tilde{y}_{\pm a}(z)\Psi_{p}(z),   \label{y-}
\end{equation}
where
$$\tilde{y}_{\pm a}(z)=e^{\pm c(a)z}\frac{\sigma(z\mp a)}%
{\sigma(z)}=e^{\frac{\pm1}{2}[\zeta
(a+p)+\zeta(a-p)]z}\frac{\sigma(z\mp a)}%
{\sigma(z)}$$
has no branch points and is meromorphic.
Define $(r,s)=(r(A), s(A))\in \mathbb{C}^2$ by
\begin{equation}\label{61-37-1}-2\pi i s:=c(a)-\eta_1a,\qquad 2\pi i r:=c(a)\tau-\eta_2a,
\end{equation}
or equivalently,  by using the Legendre relation $\tau\eta_1-\eta_2=2\pi i$,
\begin{equation}\label{61-37-2} \begin{cases}r+s\tau =a,\\
 r\eta_1+s\eta_2=c(a)=\frac{1}{2}[\zeta
(a+p)+\zeta(a-p)].\end{cases}
\end{equation}
Clearly we can rewrite \eqref{61-37-2} as
\begin{equation}\label{61-37-22}
\zeta
(a+p)+\zeta(a-p)-2(r\eta_1+s\eta_2)=0,\quad\text{with }\;a=r+s\tau.
\end{equation}
Then
by applying the
transformation law \eqref{123123}
to $\tilde{y}_{\pm a}(z)$, we obtain
\begin{align*}\tilde{y}_{\pm a}(z+1)&=e^{\pm c(a)(z+1)}\frac{\sigma(z+1\mp a)}%
{\sigma(z+1)}\\
&=e^{\mp 2\pi is}e^{\pm c(a)z}\frac{\sigma(z\mp a)}%
{\sigma(z)}=e^{\mp 2\pi is}\tilde{y}_{\pm a}(z),\end{align*}
and similarly,
$$\tilde{y}_{\pm a}(z+\tau)=e^{\pm 2\pi ir}\tilde{y}_{\pm a}(z),$$
so
\begin{equation}\label{eqfc-1}{y}_{\pm a}(z+1)=\ell_1^*{y}_{\pm a}(z)=e^{\mp 2\pi is}{y}_{\pm a}(z),\end{equation}
\begin{equation}\label{eqfc-2} {y}_{\pm a}(z+\tau)=\ell_2^*{y}_{\pm a}(z)=e^{\pm 2\pi ir}{y}_{\pm a}(z).\end{equation}
Consequently, there are two cases.

{\bf Case 1.} $Q(A)\neq 0$. Then $y_a(z)$ and $y_{-a}(z)$ are linearly independent, so it follows from \eqref{eqfc-1}-\eqref{eqfc-2} that the monodromy matrices are given by
\begin{equation}
N_1=N_1(A)=%
\begin{pmatrix}
e^{-2\pi is} & 0\\
0 & e^{2\pi is}%
\end{pmatrix}
,\text{ \  \  \ }N_2=N_2(A)=%
\begin{pmatrix}
e^{2\pi ir} & 0\\
0 & e^{-2\pi ir}%
\end{pmatrix}.
\label{Mono-001}%
\end{equation}
Furthermore, the uniqueness of the even elliptic solution $\Phi_e(z)=y_a(z)y_{-a}(z)$ implies $(r,s)\not\in \frac12\mathbb{Z}^2$ (see \cite[Corollary 2.5]{CKL1}). In
particular,
\begin{equation}
(\text{tr}N_1(A),\text{tr}N_2(A))=(2\cos2\pi s,2\cos2\pi
r)\not \in \{ \pm(2,2),\pm(2,-2)\}. \label{complete-rs}%
\end{equation}
Clearly,
\begin{equation} \label{eq-data5}
  \parbox{\dimexpr\linewidth-5em}{$N_1(A), N_2(A)$ are unitary matrices, i.e. the monodromy of \eqref{GLE} is unitary, if and only if the corresponding $(r,s)=(r(A), s(A))\in\mathbb{R}^2\setminus\frac12\mathbb Z^2$.
  }
\end{equation}

{\bf Case 2.} $Q(A)= 0$, i.e. $A=A_k$ and $a=\frac{\omega_k}{2}$ for $k\in\{0,1,2,3\}$. Then we see from \eqref{61-37-2} that $(r,s)\in\frac12\mathbb{Z}^2$ and more precisely,
$$(r,s)\equiv\begin{cases}(0,0)& k=0\\
(\frac12,0)&k=1\\
(0,\frac12)&k=2\\
(\frac12,\frac12)&k=3\end{cases}\quad\operatorname{mod}\;\mathbb{Z}^2.$$
Denote
\begin{equation}\label{var-jjkk}\varepsilon_{j,k}=\begin{cases}1& (j,k)=(1,0),(1,1),(2,0),(2,2),\\
-1&(j,k)=(1,2),(1,3),(2,1),(2,3).\end{cases}\end{equation}
Then  \eqref{eqfc-1}-\eqref{eqfc-2} imply that the eigenvalues of $N_j(A_k)$ are $\{\varepsilon_{j,k},\varepsilon_{j,k}\}$. From here and noting that $y_a(z)$ and $y_{-a}(z)$ are linearly dependent, it follows from the uniqueness of the even elliptic solution $\Phi_e(z)=y_a(z)y_{-a}(z)$ that \eqref{GLE} has no another eigenfunction that is linearly independent with $y_a(z)$. Thus, up to a common conjugation,
\begin{equation}
N_1=N_1(A_k)=\varepsilon_{1,k}%
\begin{pmatrix}
1 & 0\\
1 & 1
\end{pmatrix}
,\text{ \  \  \ }N_2=N_2(A_k)=\varepsilon_{2,k}%
\begin{pmatrix}
1 & 0\\
\mathcal{C}_k & 1
\end{pmatrix}
, \label{Mono-21}%
\end{equation}
where $\mathcal{C}_k%
\in \mathbb{C}\cup \{ \infty \}$. In particular,
\begin{equation}
(\text{tr}N_1(A_k),\text{tr}N_2(A_k))=(2\varepsilon_{1,k}%
,2\varepsilon_{2,k})\in \{ \pm(2,2),\pm(2,-2)\}. \label{notcompleteC}%
\end{equation}
Here if $\mathcal{C}_k=\infty$, then (\ref{Mono-21}) should be understood
as%
\[
N_1=N_1(A_k)=\varepsilon_{1,k}%
\begin{pmatrix}
1 & 0\\
0 & 1
\end{pmatrix}
,\text{ \  \  \ }N_2=N_2(A_k)=\varepsilon_{2,k}%
\begin{pmatrix}
1 & 0\\
1 & 1
\end{pmatrix}
. \]
See \cite{CKL1} for the detaied proof. In particular, at least one of $N_1(A_k), N_2(A_k)$ are not unitary matrices for all $k=0,1,2,3$. 

Clearly there is a one-to-one correspondence between $\{A_k\}_{k=0}^3$ and the trivial critical points $\{\frac{\omega_k}{2}\}_{k=0}^3$ of $G_p(z)$.
We conclude this section by establishing the deep connection with nontrivial critical points of $G_p(z)$.

\begin{theorem}\label{rmk2-5}
There is a one-to-one correspondence between those $A$'s such that $(r,s)=(r(A), s(A))\in\mathbb{R}^2\setminus\frac12\mathbb Z^2$ and pairs of nontrivial critical points $\pm a=\pm (r+s\tau)$ of $G_p(z)$. 
\end{theorem}

\begin{proof} Let $a\in E_{\tau}\setminus E_{\tau}[2]$ and write $a=r+s\tau$ with $(r,s)\in\mathbb{R}^2$, then $(r,s)\notin\frac12\mathbb Z^2$. Recall \eqref{a+001p} that $\pm a$ is a pair of nontrivial critical points of $G_p(z)$ if and only if 
\begin{equation}\label{61-37-22-2}
\zeta
(a+p)+\zeta(a-p)-2(r\eta_1+s\eta_2)=0,\;\text{with}\;a=r+s\tau, \;(r,s)\in\mathbb{R}^2\setminus\frac12\mathbb Z^2,
\end{equation}
which is precisely  \eqref{61-37-22}.

Suppose $(r(A), s(A))\in\mathbb{R}^2\setminus\frac12\mathbb Z^2$ for some $A$, then it follows from \eqref{61-37-22} that $\pm(r(A)+s(A)\tau)$ is a pair of nontrivial critical points of $G_p(z)$.

Conversely, suppose $\pm a=\pm (r+s\tau)$ with 
$(r,s)\in\mathbb{R}^2\setminus\frac12\mathbb Z^2$ is a pair of nontrivial critical points of $G_p(z)$, then \eqref{61-37-22-2} holds. Define $A$ by \eqref{Aap}, i.e.
\[A:=\frac{1}{2}\left[  \zeta(p+a)+\zeta(p-a)-\zeta(2p)\right].\]
Then a direct computation shows that $y_a(z)$ given by \eqref{exp1}-\eqref{61-38}
is a solution of GLE \eqref{GLE} and so is $y_{-a}(z)$. Furthermore, by \eqref{61-37-1} and \eqref{61-37-22-2}, we see that $(r(A), s(A))=(r,s)\in\mathbb{R}^2\setminus\frac12\mathbb Z^2$. This completes the proof.
\end{proof}

\section{Conditional stability sets: A general theory}
\label{section3}

Thanks to Theorem \ref{rmk2-5}, to count the number of pairs of nontrivial critical points of $G_p(z)$, we only need to count the number of $A$'s such that $(r(A), s(A))\in\mathbb{R}^2\setminus\frac12\mathbb Z^2$. This motivates us to study the conditional stability sets of GLE \eqref{GLE}. This idea is inspired by our previous work \cite{CL-AJM}, where we used the conditional stabilitiy sets of the Lam\'{e} equation \begin{equation}\label{Lame}y''(z)=[n(n+1)\wp(z)+B]y(z),\quad n\in\mathbb{N}_{\geq 1},\end{equation} to prove the non-existence of solutions for the curvature equation
$$\Delta u+e^u=8\pi n\delta_0\;\;\text{on }\;E_{\tau},\quad n\in\mathbb{N}_{\geq 1}.$$  
See also \cite{FL} for the study of conditional stabilitiy sets (as parts of the spectral geometry) of other linear ODEs. Remark that the conditional stabilitiy sets of the Lam\'{e} equation \eqref{Lame} coincide with the spectrums of the associated Lam\'{e} operators, but the conditional stability sets of GLE \eqref{GLE} are not because $A$ can not be seen as an eigenvalue.

Recalling \eqref{eqfc-1}-\eqref{eqfc-2}, we have
\begin{equation}\label{eqfc-001}{y}_{\pm a}(z+2\tau-1)=e^{\pm 2\pi i(2r+s)}{y}_{\pm a}(z).\end{equation}
Define 
\begin{equation}\label{eqfc-3}\triangle_j (A):=\frac12\operatorname{tr} N_j(A)=\begin{cases}\frac12(e^{-2\pi is(A)}+e^{2\pi is(A)}),& j=1,\\
\frac12(e^{2\pi ir(A)}+e^{-2\pi ir(A)}),& j=2,\end{cases}\quad A\in\mathbb{C},\end{equation}
\begin{align}\label{eqfc-003}\triangle_3 (A):=&\frac12\operatorname{tr} N_2(A)^2N_1(A)^{-1}\nonumber\\
=&\frac12(e^{2\pi i(2r(A)+s(A))}+e^{-2\pi i(2r(A)+s(A))}),\quad A\in\mathbb{C}.\end{align}
Clearly $\triangle_j (A)$ are holomorphic functions (see e.g. \eqref{eqcs} below). Define
\begin{equation}\label{eqfc-4}\sigma_j:=\triangle_j^{-1}([-1,1]),\quad j=1,2,3.\end{equation}
Then $\sigma_j$ consists of analytic arcs.  Observe that
\begin{itemize}
\item When $A\in \sigma_1$, the corresponding $s(A)\in\mathbb{R}$, so $y_{a}(z)$ is bounded for $z\in q_0+\mathbb{R}$. 
\item When $A\in \sigma_2$, the corresponding $r(A)\in\mathbb{R}$, so $y_{a}(z)$ is bounded for $z\in q_0+\tau\mathbb{R}$. 
\item When $A\in\sigma_3$, the corresponding $2r(A)+s(A)\in\mathbb{R}$, so $y_{a}(z)$ is bounded for $z\in q_0+(2\tau-1)\mathbb{R}$. 
\end{itemize}
Therefore, $\sigma_j$ can be considered as {\it conditional stability sets} of GLE (\ref{GLE}).
The reason why we study these conditional stability sets is clear from the following result.
\begin{theorem}\label{rmk2-8}
The number of pairs of nontrivial critical points of $G_p(z)$ equals to $\#(\sigma_1\cap\sigma_j\setminus\{A_0, A_1, A_2, A_3\})$ for $j=2,3$.
\end{theorem}

\begin{proof} Note that $(r,s)\in\mathbb R^2$ if and only if $(2r+s, s)\in\mathbb R^2$.
Suppose $A\in \sigma_1\cap\sigma_2\setminus\{A_0, A_1, A_2, A_3\}$, then $Q(A)\neq 0$, i.e. Case 1 (see the argument after \eqref{eqfc-2}) happens, and it follows from \eqref{Mono-001}-\eqref{complete-rs} and \eqref{eqfc-3}-\eqref{eqfc-4} that the corresponding $(r(A),s(A))\in \mathbb R^2\setminus\frac12\mathbb{Z}^2$, so $A\in \sigma_1\cap\sigma_3\setminus\{A_0, A_1, A_2, A_3\}$.
Therefore,
$$\sigma_1\cap\sigma_2\setminus\{A_0, A_1, A_2, A_3\}=\sigma_1\cap\sigma_3\setminus\{A_0, A_1, A_2, A_3\},$$
and $(r(A),s(A))\in \mathbb R^2\setminus\frac12\mathbb{Z}^2$ if and only if $A\in \sigma_1\cap\sigma_2\setminus\{A_0, A_1, A_2, A_3\}$. The proof is complete by using Theorem \ref{rmk2-5}.
\end{proof}

In the next section,
we will see that the set $\sigma_1\cap\sigma_2$
has important applications to study the critical points of $G_p(z)$ for the rectangular case $\tau = ib$ with $b>0$. In a subsequent work, we will apply the set $\sigma_1\cap\sigma_3$ to study the critical points of $G_p(z)$ for the rhombus case $\tau=\frac12+ib$ with $b>0$.

The main purpose of this section is to prove that $\sigma_j$ contains at most $5$ analytic arcs; see Corollaries \ref{coro2-11}-\ref{coro2-11-3} below. First, we need to give the definitions of endpoints, cusps and branch points.

\begin{figure}[ht]
\label{O511-0}\includegraphics[scale=0.5, trim=250 370 50 80]{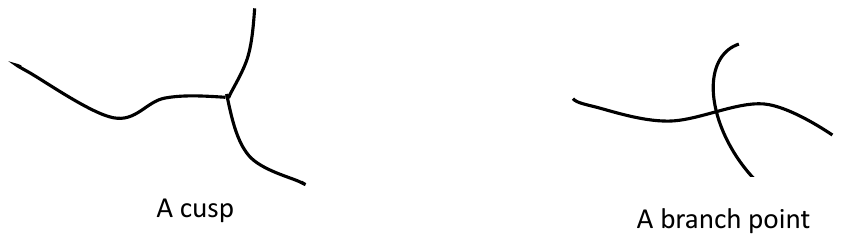}\caption{Cusp and branch point}%
\end{figure}

\begin{remark}\label{rmk2-6}
An endpoint $\tilde{A}$ of an arc of $\sigma_j$ is a point where the arc can not be analytically continued (cf. \cite{GW}). By the Taylor expansion, we see that $\tilde{A}$ is an endpoint if and only if $\triangle_j(\tilde A)^2-1=0$ and
\[d_j(\tilde{A}):=\text{ord}_{\tilde{A}}(\triangle_j(\cdot)^2-1)\]
is \emph{odd}, and in this case there are $d_j(\tilde{A})$ arcs meeting at $\tilde{A}$, each of which can not be analytically continued at $\tilde{A}$ because adjacent arcs meet at $\tilde{A}$ with the same angle $2\pi/d_j(\tilde{A})$. 
Notice that if $\triangle_j(\tilde A)^2-1=0$ and $d_j(\tilde{A})=2k\geq 2$ is even, then $\tilde{A}$ is an inner point of $k$ arcs which are all analytic at $\tilde{A}$, so such $\tilde{A}$ is not considered as an endpoint.
\begin{itemize}
\item
We call an endpoint $\tilde{A}$ of $\sigma_j$ to be {\bf a cusp} if $d_j(\tilde{A})\geq 3$, i.e. there are odd numbers (at least $3$) of arcs meeting at $\tilde{A}$. See the left one in Figure 2 for example.
\item
We call a point $\tilde{A}$ of $\sigma_j$ to be {\bf a branch point} if $\tilde{A}$ is not an endpoint and $\tilde{A}$ is a common inner point of at least two analytic arcs. See the right one in Figure 2 for example.
\end{itemize}
\end{remark}

Note from \eqref{var-jjkk} that $\{A_0, A_1, A_2, A_3\}\subset \sigma_j$ with
\begin{equation}\label{eqfc-vjk1}\triangle_j(A_k)=\varepsilon_{j,k}=\begin{cases}1& (j,k)=(1,0),(1,1),(2,0),(2,2),(3,0),(3,1)\\
-1&(j,k)=(1,2),(1,3),(2,1),(2,3),(3,2),(3,3),\end{cases}\end{equation}
i.e. $\triangle_j(A_k)^2-1=0$. We prove that these $A_k$ give all the endpoints of $\sigma_j$.

\begin{lemma}\label{lemma2-7} Fix $j\in \{1,2,3\}$. Then 
\begin{itemize}
\item[(1)]
 $\{A_0, A_1, A_2, A_3\}$ are precisely the set of endpoints of $\sigma_j$.
 \item[(2)] Any connected component of $\mathbb{C}\setminus\sigma_j$ is unbounded.
 \end{itemize}
\end{lemma}

\begin{proof}
(1). There are two approaches to prove this statment, one is to use the associated hyperelliptic curve as \cite{FL}. Here we use another method from \cite{GW}. Let us take $j=1$ for example (the case $j\in\{2,3\}$ can be proved similarly and is omitted). Recall that $q_0$ is a fixed base point such that $(q_0+\mathbb R)\cap(\pm p+\Lambda_{\tau})=\emptyset$, so $I(x+q_0; p, A, \tau)$ is analytic for $x\in\mathbb R$.
Let $c(x; A)$ and $s(x; A)$ be the solutions of \begin{equation}\label{eqy}y''(x)=I(x+q_0; p, A, \tau)y(x),\quad x\in\mathbb{R}\end{equation} satisfying the initial condition
$$c(0; A)=s'(0; A)=1,\quad c'(0; A)=s(0; A)=0.$$
Since $I(x+q_0; p, A, \tau)$ is of period $1$, so $c(x+1; A)$ and $s(x+1; A)$ are also solutions of \eqref{eqy}, and then
$$\begin{pmatrix}
c(x+1; A)\\
s(x+1; A)
\end{pmatrix}=\begin{pmatrix}
c(1; A)&c'(1;A)\\
s(1; A)& s'(1;A)
\end{pmatrix}\begin{pmatrix}
c(x; A)\\
s(x; A)
\end{pmatrix}.$$
On the other hand, since $y_{\pm a}(x+q_0)$ are also solutions of \eqref{eqy} and
$$y_{\pm a}(x+1+q_0)=e^{\mp 2\pi i s}y_{\pm a}(x+q_0),$$
it follows that $e^{\mp 2\pi is}$ are eigenvalues of $\begin{pmatrix}
c(1; A)&c'(1;A)\\
s(1; A)& s'(1;A)
\end{pmatrix}$, i.e.
\begin{equation}\label{eqcs}\triangle_1(A)=\frac12(e^{2\pi is}+e^{-2\pi is})=\frac12(c(1; A)+s'(1;A)).\end{equation}
In particular, this implies that $\triangle_1(A)$ is holomorphic in $A$.

Take any $A$ such that $s(1; A)\neq 0$. A direct computation shows that $(1, \frac{e^{\mp 2\pi i s}-c(1; A)}{s(1;A)})$ is the 
eigenvector of $\begin{pmatrix}
c(1; A)&c'(1;A)\\
s(1; A)& s'(1;A)
\end{pmatrix}$ with respect to the eigenvalue $e^{\mp 2\pi is}$, i.e.
$$\Big(1, \frac{e^{\mp 2\pi i s}-c(1; A)}{s(1;A)}\Big)\begin{pmatrix}
c(1; A)-e^{\mp 2\pi i s}&c'(1;A)\\
s(1; A)& s'(1;A)-e^{\mp 2\pi i s}
\end{pmatrix}=(0,0).$$
Then we see from $y_1(z; A)=y_a(z)$ and $y_2(z; A)=y_{-a}(z)$ that (write $y_j(z)=y_j(z; A)$ to emphasize the dependence on $A$)
$$y_{1}(x+q_0; A)=y_{1}(q_0; A)\left(c(x;A)+\frac{e^{-2\pi i s}-c(1; A)}{s(1;A)}s(x;A)\right).$$
$$y_{2}(x+q_0; A)=y_{2}(q_0; A)\left(c(x;A)+\frac{e^{2\pi i s}-c(1; A)}{s(1;A)}s(x;A)\right).$$
Therefore, it follows from $cs'-c's\equiv 1$ that
\begin{align*}
W=y_1y_2'-y_1'y_2&=y_{1}(q_0; A)y_{2}(q_0; A)\frac{e^{2\pi i s}-e^{-2\pi i s}}{s(1; A)}(cs'-c's)\\
&=y_{1}(q_0; A)y_{2}(q_0; A)\frac{e^{2\pi i s}-e^{-2\pi i s}}{s(1; A)}.
\end{align*}
From here, $W^2=Q(A)$ and $(e^{2\pi i s}-e^{-2\pi i s})^2=4(\triangle_1(A)^2-1)$, we obtain
\begin{align*}\triangle_1(A)^2-1&=\left(\frac{s(1; A)}{2y_{1}(q_0; A)y_{2}(q_0; A)}\right)^2Q(A)\\
&=\left(\frac{2s(1; A)}{y_{1}(q_0; A)y_{2}(q_0; A)}\right)^2\prod_{k=0}^3(A-A_k).\end{align*}
This formula holds for any $A$ satisfying $s(1; A)\neq 0$ and so holds for all $A$ by continuity.
Thus, $\text{ord}_{A}(\triangle_1(\cdot)^2-1)$ is odd if and only if $A=A_k$ for $k\in\{0,1,2,3\}$. This implies that $\{A_0, A_1, A_2, A_3\}$ are precisely the set of endpoints of $\sigma_1$.

(2). Assume by contradiction that $\mathbb{C}\setminus\sigma_j$ has a bounded connected component $\Omega$. Then $\partial\Omega\subset\sigma_j$ consists of closed curves, and the harmonic function $\operatorname{Im}\triangle_j$ satisfies $\operatorname{Im}\triangle_j\equiv 0$ on $\partial\Omega$, so $\operatorname{Im}\triangle_j\equiv 0$ in $\Omega$, clearly a contradiction.
\end{proof}

\begin{lemma}\label{lemma2-10}
\begin{itemize}
\item[(1)]
The endpoint $A_k$ is a cusp of $\sigma_1$ if and only if $\wp(p-\frac{\omega_k}{2})+\eta_1=0$.
\item[(2)] The endpoint $A_k$ is a cusp of $\sigma_2$ if and only if $\tau\wp(p-\frac{\omega_k}{2})+\eta_2=0$, or equivalently, $\wp(p-\frac{\omega_k}{2})+\eta_1=\frac{2\pi i}{\tau}$.
\item[(3)] The endpoint $A_k$ is a cusp of $\sigma_3$ if and only if $(2\tau-1)\wp(p-\frac{\omega_k}{2})+2\eta_2-\eta_1=0$, or equivalently, $\wp(p-\frac{\omega_k}{2})+\eta_1=\frac{4\pi i}{2\tau-1}$.
\item[(4)] $\sigma_j$ has at most one cusp for $j=1,2,3$ as long as $2p\notin E_{\tau}[2]$.
\end{itemize}
\end{lemma}

\begin{proof} 
Fix any $\tilde{A}$ such that $Q(\tilde{A})\neq 0$, and let $a_0\notin E_{\tau}[2]$ such that \eqref{Aap} holds with $(A, a)=(\tilde{A}, a_0)$. Take $\epsilon>0$ small such that $Q(A)\neq 0$ for any $A\in B_{\epsilon}(\tilde{A}):=\{A : |A-\tilde{A}|<\epsilon\}$. By \eqref{Aap} and \eqref{eq-data1}, and by taking $\epsilon>0$ smaller if necessary, each $A\in B_{\epsilon}(\tilde{A})$ determines a unique $a=a(A)\notin E_{\tau}[2]$ in a small neighborhood of $a_0$ such that $a(\tilde A)=a_0$ and
\begin{equation}\label{eqfc-51}A=\frac{1}{2}\left[  \zeta(p+a(A))+\zeta(p-a(A))-\zeta(2p)\right].\end{equation}
Noting that $\wp(p-a(A))-\wp(p+a(A))\neq 0$ (otherwise $p-a(A)=\pm (p+a(A))$ in $E_{\tau}$ and so $p\in E_{\tau}[2]$ or $a(A)\in E_{\tau}[2]$, a contradiction), we obtain from \eqref{eqfc-51} and the implicit function theorem that
\begin{align}
a'(A)=\frac{2}{\wp(p-a(A))-\wp(p+a(A))}\neq 0,\quad\forall A\in B_{\epsilon}(\tilde A).
\end{align}
Recalling \eqref{61-38} and \eqref{61-37-1}, we have
\begin{align}
2\triangle_1(A)&=e^{-2\pi is}+e^{2\pi is}\nonumber\\
&=e^{\frac{1}{2}[\zeta
(a+p)+\zeta(a-p)]-\eta_1a}+e^{-(\frac{1}{2}[\zeta
(a+p)+\zeta(a-p)]-\eta_1a)},
\end{align}
\begin{align}
2\triangle_2(A)&=e^{2\pi ir}+e^{-2\pi ir}\nonumber\\
&=e^{\frac{\tau}{2}[\zeta
(a+p)+\zeta(a-p)]-\eta_2a}+e^{-(\frac{\tau}{2}[\zeta
(a+p)+\zeta(a-p)]-\eta_2a)},
\end{align}
{\allowdisplaybreaks
\begin{align}
2\triangle_3(A)&=e^{2\pi i(2r+s)}+e^{-2\pi i(2r+s)}\nonumber\\
&=e^{{\tau}[\zeta
(a+p)+\zeta(a-p)]-2\eta_2a-(\frac{1}{2}[\zeta
(a+p)+\zeta(a-p)]-\eta_1a)}\nonumber\\&\quad+e^{-({\tau}[\zeta
(a+p)+\zeta(a-p)]-2\eta_2a)+(\frac{1}{2}[\zeta
(a+p)+\zeta(a-p)]-\eta_1a)}.
\end{align}}%
Consequently,
{\allowdisplaybreaks
\begin{align}\label{triangle1}
2\triangle_1'(A)&=\bigg(e^{\frac{1}{2}[\zeta
(a+p)+\zeta(a-p)]-\eta_1a}-e^{-(\frac{1}{2}[\zeta
(a+p)+\zeta(a-p)]-\eta_1a)})\nonumber\\
&\quad\times\frac{-a'(A)}{2}(\wp(a+p)+\wp(a-p)+2\eta_1)\nonumber\\
&=\frac{-a'(A)}{2}(e^{-2\pi is}-e^{2\pi is})(\wp(a+p)+\wp(a-p)+2\eta_1),
\end{align}}%
\begin{align}\label{triangle2}
2\triangle_2'(A)&=\bigg(e^{\frac{\tau}{2}[\zeta
(a+p)+\zeta(a-p)]-\eta_2a}-e^{-(\frac{\tau}{2}[\zeta
(a+p)+\zeta(a-p)]-\eta_2a)})\nonumber\\
&\quad\times\frac{-a'(A)}{2}(\tau\wp(a+p)+\tau\wp(a-p)+2\eta_2).
\end{align}
{\allowdisplaybreaks
\begin{align}\label{triangle3}
2\triangle_3'(A)&=\bigg(e^{{\tau}[\zeta
(a+p)+\zeta(a-p)]-2\eta_2a-(\frac{1}{2}[\zeta
(a+p)+\zeta(a-p)]-\eta_1a)}\nonumber\\
&\quad-e^{-({\tau}[\zeta
(a+p)+\zeta(a-p)]-2\eta_2a)+(\frac{1}{2}[\zeta
(a+p)+\zeta(a-p)]-\eta_1a)}\bigg)\nonumber\\
&\quad\times\frac{-a'(A)}{2}\Big((2\tau-1)(\wp(a+p)+\wp(a-p))+4\eta_2-2\eta_1\Big).
\end{align}}%

For $A=A_k$, the corresponding $a=\frac{\omega_k}{2}$ and $\wp(p-\frac{\omega_k}{2})-\wp(p+\frac{\omega_k}{2})=0$, 
$$e^{\frac{1}{2}[\zeta
(\frac{\omega_k}{2}+p)+\zeta(\frac{\omega_k}{2}-p)]-\eta_1\frac{\omega_k}{2}}=\varepsilon_{1,k}\in\{\pm 1\}.$$
Then by L'h\^{o}pital's rule, we have
{\allowdisplaybreaks
\begin{align*}
2\triangle_1'(A_k)&=\lim_{A\to A_k} 2\triangle_1'(A)\\
&=\lim_{a\to \frac{\omega_k}{2}}\frac{e^{\frac{1}{2}[\zeta
(a+p)+\zeta(a-p)]-\eta_1a}-e^{-(\frac{1}{2}[\zeta
(a+p)+\zeta(a-p)]-\eta_1a)}}{\wp(p+a)-\wp(p-a)}\\
&\qquad\qquad\times(\wp(a+p)+\wp(a-p)+2\eta_1)\\
&=-\frac{2\varepsilon_{1,k}}{\wp'(p-\frac{\omega_k}{2})}\left(\wp\Big(p-\frac{\omega_k}{2}\Big)+\eta_1\right)^2.
\end{align*}
}%
Simlarly, we have
\begin{align*}
2\triangle_2'(A_k)=-\frac{2\varepsilon_{2,k}}{\wp'(p-\frac{\omega_k}{2})}\left(\tau\wp\Big(p-\frac{\omega_k}{2}\Big)+\eta_2\right)^2,
\end{align*}
\begin{align*}
2\triangle_3'(A_k)=-\frac{2\varepsilon_{3,k}}{\wp'(p-\frac{\omega_k}{2})}\left((2\tau-1)\wp\Big(p-\frac{\omega_k}{2}\Big)+2\eta_2-\eta_1\right)^2.
\end{align*}
Together with Remark \ref{rmk2-6}, we see that $A_k$ is a cusp of $\sigma_j$ if and only if the odd number $\text{ord}_{{A}_k}(\triangle_j(\cdot)^2-1)\geq 3$, if and only if $\triangle_j'(A_k)=0$, if and only if $\wp(p-\frac{\omega_k}{2})+\eta_1=0$ for $j=1$, or $\tau\wp(p-\frac{\omega_k}{2})+\eta_2=0$ for $j=2$, or $(2\tau-1)\wp(p-\frac{\omega_k}{2})+2\eta_2-\eta_1=0$ for $j=3$. This proves the assertions (1)-(3) by using the Legendre relation $\tau\eta_1-\eta_2=2\pi i$.

Finally, if $\sigma_j$ has two cusps $A_k$ and $A_l$, $k\neq l$, then we see from (1)-(3) that
$\wp(p-\frac{\omega_k}{2})=\wp(p-\frac{\omega_l}{2})$, which implies $p-\frac{\omega_k}{2}=-(p-\frac{\omega_l}{2})$ and then $2p\in E_{\tau}[2]$. Thus, $\sigma_j$ has at most one cusp as long as $2p\notin E_{\tau}[2]$.
\end{proof}

Recalling the definition of branch points in Remark \ref{rmk2-6}, similarly we have the following result.

\begin{lemma}\label{lemma2-10-3} Let $A\in\sigma_j\setminus\{A_0, A_1, A_2, A_3\}$. Then the following statements hold.
\begin{itemize}
\item[(1)]
 $A$ is a branch point of $\sigma_1$ if and only if $\wp(a+p)+\wp(a-p)+2\eta_1=0$.
\item[(2)]  $A$ is a branch point of $\sigma_2$ if and only if $\tau(\wp(a+p)+\wp(a-p))+2\eta_2=0$, or equivalently, $\wp(a+p)+\wp(a-p)+2\eta_1=\frac{4\pi i}{\tau}$.
\item[(3)] $A$ is a branch point of $\sigma_3$ if and only if $(2\tau-1)(\wp(a+p)+\wp(a-p))+4\eta_2-2\eta_1=0$, or equivalently, $\wp(a+p)+\wp(a-p)+2\eta_1=\frac{8\pi i}{2\tau-1}$.
\end{itemize}
\end{lemma}

\begin{proof} The proof is similar to that of Lemma \ref{lemma2-10}. Since $A\in\sigma_j\setminus\{A_0, A_1, A_2, A_3\}$, we have $Q(A)\neq 0$ and $\triangle_j(A)\in [-1,1]$. Let us take $j=1$ for example.

{\bf Case 1.} $\triangle_1(A)\in (-1,1)$. Then by the Taylor expansion
$$[-1,1]\ni\triangle_1(\tilde A)=\triangle_1(A)+\triangle_1'(A)(\tilde A-A)+O((\tilde A-A)^2),$$
  we see that $A$ is a branch point of $\sigma_1$ if and only if $\triangle_1'(A)=0$, which is equivalent to $\wp(a+p)+\wp(a-p)+2\eta_1=0$ by using \eqref{triangle1}. 
  
  {\bf Case 2.} $\triangle_1(A)=\pm 1$. Then \eqref{triangle1} implies $\triangle_1'(A)=0$ automatically. Again by  
  the Taylor expansion
$$[-1,1]\ni\triangle_1(\tilde A)=\pm 1+\frac12\triangle_1''(A)(\tilde A-A)^2+O((\tilde A-A)^3),$$
we see that $A$ is a branch point of $\sigma_1$ if and only if $\triangle_1''(A)=0$. By using \eqref{triangle1} and $\triangle_1(A)=\pm 1$, i.e. $e^{2\pi is}=e^{-2\pi is}=\pm 1$, we see that
$$2\triangle_1''(A)=\frac{a'(A)^2}{4}(e^{-2\pi is}+e^{2\pi is})(\wp(a+p)+\wp(a-p)+2\eta_1)^2,$$
so $\triangle_1''(A)=0$ is equivalent to $\wp(a+p)+\wp(a-p)+2\eta_1=0$. 
The proof is complete.
\end{proof}

\begin{lemma}\label{lemma2-8}
Denote $B_R:=\{A\in\mathbb{C} : |A|<R\}$. Then for $R>0$ large, the following statements hold.
\begin{itemize}
\item[(1)] $\sigma_1\setminus{B_R}$ consists of only two disjoint analytic arcs that can be parametrized by 
\begin{equation}\label{sigma1}
A=it-\frac{2p\eta_1-\zeta(2p)}{2}+O(t^{-1})
,\quad t\in(-\infty, -t_1]\cup[t_2,+\infty),\end{equation}
for some $t_j>0$ large.

\item[(2)] $\sigma_2\setminus{B_R}$ consists of only two disjoint analytic arcs that can be parametrized by 
\begin{equation}\label{sigma2}
A=\frac{it}{\tau}-\frac{2p\eta_2-\tau\zeta(2p)}{2\tau}+O(t^{-1})
,\quad t\in(-\infty, -t_3]\cup[t_4,+\infty),\end{equation}
for some $t_j>0$ large.

\item[(3)] $\sigma_3\setminus{B_R}$ consists of only two disjoint analytic arcs that can be parametrized by 
\begin{equation}\label{sigma3}
A=\frac{it}{2\tau-1}-\frac{2p(2\eta_2-\eta_1)-(2\tau-1)\zeta(2p)}{2(2\tau-1)}+O(t^{-1})
,\end{equation}
$t\in(-\infty, -t_5]\cup[t_6,+\infty)$ for some $t_j>0$ large.
\end{itemize}

\end{lemma}

\begin{proof}
By \eqref{phie}-\eqref{eq219}, \eqref{exp1} and \eqref{eqfc-1}-\eqref{eqfc-2}, we have
\begin{align*}
e^{-2\pi i s}&=-\left(\frac{\Phi_e(z+1;A)}{\Phi_e(z;A)}\right)^{\frac12}\exp\int_{q_0}^{q_0+1}\frac{Q(A)^{\frac12}}{2\Phi_e(\xi; A)}d\xi\\
&=\varepsilon_1\exp\int_{q_0}^{q_0+1}\frac{Q(A)^{\frac12}}{2\Phi_e(\xi; A)}d\xi,
\end{align*}
\begin{align*}
e^{2\pi i r}&=-\left(\frac{\Phi_e(z+\tau;A)}{\Phi_e(z;A)}\right)^{\frac12}\exp\int_{q_0}^{q_0+\tau}\frac{Q(A)^{\frac12}}{2\Phi_e(\xi; A)}d\xi\\
&=\varepsilon_2\exp\int_{q_0}^{q_0+\tau}\frac{Q(A)^{\frac12}}{2\Phi_e(\xi; A)}d\xi,
\end{align*}
where $\varepsilon_j\in\{\pm 1\}$ because $\Phi_e(z)=\Phi_e(z;A)$ is an elliptic function. For $|A|\geq R$ with $R$ large, we have that for any $z\in [q_0, q_0+1]\cup[q_0, q_0+\tau]\not\ni\pm p$,
\begin{align*}
\frac{Q(A)^{\frac12}}{2\Phi_e(z)}&=\frac{\left(16A^4+O(A^2)\right)^\frac12}{-4A(1-\frac{\zeta(z+p)-\zeta(z-p)-\zeta(2p)}{2}A^{-1})}\\
&=\varepsilon_3A\left(1+\frac{\zeta(z+p)-\zeta(z-p)-\zeta(2p)}{2}A^{-1}+O(A^{-2})\right),
\end{align*}
where $\varepsilon_3\in\{\pm 1\}$. By using the transformation law \eqref{123123}, we have
\begin{align*}\int_{q_0}^{q_0+1}(\zeta(z+p)-\zeta(z-p))dz&=\log\frac{\sigma(q_0+1+p)\sigma(q_0-p)}{\sigma(q_0+p)\sigma(q_0+1-p)}\\
&=\log e^{2p\eta_1}=2p\eta_1+2m_1\pi i,\end{align*}
\begin{align*}\int_{q_0}^{q_0+\tau}(\zeta(z+p)-\zeta(z-p))dz&=\log\frac{\sigma(q_0+\tau+p)\sigma(q_0-p)}{\sigma(q_0+p)\sigma(q_0+\tau-p)}\\
&=\log e^{2p\eta_2}=2p\eta_2+2m_2\pi i,\end{align*}
where $m_j\in\mathbb Z$, so
\begin{align}\label{sxds}e^{-2\pi i s}&=\varepsilon_1\exp\left({\varepsilon_3A\left(1+\frac{2p\eta_1+2m_1\pi i-\zeta(2p)}{2}A^{-1}+O(A^{-2})\right)}\right)\nonumber\\
&=\varepsilon_4\exp\left({\varepsilon_3A\left(1+\frac{2p\eta_1-\zeta(2p)}{2}A^{-1}+O(A^{-2})\right)}\right),
\end{align}
and similarly,
\begin{align*}e^{2\pi i r}
=\varepsilon_5\exp\left({\varepsilon_3A\left(\tau+\frac{2p\eta_2-\tau\zeta(2p)}{2}A^{-1}+O(A^{-2})\right)}\right),
\end{align*}
where $\varepsilon_j\in\{\pm 1\}$. 

Hence, $\triangle_1(A)=\frac12(e^{-2\pi i s}+e^{2\pi i s})\in [-1,1]$ if and only if $s\in\mathbb{R}$, if and only if
$$A\left(1+\frac{2p\eta_1-\zeta(2p)}{2}A^{-1}+O(A^{-2})\right)=it\quad\text{with}\;t\in\mathbb{R}, |t|\text{ large},$$
which is equivalent to \eqref{sigma1}. Clearly the above argument of proving \eqref{sxds} actually implies
$$e^{-2\pi i s}=\varepsilon_4\exp\left({\varepsilon_3A\left(1+\frac{2p\eta_1-\zeta(2p)}{2}A^{-1}+\varphi(A^{-1})\right)}\right)$$
for $|A|$ large, where $\varphi(z)$ is holomorphic for $|z|$ small and $\varphi(z)=O(z^2)$ for $|z|$ small. Thus, given any $t\in\mathbb R$ with $|t|$ large enough, there is a unique $A\in\mathbb C$ with $|A|$ large such that
$$A\left(1+\frac{2p\eta_1-\zeta(2p)}{2}A^{-1}+\varphi(A^{-1})\right)=it,$$
which implies $\triangle_1(A)\in [-1, 1]$, i.e. $A\in\sigma_1$. This proves the assertion (1).

Similarly, $\triangle_2(A)=\frac12(e^{2\pi i r}+e^{-2\pi i r})\in [-1,1]$ if and only if $r\in\mathbb{R}$, if and only if
$$A\left(\tau+\frac{2p\eta_2-\tau\zeta(2p)}{2}A^{-1}+O(A^{-2})\right)=it\quad\text{with}\;t\in\mathbb{R}, |t|\text{ large},$$
which is equivalent to \eqref{sigma2}. 
Since
\begin{align*}&e^{2\pi i (2r+s)}\\
=&\varepsilon_4\exp\left({\varepsilon_3A\left((2\tau-1)+\frac{2p(2\eta_2-\eta_1)-(2\tau-1)\zeta(2p)}{2}A^{-1}+O(A^{-2})\right)}\right),
\end{align*}
we see that $\triangle_3(A)=\frac12(e^{2\pi i (2r+s)}+e^{-2\pi i (2r+s)})\in [-1,1]$ if and only if $2r+s\in\mathbb{R}$, if and only if
$$A\left((2\tau-1)+\frac{2p(2\eta_2-\eta_1)-(2\tau-1)\zeta(2p)}{2}A^{-1}+O(A^{-2})\right)=it,$$
with $t\in\mathbb{R}, |t|$ large, which is equivalent to \eqref{sigma3}. 
Then a simiar argument implies the assertions (2)-(3).
This completes the proof.
\end{proof}

\begin{corollary}\label{coro2-11} Let $p\in E_{\tau}\setminus E_{\tau}[2]$ and $j\in\{1,2\}$. Then $\sigma_j$ consists of $m_j\in [3,5]$ analytic arcs with finite endpoints $A_0, A_1, A_2, A_3$, among which there is one or two unbounded arcs. 
More precisely,
\begin{itemize}
\item[(1)] If $\omega_j\wp(p-\frac{\omega_k}{2})+\eta_j\neq 0$ for all $k$, then $\sigma_j$ has no cusps and $\sigma_j$ consists of $3$ analytic arcs.

\item[(2)] If $\omega_j\wp(p-\frac{\omega_k}{2})+\eta_j=0$ and $12\eta_j^2-\omega_j^2g_2\neq 0$ for some $k$, then $A_k$ is a cusp of $\sigma_j$  with $\text{ord}_{{A}_k}(\triangle_j(\cdot)^2-1)=3$. Furthermore,
\begin{itemize}
\item[(2-1)] If $\omega_j\wp(p-\frac{\omega_l}{2})+\eta_j\neq 0$ for any $l\neq k$, then $A_k$ is the unique cusp of $\sigma_j$,
and $\sigma_j$ consists of $4$ analytic arcs, among which there are three of them having the common endpoint $A_k$.

\item[(2-2)] If $\omega_j\wp(p-\frac{\omega_l}{2})+\eta_j=0$ for some $l\neq k$, then $A_l$ is also a cusp of $\sigma_j$ with $\text{ord}_{{A}_l}(\triangle_j(\cdot)^2-1)=3$,
and $\sigma_j$ is path-connected and consists of $5$ analytic arcs, among which one has endpoints $A_k$ and $A_l$, two others have the common endpoints $A_k$, and the remaining two have the common endpoints $A_l$.
\end{itemize}

\item[(3)] If $\omega_j\wp(p-\frac{\omega_k}{2})+\eta_j=0$ and $12\eta_j^2-\omega_j^2g_2=0$ for some $k$, then $A_k$ is the unique cusp of $\sigma_j$  with $\text{ord}_{{A}_k}(\triangle_j(\cdot)^2-1)=5$, and $\sigma_j$ is path-connected and consists of $5$ analytic arcs that have the common endpoint $A_k$.
\end{itemize}

\end{corollary}

\begin{proof} (1). Suppose $\omega_j\wp(p-\frac{\omega_k}{2})+\eta_j\neq0$ for all $k$. Then Lemma \ref{lemma2-10} says that $\sigma_j$ has no cusps, namely each $A_k$ is an endpoint of exactly one analytic arc in $\sigma_j$. Together with Lemmas \ref{lemma2-7} and \ref{lemma2-8}, we see that $\sigma_j$ consists of $3$ analytic arcs with finite endpoints $A_0, A_1, A_2$, $A_3$, among which there is one or two unbounded arcs. Note that different analytic arcs might intersect at branch points.

(2)-(3). 
Suppose $\omega_j\wp(p-\frac{\omega_k}{2})+\eta_j=0$. Then Lemma \ref{lemma2-10} says that $A_k$ is a cusp of $\sigma_j$. Together with Lemmas \ref{lemma2-7}, \ref{lemma2-8} and the fact that $\text{ord}_{{A}_k}(\triangle_j(\cdot)^2-1)>1$ is odd, we see that $\text{ord}_{{A}_k}(\triangle_j(\cdot)^2-1)\in\{3,5\}$. 

More precisely, when $\text{ord}_{{A}_k}(\triangle_j(\cdot)^2-1)=5$, then
$\sigma_j$ consists of $5$ analytic arcs that have the common endpoint $A_k$, and two of them are unbounded, while the remaining three arcs are bounded with the other endpoints $A_l$, $l\in\{0,1,2,3\}\setminus\{k\}$, respectively. In particular, all other $A_l$, $l\in\{0,1,2,3\}\setminus\{k\}$, are not cusps of $\sigma_j$. 

While when $\text{ord}_{{A}_k}(\triangle_j(\cdot)^2-1)=3$, there are two possibilities. The first case is $\omega_j\wp(p-\frac{\omega_l}{2})+\eta_j\neq 0$ for any $l\neq k$, then $A_k$ is the unique cusp of $\sigma_j$,  so $\sigma_j$ consists of $4$ analytic arcs, among which there are three of them having the common endpoint $A_k$, and the $4$ analytic arcs contain one or two unbounded arcs. The second case is $\omega_j\wp(p-\frac{\omega_l}{2})+\eta_j=0$ for some $l\neq k$, then $A_l$ is also a cusp of $\sigma_j$ with $\text{ord}_{{A}_l}(\triangle_j(\cdot)^2-1)\geq3$. Since any connected components of $\mathbb{C}\setminus \sigma_j$ are unbounded,  we conclude that $\sigma_j$ consists of $5$ analytic arcs, among which one has endpoints $A_k$ and $A_l$, two others have the common endpoints $A_k$, and the remaining two have the common endpoints $A_l$. In particular, $\text{ord}_{{A}_l}(\triangle_j(\cdot)^2-1)=3$.


Now we determine when $\text{ord}_{{A}_k}(\triangle_j(\cdot)^2-1)=3$. Let us take $j=1$ for example.
By $\wp(p-\frac{\omega_k}{2})+\eta_1=0$, direct computations give that as $a\to\frac{\omega_k}{2}$,
$$\wp(a+p)+\wp(a-p)+2\eta_1=\wp''\Big(p-\frac{\omega_k}{2}\big)\Big(a-\frac{\omega_k}{2}\Big)^2+O\Big(\Big(a-\frac{\omega_k}{2}\Big)^4\Big),$$
$$\wp(p+a)-\wp(p-a)=2\wp'\Big(p-\frac{\omega_k}{2}\Big)\Big(a-\frac{\omega_k}{2}\Big)+O\Big(\Big(a-\frac{\omega_k}{2}\Big)^3\Big),$$
\begin{align*}
e^{\frac{1}{2}[\zeta
(a+p)+\zeta(a-p)]-\eta_1a}-e^{-(\frac{1}{2}[\zeta
(a+p)+\zeta(a-p)]-\eta_1a)}\\
=-\frac{\varepsilon_{1,k}}{3}\wp''\Big(p-\frac{\omega_k}{2}\Big)\Big(a-\frac{\omega_k}{2}\Big)^3+O\Big(\Big(a-\frac{\omega_k}{2}\Big)^5\Big).
\end{align*}
Furthermore, \eqref{eqfc-51} leads to
$$A-A_k=-\frac12\wp'\Big(p-\frac{\omega_k}{2}\Big)\Big(a-\frac{\omega_k}{2}\Big)^2+O\Big(\Big(a-\frac{\omega_k}{2}\Big)^4\Big).$$
Therefore, we deduce from \eqref{triangle1} that
{\allowdisplaybreaks
\begin{align*}
2\triangle_1'(A)
&=\frac{e^{\frac{1}{2}[\zeta
(a+p)+\zeta(a-p)]-\eta_1a}-e^{-(\frac{1}{2}[\zeta
(a+p)+\zeta(a-p)]-\eta_1a)}}{\wp(p+a)-\wp(p-a)}\\
&\qquad\qquad\times(\wp(a+p)+\wp(a-p)+2\eta_1)\\
&=-\frac{\varepsilon_{1,k}\wp''(p-\frac{\omega_k}{2})^2}{6\wp'(p-\frac{\omega_k}{2})}\Big(a-\frac{\omega_k}{2}\Big)^4+O((a-\frac{\omega_k}{2})^6)\\
&=-\frac{2\varepsilon_{1,k}\wp''(p-\frac{\omega_k}{2})^2}{3\wp'(p-\frac{\omega_k}{2})^3}(A-A_k)^2+O((A-A_k)^3).
\end{align*}
}%
Note from $\wp(p-\frac{\omega_k}{2})+\eta_1=0$ that $\wp''(p-\frac{\omega_k}{2})=6\wp(p-\frac{\omega_k}{2})^2-\frac{g_2}{2}=6\eta_1^2-\frac{g_2}{2}$. 

If $12\eta_1^2-g_2\neq 0$, then $\wp''(p-\frac{\omega_k}{2})\neq 0$, so $\text{ord}_{{A}_k}(\triangle_1(\cdot)^2-1)=3$. This proves the assertion (2) for $j=1$.
If $12\eta_1^2-g_2=0$, then $\wp''(p-\frac{\omega_k}{2})=0$, so $\text{ord}_{{A}_k}(\triangle_1(\cdot)^2-1)>3$ and consequently, $\text{ord}_{{A}_k}(\triangle_1(\cdot)^2-1)=5$. This proves the assertion (3) for $j=1$.

For $j=2$,
by $\tau\wp(p-\frac{\omega_k}{2})+\eta_2=0$, a similar argument leads to
$$2\triangle_2'(A)=-\frac{2\varepsilon_{2,k}\tau^2\wp''(p-\frac{\omega_k}{2})^2}{3\wp'(p-\frac{\omega_k}{2})^3}(A-A_k)^2+O((A-A_k)^3).$$
Note that $\wp''(p-\frac{\omega_k}{2})=6\wp(p-\frac{\omega_k}{2})^2-\frac{g_2}{2}=6\frac{\eta_2^2}{\tau^2}-\frac{g_2}{2}$, so $\wp''(p-\frac{\omega_k}{2})=0$ if and only if $12\eta_2^2-\tau^2g_2=0$. Then $12\eta_2^2-\tau^2g_2\neq0$ implies $\text{ord}_{{A}_k}(\triangle_2(\cdot)^2-1)=3$, and $12\eta_2^2-\tau^2g_2=0$ implies $\text{ord}_{{A}_k}(\triangle_2(\cdot)^2-1)=5$. The proof is complete.
\end{proof}

\begin{corollary}\label{coro2-11-3} Let $p\in E_{\tau}\setminus E_{\tau}[2]$. Then $\sigma_3$ consists of $m_3\in [3,5]$ analytic arcs with finite endpoints $A_0, A_1, A_2, A_3$, among which there is one or two unbounded arcs. 
More precisely,
\begin{itemize}
\item[(1)] If $(2\tau-1)\wp(p-\frac{\omega_k}{2})+2\eta_2-\eta_1\neq 0$ for all $k$, then $\sigma_3$ has no cusps and $\sigma_3$ consists of $3$ analytic arcs.

\item[(2)] If $(2\tau-1)\wp(p-\frac{\omega_k}{2})+2\eta_2-\eta_1=0$ and $12(2\eta_2-\eta_1)^2-(2\tau-1)^2g_2\neq 0$ for some $k$, then $A_k$ is a cusp of $\sigma_3$  with $\text{ord}_{{A}_k}(\triangle_3(\cdot)^2-1)=3$. 
Furthermore,
\begin{itemize}
\item[(2-1)] If $(2\tau-1)\wp(p-\frac{\omega_l}{2})+2\eta_2-\eta_1\neq0$ for any $l\neq k$, then $A_k$ is the unique cusp of $\sigma_3$,
and $\sigma_3$ consists of $4$ analytic arcs, among which there are three of them having the common endpoint $A_k$.

\item[(2-2)] If $(2\tau-1)\wp(p-\frac{\omega_l}{2})+2\eta_2-\eta_1=0$ for some $l\neq k$, then $A_l$ is also a cusp of $\sigma_3$ with $\text{ord}_{{A}_l}(\triangle_3(\cdot)^2-1)=3$,
and $\sigma_3$ is path-connected and consists of $5$ analytic arcs, among which one has endpoints $A_k$ and $A_l$, two others have the common endpoints $A_k$, and the remaining two have the common endpoints $A_l$.
\end{itemize}

\item[(3)] If $(2\tau-1)\wp(p-\frac{\omega_k}{2})+2\eta_2-\eta_1=0$ and $12(2\eta_2-\eta_1)^2-(2\tau-1)^2g_2=0$ for some $k$, then $A_k$ is the unique cusp of $\sigma_3$  with $\text{ord}_{{A}_k}(\triangle_3(\cdot)^2-1)=5$, and $\sigma_3$ consists of $5$ analytic arcs that have the common endpoint $A_k$.
\end{itemize}

\end{corollary}

\begin{proof} By $(2\tau-1)\wp(p-\frac{\omega_k}{2})+2\eta_2-\eta_1=0$, a similar argument as Corollary \ref{coro2-11} leads to
$$2\triangle_3'(A)=-\frac{2\varepsilon_{3,k}(2\tau-1)^2\wp''(p-\frac{\omega_k}{2})^2}{3\wp'(p-\frac{\omega_k}{2})^3}(A-A_k)^2+O((A-A_k)^3),$$
so $\wp''(p-\frac{\omega_k}{2})=0$ if and only if $12(2\eta_2-\eta_1)^2-(2\tau-1)^2g_2=0$.
The rest proof is similar to that of Corollary \ref{coro2-11} and is omitted.
\end{proof}

Remark that it was proved in \cite[Theorem 1.2]{CL-JDG19} that there are countably many $\tau$ such that $12\eta_1(\tau)^2-g_2(\tau)=0$. Furthermore, since $$12\eta_2(\tau)^2-\tau^2g_2(\tau)=\frac1{\tau^2}\Big(12\eta_1\Big(\frac{-1}{\tau}\Big)^2-g_2\Big(\frac{-1}{\tau}\Big)\Big)$$ by using the modular property of $\eta_1, \eta_2, g_2$ (see e.g. \cite[page 205]{CL-JDG19}), there are also countably many $\tau$ such that $12\eta_2(\tau)^2-\tau^2g_2(\tau)=0$. Similarly, \begin{align*}&12(2\eta_2(\tau)-\eta_1(\tau))^2-(2\tau-1)^2g_2(\tau)\\=&\frac{1}{(2\tau-1)^2}\Big(12\eta_1\Big(\frac{\tau-1}{2\tau-1}\Big)^2-g_2\Big(\frac{\tau-1}{2\tau-1}\Big)\Big),\end{align*}so there are also countably many $\tau$ such that $12(2\eta_2(\tau)-\eta_1(\tau))^2-(2\tau-1)^2g_2(\tau)=0$. Thus Corollary \ref{coro2-11}-(3) and Corollary \ref{coro2-11-3}-(3) really occurs for some $(\tau, p)$.

\section{Rectangular torus: Conditional Stability sets}
\label{section4}

In this section, we consider $\tau=ib$ with $b>0$, i.e. $E_{\tau}$ is a rectangular torus, and give a more accurate characterization of the conditional stability sets $\sigma_1$ and $\sigma_2$.

Note from $\tau=ib$ that $\bar\tau=-\tau$. Then it follows from the definition of $\wp(z)=\wp(z;\tau)$ and $\zeta(z)=\zeta(z;\tau)$ that
\begin{equation}\label{barwp}\overline{\wp(\bar{z})}=\wp(z),\quad\overline{\zeta(\bar{z})}=\zeta(z).\end{equation}
In particular, $\wp(z)\in\mathbb{R}$ if and only if $\wp(z)=\wp(\bar z)$, if and only if $z=\pm \bar z$ in $E_{\tau}$, which is equivalent to 
\begin{equation}\label{sdsd}\pm z\in \Big(0,\frac{1}{2} \Big]\cup  \Big[\frac{1}{2},\frac{1+\tau}{2} \Big]\cup  \Big[\frac{\tau}{2},\frac{1+\tau}{2} \Big]\cup  \Big(0,\frac{\tau}{2} \Big],\quad\operatorname{mod}\;\mathbb{Z}+\mathbb{Z}\tau.\end{equation}
Here
$[z_{1},z_{2}]:=\{(1-t)z_{1}+tz_{2} :\;0\leq t\leq1\}$, and $[z_1,z_2)$, $(z_1, z_2]$, $(z_1,z_2)$ are defined similarly. Then it is easy to obtain the following result.

\begin{lemma} (see e.g. \cite[Lemma 3.2]{CKL-CAG})
\label{Lemma3-1} Let $\omega_{2}=\tau=ib$ with $b>0$. Then $\wp$ is one to
one from $(0,\frac{\tau}{2}]\cup \lbrack
\frac{\tau}{2},\frac{1+\tau}{2}]\cup \lbrack \frac{1+\tau}{2},\frac{1}{2}%
]\cup \lbrack \frac{1}{2},0)$ onto $(-\infty,+\infty)$, with
\[
\lim_{\,[\frac{1}{2},0)\ni z\rightarrow0}\wp(z)=+\infty,\quad \lim
_{(0,\frac{\tau}{2}]\ni z\rightarrow0}\wp(z)=-\infty.
\] 
\end{lemma}

\begin{remark}\label{rmk3-2} Write $z=x_1+ix_2$ for $x_1, x_2\in\mathbb R$. 
Recall $e_k=\wp(\frac{\omega_k}{2})$ for $k=1,2,3$ and
\begin{align*}
 &\wp'(z)^2 = 4\wp(z)^3 - g_2\wp(z) - g_3=4\prod_{k=1}^3(\wp(z)-e_k). \end{align*}
From Lemma \ref{Lemma3-1} and $e_1+e_2+e_3=0$, we have $e_{j}, g_2, g_3
\in \mathbb{R}$, $e_{2}<e_{3}<e_{1}$, and $e_{2}<0<e_{1}$. Moreover, $\wp^{\prime
}(z)=\frac{\partial \wp(z)}{\partial x_{1}}\in \mathbb{R}$ if $z\in(0,\frac
{1}{2}]\cup[\frac{\tau}{2},\frac{1+\tau}{2}]$, and
$\wp^{\prime}(z)=-i\frac{\partial \wp(z)}{\partial x_{2}}\in i\mathbb{R}$ if
$z\in(0,\frac{\tau}{2}]\cup[\frac{1}{2},\frac{1+\tau}{2}]$. Lemma~\ref{Lemma3-1} and \eqref{barwp} also imply that
\begin{itemize}
\item For $z\in(0,1)$, 
$\zeta(z)\in \mathbb{R}$
 and so $\eta_{1}\in \mathbb{R}$, $\eta_2=\tau\eta_1-2\pi i\in i\mathbb{R}$.
\item For $z\in(0,{\tau})$, 
$\zeta(z)\in i\mathbb{R}$.
\end{itemize}
Consequently,
\begin{itemize}
\item For $z\in(0,\frac{1}{2})$, 
$\zeta(2z)\in \mathbb{R}$.
\item For $z\in(0,\frac{\tau}{2})$, 
$\zeta(2z)\in i\mathbb{R}$.
\item For $z\in(\frac{1}{2},\frac{1+\tau}{2})$, $2z=1+t\tau$ for some $t\in (0,1)$, so $\zeta(2z)=\eta_1+\zeta(t\tau)\in\eta_1+i\mathbb R$. 
\item For $z\in(\frac{\tau}{2}, \frac{1+\tau}{2})$, $2z=\tau+t$ for some $t\in (0,1)$, so  $\zeta(2z)=\eta_2+\zeta(t)\in \eta_2+\mathbb R$.
\end{itemize}
\end{remark}

\begin{lemma}\label{lemma3-3}
Let $\tau=ib$ with $b>0$, $p\notin E_{\tau}[2]$ and $\wp(p)\in\mathbb{R}$. Then the following statements hold.
\begin{itemize}
\item[(1)] For $p\in (0, \frac12)$, we have $\wp'(p)<0$ and so $A_1<A_3<A_2<A_0$.
\item[(2)] For $p\in (\frac12, \frac{1+\tau}{2})$, we have $i\wp'(p)=\frac{\partial \wp(p)}{\partial x_2}<0$ and so $A_k=ia_k$ with $a_1<a_0<a_2<a_3$.
\item[(3)] For $p\in (\frac{\tau}{2}, \frac{1+\tau}2)$, we have $\wp'(p)>0$ and so $A_3<A_1<A_0<A_2$.
\item[(4)] For $p\in (0, \frac{\tau}{2})$, we have $i\wp'(p)=\frac{\partial \wp(p)}{\partial x_2}>0$ and so $A_k=ia_k$ with $a_0<a_1<a_3<a_2$.
\end{itemize}
\end{lemma}

\begin{proof}
By Lemma \ref{Lemma3-1} and by replacing $p$ with $-p$ if necessary, the assumption $\wp(p)\in\mathbb{R}$ and $p\notin E_{\tau}[2]$ imply that $$p\in\Big(0,\frac{1}{2}\Big)\cup \Big(\frac{1}{2},\frac{1+\tau}{2}\Big)\cup \Big(\frac{\tau}{2},\frac{1+\tau}{2}\Big)\cup\Big(0,\frac{\tau}{2}\Big).$$
Furthermore, $\wp''(p)=6\wp(p)^2-g_2/2\in\mathbb{R}$.

(1). Since $p\in (0, \frac12)$, we have $e_2<e_3<e_1<\wp(p)$, which together with $\wp'(p)<0$ implies
$$\frac{\wp'(p)}{2(\wp(p)-e_1)}<\frac{\wp'(p)}{2(\wp(p)-e_3)}<\frac{\wp'(p)}{2(\wp(p)-e_2)}<0,$$
so we see from \eqref{Ak} that $A_1<A_3<A_2<A_0$.

(2). Since $p\in (\frac12, \frac{1+\tau}{2})$ implies $e_2<e_3<\wp(p)<e_1$, it follows from $\wp'(p)=-i\frac{\partial \wp(p)}{\partial x_2}$ with $-\frac{\partial \wp(p)}{\partial x_2}>0$ that
$$\frac{-\frac{\partial \wp(p)}{\partial x_2}}{2(\wp(p)-e_1)}<0<\frac{-\frac{\partial \wp(p)}{\partial x_2}}{2(\wp(p)-e_2)}<\frac{-\frac{\partial \wp(p)}{\partial x_2}}{2(\wp(p)-e_3)},$$
so $A_k=ia_k$ with $a_1<a_0<a_2<a_3$.

(3) Since $p\in (\frac{\tau}{2}, \frac{1+\tau}2)$ implies $e_2<\wp(p)<e_3<e_1$, a similar argument gives $A_3<A_1<A_0<A_2$.

(4) Since   $p\in (0, \frac{\tau}{2})$ implies $\wp(p)<e_2<e_3<e_1$, a similar argument gives $A_k=ia_k$ with $a_0<a_1<a_3<a_2$.
\end{proof}

\begin{lemma}\label{lemma3-4}
Let $\tau=ib$ with $b>0$.
\begin{itemize}
\item[(1)] If $p\in (0, \frac12)\cup (\frac{\tau}{2}, \frac{1+\tau}2)$, then $\sigma_j$ is symmetric with respect to the real axis for each $j=1,2$.
\item[(2)] If $p\in  (0, \frac{\tau}{2})\cup(\frac12, \frac{1+\tau}{2})$, then $\sigma_j$ is symmetric with respect to the imaginary axis for each $j=1,2$.
\end{itemize}

\end{lemma}

\begin{proof} Note from $\tau=ib$ that $\bar{\tau}=-\tau$.
Set $\hat{y}_a(z):=\overline{y_{a}(\bar z)}$. Then \eqref{eqfc-1}-\eqref{eqfc-2} imply that
\begin{equation}\label{eqfc-5}\hat{y}_{a}(z+1)=\ell_1^*\hat{y}_a(z)=\overline{e^{-2\pi is}}\hat{y}_{a}(z),\end{equation}
\begin{equation}\label{eqfc-6} \hat{y}_{a}(z+\tau)=\ell_2^*\hat{y}_{a}(z)=\overline{e^{-2\pi ir}}\hat{y}_{ a}(z).\end{equation}
Furthermore, $\hat{y}_a''(z)=\overline{I(\bar z; p, A, \tau)}\hat{y}_a(z)$, where by using \eqref{barwp}, 
\begin{align}\label{eqfc-7}
\overline{I(\bar z; p, A, \tau)}=
&\frac{3}{4}
(\wp(z+\bar{p})+\wp(z-\bar{p})-\wp (2\bar p))\nonumber\\
&+\overline{A}(\zeta(z+\bar{p})-\zeta(z-\bar{p})-\zeta(2\bar p))+\overline{A}^{2}\nonumber\\
=&I(z; \bar p, \overline A, \tau).
\end{align}

(1) For $p\in (0, \frac12)\cup (\frac{\tau}{2}, \frac{1+\tau}2)$, we have either $\bar p=p$ or $\bar p=p-\tau\equiv p$ mod $\Lambda_\tau$, so $I(z; \bar p, \overline A, \tau)=I(z; p, \overline A, \tau)$. Then it follows from \eqref{eqfc-5}-\eqref{eqfc-6} that \begin{equation}\label{eqfc-8}\triangle_1(\overline{A})=\overline{\cos 2\pi s}=\overline{\triangle_1(A)},\quad \triangle_2(\overline{A})=\overline{\cos 2\pi r}=\overline{\triangle_2(A)}.\end{equation}
Thus, $\overline{A}\in \sigma_j$ if and only if $A\in\sigma_j$, namely $\sigma_j$ is symmetric with respect to the real axis.

(2) For $p\in  (0, \frac{\tau}{2})\cup (\frac12, \frac{1+\tau}{2})$, we have either $\bar{p}=-p$ or $\bar p=1-p\equiv -p$ mod $\Lambda_{\tau}$, so $I(z; \bar p, \overline A, \tau)=I(z; p, -\overline A, \tau)$. Consequently,
\begin{equation}\label{eqfc-9}\triangle_1(-\overline{A})=\overline{\cos 2\pi s}=\overline{\triangle_1(A)},\quad \triangle_2(-\overline{A})=\overline{\cos 2\pi r}=\overline{\triangle_2(A)}.\end{equation}
Thus, $-\overline{A}\in \sigma_j$ if and only if $A\in\sigma_j$, namely $\sigma_j$ is symmetric with respect to the imaginary axis.
\end{proof}


\begin{lemma}\label{Lemma3-5} Let $\tau=ib$ with $b>0$ and $p\in (0,\frac12)$. Then
\begin{itemize}
\item[(1)] $\sigma_1=\sigma_{1,\infty}\cup[A_1, A_3]\cup [A_2, A_0]$, where $\sigma_{1,\infty}$ is an unbounded simple curve tending to $-\frac{2p\eta_1-\zeta(2p)}{2}\pm i\infty$, which is symmetric with respect to the real axis with $\sigma_{1,\infty}\cap\mathbb{R}$ containing a single point (this single point might be contained in $[A_1, A_3]\cup [A_2, A_0]$). 

\item[(2)] $\sigma_2=(-\infty, A_1]\cup [A_3, A_2]\cup [A_0, +\infty)$.
\end{itemize}

\end{lemma}

\begin{proof}
For $p\in (0,\frac12)$, we have $2p\eta_1-\zeta(2p)\in \mathbb{R}$ and
$$\frac{2p\eta_2-\tau\zeta(2p)}{2\tau}=p\Big(\eta_1-\frac{2\pi}{b}\Big)-\frac{\zeta(2p)}{2}\in\mathbb{R}.$$
Together with Lemmas \ref{lemma2-7}, \ref{lemma2-8}, \ref{lemma3-3}, \ref{lemma3-4} and Corollary \ref{coro2-11}, we obtain the following facts:

\begin{itemize}
\item[(a)] $A_1<A_3<A_2<A_0$ are all the endpoints of $\sigma_j$;

\item[(b)] $\sigma_j$ is symmetric with respect to the
real line $\mathbb{R}$;

\item[(c)] any connected component of
$\mathbb{C}\setminus \sigma_j$ is unbounded;

\item[(d)] for $R>0$ large, $\sigma_2\setminus B_R$ consists of exactly two unbounded arcs that tend to $+\infty$ and $-\infty$ separately, and $\sigma_1\setminus B_R$ consists of exactly two unbounded arcs that tend to $-\frac{2p\eta_1-\zeta(2p)}{2}+ i\infty$ and $-\frac{2p\eta_1-\zeta(2p)}{2}-i\infty$ separately.
\end{itemize}
Then $\sigma_2\subset \mathbb{R}$, and so there are at most two semi-arcs of $\sigma_2$ meeting at each $A_j$. This, together with the fact
$$d_j(A_k)=\text{ord}_{A_k}(\triangle_j(\cdot)^2-1)\quad\text{is odd},$$
yields that $d_2(A_k)=1$ for all $k$, namely there is exactly one semi-arc of $\sigma_2$ ending at $A_k$, which finally implies $\sigma_2=(-\infty, A_1]\cup [A_3, A_2]\cup [A_0, +\infty)$.
A similar argument implies the assertion (1).
\end{proof}

\begin{lemma}\label{Lemma3-6} Let $\tau=ib$ with $b>0$ and $p\in (\frac{\tau}{2}, \frac{1+\tau}2)$. Then
\begin{itemize}
\item[(1)] $\sigma_1=\sigma_{1,\infty}\cup[A_3, A_1]\cup [A_0, A_2]$, where $\sigma_{1,\infty}$ is an unbounded simple curve tending to $-\frac{(2p-\tau)\eta_1-\zeta(2p-\tau)}{2}\pm i\infty$, which is symmetric with respect to the real axis with $\sigma_{1,\infty}\cap\mathbb{R}$ containing a single point (this single point might be contained in $[A_3, A_1]\cup [A_0, A_2]$). 

\item[(2)] $\sigma_2=(-\infty, A_3]\cup [A_1, A_0]\cup [A_2, +\infty)$.
\end{itemize}

\end{lemma}

\begin{proof}
For $p\in (\frac{\tau}{2}, \frac{1+\tau}2)$, i.e. $p=t+\frac{\tau}{2}$ for $t\in (0,\frac12)$, we have \begin{align*}2p\eta_1-\zeta(2p)&=2t\eta_1+\tau\eta_1-\zeta(2t)-\eta_2\\&=2t\eta_1-\zeta(2t)+2\pi i\in 2t\eta_1-\zeta(2t)+ i\mathbb{R},\end{align*} 
\begin{align*}
\frac{2p\eta_2-\tau\zeta(2p)}{2\tau}=t\Big(\eta_1-\frac{2\pi}{b}\Big)-\frac{\zeta(2t)}{2}\in\mathbb{R}.
\end{align*}
Consequently, this lemma can be proved similarly as Lemma \ref{Lemma3-5}.
\end{proof}

\begin{lemma}\label{Lemma3-7} Let $\tau=ib$ with $b>0$ and $p\in (0,\frac\tau2)$. Then
\begin{itemize}
\item[(1)] $\sigma_1=(-i\infty, A_0]\cup [A_1, A_3]\cup [A_2, +i\infty)\subset i\mathbb{R}$.

\item[(2)] $\sigma_2=\sigma_{2,\infty}\cup[A_0, A_1]\cup [A_3, A_2]$, where $\sigma_{2,\infty}$ is an unbounded simple curve tending to $-\frac{2p\eta_2-\tau\zeta(2p)}{2\tau}\pm \infty$, which is symmetric with respect to the imaginary axis with $\sigma_{2,\infty}\cap i\mathbb{R}$ containing a single point (this single point might be contained in $[A_0, A_1]\cup [A_3, A_2]$). 
\end{itemize}

\end{lemma}

\begin{proof}
For $p\in (0, \frac{\tau}{2})$, we have $2p\eta_1-\zeta(2p)\in i\mathbb{R}$ and
\begin{align*}
\frac{2p\eta_2-\tau\zeta(2p)}{2\tau}=p\Big(\eta_1-\frac{2\pi}{b}\Big)-\frac{\zeta(2p)}{2}\in i\mathbb{R}.
\end{align*}
Consequently, this lemma can be proved similarly as Lemma \ref{Lemma3-5}.
\end{proof}

\begin{lemma}\label{Lemma3-8} Let $\tau=ib$ with $b>0$ and $p\in (\frac12,\frac{1+\tau}2)$. Then
\begin{itemize}
\item[(1)] $\sigma_1=(-i\infty, A_1]\cup [A_0, A_2]\cup [A_3, +i\infty)\subset i\mathbb{R}$.

\item[(2)] $\sigma_2=\sigma_{2,\infty}\cup[A_1, A_0]\cup [A_2, A_3]$, where $\sigma_{2,\infty}$ is an unbounded simple curve tending to $-\frac{(2p-1)\eta_2-\tau\zeta(2p-1)}{2\tau}\pm \infty$, which is symmetric with respect to the imaginary axis with $\sigma_{2,\infty}\cap i\mathbb{R}$ containing a single point (this single point might be contained in $[A_1, A_0]\cup [A_2, A_3]$). 
\end{itemize}

\end{lemma}

\begin{proof}
For $p\in (\frac12, \frac{1+\tau}{2})$, i.e. $p=\frac12+it$ for $t\in (0,\frac12)$, we have \begin{align*}2p\eta_1-\zeta(2p)=2it\eta_1-\zeta(2it)\in i\mathbb{R},\end{align*} 
\begin{align*}
\frac{2p\eta_2-\tau\zeta(2p)}{2\tau}=-\frac{\pi}{b}+it\Big(\eta_1-\frac{2\pi}{b}\Big)-\frac{\zeta(2it)}{2}\in it\Big(\eta_1-\frac{2\pi}{b}\Big)-\frac{\zeta(2it)}{2}+\mathbb{R}.
\end{align*}
Consequently, this lemma can be proved similarly as Lemma \ref{Lemma3-5}.
\end{proof}

\begin{corollary}\label{coro-s4-1}
Let $\tau=ib$ with $b>0$, $p\notin E_{\tau}[2]$ and $\wp(p)\in\mathbb{R}$. Then $\sigma_j$ has at most one cusp for $j\in\{1,2\}$, and  $\sigma_1\cap\sigma_2\setminus\{A_0, A_1, A_2, A_3\}$ consists of at most one point.
\end{corollary}

\begin{proof} Again we may assume $$p\in\Big(0,\frac{1}{2}\Big)\cup \Big(\frac{1}{2},\frac{1+\tau}{2}\Big)\cup \Big(\frac{\tau}{2},\frac{1+\tau}{2}\Big)\cup\Big(0,\frac{\tau}{2}\Big).$$
Consequently, the assertions follow directly from Lemmas \ref{Lemma3-5}-\ref{Lemma3-8}.
\end{proof}

\section{Rectangular torus: Proof of main results}

\label{section5}

In this section, we continue to consider $\tau=ib$ with $b>0$, and give the proofs of Theorems \ref{main-thm-b-2} and \ref{main-thm-b-22}.

First, we need to prove the following weaker version of Theorem \ref{main-thm-b-2}.

\begin{theorem}
\label{thm3-10} Let $\tau=ib$ with $b>0$, $p\in E_{\tau}\setminus E_{\tau}[2]$ and $\wp(p)\in\mathbb{R}$. Then $G_p(z)$ has at most one pair of nontrivial critical points.

Conversely,
suppose the unique pair of nontrivial critical points $\pm a_0=\pm(r_0+s_0\tau)$ exists, where $(r_0,s_0)\in\mathbb{R}^2\setminus\frac12\mathbb{Z}^2$. Then $\pm a_0$ are non-degenerate critical points and the following statements hold.
\begin{itemize}
\item[(1)] For  $p\in (0, \frac12)\cup(\frac\tau2,\frac{1+\tau}{2})$, we have $a_0=\overline{a_0}$ and $s_0\in\frac12\mathbb{Z}$. 
\item[(2)] For $p\in  (0, \frac{\tau}{2})\cup (\frac12, \frac{1+\tau}{2})$, we have $a_0=-\overline{a_0}$ and $r_0\in\frac12\mathbb{Z}$.
\end{itemize}
\end{theorem}

\begin{proof} Again we may assume $$p\in\Big(0,\frac{1}{2}\Big)\cup \Big(\frac{1}{2},\frac{1+\tau}{2}\Big)\cup \Big(\frac{\tau}{2},\frac{1+\tau}{2}\Big)\cup\Big(0,\frac{\tau}{2}\Big).$$
It follows from Theorem \ref{rmk2-8} and Corollary \ref{coro-s4-1} that $G_p(z)$ has at most one pair of nontrivial critical points.

Suppose
the unique pair of nontrivial critical points $\pm a_0=\pm(r_0+s_0\tau)$ exists with $(r_0,s_0)\in\mathbb{R}^2\setminus\frac12\mathbb{Z}^2$, and denote $\{ A_4\}=\sigma_1\cap\sigma_2\setminus\{A_0, A_1, A_2, A_3\}$. Then
\begin{equation}\label{eqfc-11}
\zeta(a_0+p)+\zeta(a_0-p)-2r_0\eta_1-2s_0\eta_2=0,
\end{equation}
\begin{equation}\label{eqfc-12}A_4=\frac{1}{2}\left[  \zeta(p+a_0)+\zeta(p-a_0)-\zeta(2p)\right]. 
\end{equation}
Since $\bar p=p$ or $\bar p=-p$ or $\bar p=1-p\equiv -p$ or $\bar p=p-\tau\equiv p$ mod $\Lambda_\tau$, we see from \eqref{eqfc-11}, $\eta_1\in\mathbb{R}$ and $\eta_2\in i\mathbb{R}$ that
$$\zeta(\overline{a_0}+p)+\zeta(\overline{a_0}-p)-2r_0\eta_1+2s_0\eta_2=0.$$
Thus, $\overline{a_0}=r_0-s_0\tau$ is also a nontrivial critical point of $G_p(z)$, so $\overline{a_0}=a_0$ or $\overline{a_0}=-a_0$ in $E_{\tau}$, or equivalently, $s_0\in\frac12\mathbb{Z}$ or $r_0\in\frac{1}{2}\mathbb{Z}$. 

To prove the non-degeneracy of $a_0$, we recall Part I \cite{CFL} that the Hessian of $G_p(z)$ at $a_0$ is given by
\begin{align}\label{eqfc-15}&4\pi^2\det D^2G_p(a_0)\nonumber\\
=&\frac{\pi^2}{(\operatorname{Im}\tau)^2}-\left|\frac{\wp(a_0+p)+\wp(a_0-p)+2\eta_1}{2}-\frac{\pi}{\operatorname{Im}\tau}\right|^2.\end{align}
Denote
\begin{align}
\alpha:=\frac{\wp(a_0+p)+\wp(a_0-p)+2\eta_1}{2}.
\end{align}
It follows from $\overline{a_0}=\pm a_0$ that $\bar\alpha=\alpha\in\mathbb{R}$. Then by \eqref{eqfc-15}, we see that to prove $\det D^2G_p(a_0)\neq 0$ is equivalent to prove
\begin{align}\label{alpha}
\alpha\neq 0\quad\text{and}\quad\alpha\neq \frac{2\pi}{b}.
\end{align}

{\bf Case 1.} $p\in (0, \frac12)\cup(\frac\tau2,\frac{1+\tau}{2})$.

Then it follows from Lemmas \ref{Lemma3-5}-\ref{Lemma3-6} that $A_4\in \sigma_{1,\infty}\cap\sigma_2\subset\mathbb{R}$ and $A_4$ is an inner point of $\sigma_2$. Note from $A_4\in\sigma_1$ that $\triangle_1(A_4)\in [-1,1]$.
Also recall from \eqref{eqfc-8} that $\triangle_1(A)\in \mathbb{R}$ for any $A\in \mathbb{R}$. If $\triangle_1(A_4)\in (-1,1)$, then there is small $\delta>0$ such that for any $A\in(A_4-\delta, A_4+\delta)\subset \sigma_2$, there holds $\triangle_1(A)\in (-1,1)$, which implies $A\in\sigma_1\cap\sigma_2$ for any $A\in(A_4-\delta, A_4+\delta)$, a contradiction. Therefore, $\triangle_1(A_4)\in \{\pm 1\}$ and so $s_0\in\frac12\mathbb{Z}$, i.e. $\overline{a_0}=a_0$.

Recalling Lemmas \ref{Lemma3-5}-\ref{Lemma3-6} that $A_4$ is only an inner point of $\sigma_{1,\infty}$ among the arcs of $\sigma_1$, namely $A_4$ is not a branch point of $\sigma_1$, it follows from Remark \ref{rmk2-6} that
$$\text{ord}_{A_4}(\triangle_1(\cdot)^2-1)=2,$$
namely $\triangle_1'(A_4)=0$ and $\triangle_1''(A_4)\neq 0$. Note that $Q(A_4)\neq 0$. Then by \eqref{triangle1} and $s_0\in\frac12\mathbb{Z}$, we easily obtain that
$$0\neq 2\triangle_1''(A_4)=\frac{a'(A_4)^2}{4}(e^{-2\pi is_0}+e^{2\pi is_0})(\wp(a_0+p)+\wp(a_0-p)+2\eta_1)^2,$$
so $\alpha=\frac{\wp(a_0+p)+\wp(a_0-p)+2\eta_1}{2}\neq 0$.

On the other hand, $s_0\in\frac12\mathbb{Z}$ implies $r_0\in \mathbb{R}\setminus\frac12\mathbb Z$ (because $a_0\notin E_{\tau}[2]$), so $\triangle_2(A_4)\in (-1,1)$. Since Lemmas \ref{Lemma3-5}-\ref{Lemma3-6} says that $A_4$ can not be a branch point of $\sigma_2$, we have $\triangle_2'(A_4)\neq 0$, which together with \eqref{triangle2} imply that
\begin{align*}
0\neq \tau\wp(a_0+p)+\tau\wp(a_0-p)+2\eta_2=\tau\wp(a_0+p)+\tau\wp(a_0-p)+2\tau\eta_1-4\pi i.
\end{align*}
This yields $\alpha\neq \frac{2\pi}{b}$ by using $\tau=ib$, i.e. \eqref{alpha} holds and so $\det D^2G_p(a_0)\neq 0$.


{\bf Case 2.} $p\in  (0, \frac{\tau}{2})\cup (\frac12, \frac{1+\tau}{2})$.

Then it follows from Lemmas \ref{Lemma3-7}-\ref{Lemma3-8} that $A_4\in \sigma_1\cap\sigma_{2,\infty}\subset i\mathbb{R}$ and $A_4$ is an inner point of $\sigma_1$. Note from $A_4\in\sigma_2$ that $\triangle_2(A_4)\in [-1,1]$. Also recall from
 \eqref{eqfc-9} that $\triangle_2(A)\in \mathbb{R}$ for any $A\in i\mathbb{R}$. If $\triangle_2(A_4)\in (-1,1)$, then there is small $\delta>0$ such that for any $A\in(A_4-i\delta, A_4+i\delta)\subset \sigma_1$, there holds $\triangle_2(A)\in (-1,1)$, which implies $A\in\sigma_1\cap\sigma_2$ for any $A\in(A_4-i\delta, A_4+i\delta)$, a contradiction. Therefore, $\triangle_2(A_4)\in \{\pm 1\}$ and so $r_0\in\frac12\mathbb{Z}$, i.e. $\overline{a_0}=-a_0$.
Then a similar argument as Case 1 also implies that \eqref{alpha} holds and so $\det D^2G_p(a_0)\neq 0$.


The proof is complete.
\end{proof}

To prove Theorems \ref{main-thm-b-2} and \ref{main-thm-b-22}, we also need the following results.

\begin{theorem}\cite[Theorem 3.5]{CFL}\label{thm-B}
Let $\tau=ib$ with $b>0$. Then there is $\varepsilon>0$ small such that if $|p-\frac{\omega_k}{2}|<\varepsilon$ for some $k\in \{0,1,2,3\}$, then $G_p(z)$ has no nontrivial critical points, and all trivial critical points of $G_p(z)$ are non-degenerate.
\end{theorem}
\begin{lemma}\cite[Lemma 4.4]{CFL}\label{lemma50-0}
Fix $\tau$ and let $p_0\in E_{\tau}\setminus E_{\tau}[2]$. Suppose $G_{p_0}(z)$ has exactly $N\geq 4$ critical points which are all non-degenerate. Then there is $\varepsilon>0$ small such that for any $|p-p_0|<\varepsilon$, $G_{p}(z)$ has also exactly $N$ critical points that are all non-degenerate.
\end{lemma}

\begin{remark}
For any non-degenerate critical point $q$ of $G_p(z)$, it is well known that the local degree (denoted by $\deg_p(q)$) of $\nabla G_p(z)$ at $z=q$ is $1$ (resp. $-1$) if $\det D^2G_p(q)>0$ (resp. $\det D^2G_p(q)<0$). Suppose that all critical points of $G_p(z)$ are non-degenerate, then we proved in Part I \cite[(1.11)]{CFL} that 
\begin{equation}\label{deg-count}
\sum_{\text{$q$ is a critical point of $G_p$}}\deg_p(q)=-2.
\end{equation}
\end{remark}

\begin{lemma}\label{lemma3-10}
Let $\tau=ib$ with $b>0$. Then the following statements hold.
\begin{itemize}
\item[(1)]$G_{\frac14}(z)$ has a unique pair of nontrivial critical points $\pm(\frac14+\frac{\tau}{2})$, and all critical points of $G_{\frac14}(z)$ are non-degenerate. 
\item[(2)]$G_{\frac{\tau}{4}}$ has a unique pair of nontrivial critical points $\pm(\frac12+\frac{\tau}{4})$, and all critical points of $G_{\frac{\tau}{4}}(z)$ are non-degenerate. 
\item[(3)]$G_{\frac14+\frac{\tau}{2}}(z)$ has a unique pair of nontrivial critical points $\pm\frac14$, and all critical points of $G_{\frac14+\frac{\tau}{2}}(z)$ are non-degenerate.
\item[(4)]$G_{\frac12+\frac{\tau}{4}}(z)$ has a unique pair of nontrivial critical points $\pm\frac\tau4$, and all critical points of $G_{\frac12+\frac{\tau}{4}}(z)$ are non-degenerate.
\item[(5)] If $b\in (0, 2b_0)\cup(2b_1,+\infty)$, then there exists $\varepsilon>0$ small such that for any $|p-\frac{1+\tau}{4}|<\varepsilon$, $G_{p}(z)$ has exactly $3$ pairs of nontrivial critical points, or equivalently, has exactly $10$ critical points that are all non-degenerate.
\end{itemize}
\end{lemma}

\begin{proof} The statements (1)-(2) were proved in Part I \cite{CFL}. Here we apply a similar argument as Part I \cite{CFL} to prove (3)-(5).

(3). By Theorem \ref{thm-LW},  the Green function $G(z;\frac12,\tau)$ on the torus $E_{\frac12,\tau}:=\mathbb{C}/(\mathbb Z\frac12+\mathbb Z{\tau})$ has exactly $3$ critical points $\frac14, \frac{\tau}{2}, \frac{1}{4}+\frac{\tau}{2}$ that are all non-degenerate. 

Since $G(z)=G(z;\tau)$ is doubly periodic with periods $1$ and $\tau$, we see that $G_{\frac14+\frac{\tau}{2}}(z)=\frac12(G(z+\frac14+\frac{\tau}{2})+G(z-\frac14-\frac{\tau}{2}))$ satisfies
\begin{align*}G_{\frac14+\frac{\tau}{2}}\Big(z+\frac{1}{2}\Big)&=\frac{G(z+\frac34+\frac{\tau}{2})+G(z+\frac14-\frac{\tau}{2})}{2}\\
&=\frac{G(z-\frac14-\frac{\tau}{2})+G(z+\frac14+\frac{\tau}{2})}{2}=G_{\frac14+\frac{\tau}{2}}(z),\end{align*}
namely $G_{\frac14+\frac{\tau}{2}}(z)$ is doubly periodic with periods $\frac12$ and $\tau$, so $G_{\frac14+\frac{\tau}{2}}(z)$ is well-defined on $E_{\frac12,\tau}$. Furthermore, since $\frac14+\frac{\tau}{2}=-(\frac14+\frac{\tau}{2})$ in $E_{\frac{1}{2},\tau}$, it follows that
$$-\Delta (2G_{\frac14+\frac{\tau}{2}}(z))=\delta_{\frac14+\frac{\tau}{2}}-\frac{1}{\left \vert E_{\frac{1}{2},\tau}\right \vert }\text{
\ on }E_{\frac{1}{2},\tau},$$
i.e. $2G_{\frac14+\frac{\tau}{2}}(z)$ is the Green function of $E_{\frac{1}{2},\tau}$ with singularity at $\frac14+\frac{\tau}{2}$. By the uniqueness of the Green function up to adding a constant, it follows that
 there is a constant $C$ such that 
$$2G_{\frac14+\frac{\tau}{2}}(z)=G\Big(z-\frac14-\frac{\tau}{2}; \frac{1}{2},\tau\Big)+C.$$ 
Consequently, $G_{\frac14+\frac{\tau}{2}}(z)$ has exactly $3$ non-degenerate critical points $\frac{\tau}{2}, \frac14, 0$ on $E_{\frac{1}{2},\tau}$.
Since $z=z+\frac{1}{2}$ in $E_{\frac{1}{2},\tau}$ but $z\neq z+\frac{1}{2}$ in $E_{\tau}$, we finally conclude that 
$G_{\frac14+\frac{\tau}{2}}(z)$ has exactly $6$ critical points
$ 0,\frac12,\frac{\tau}{2},\frac{1+\tau}{2},\pm\frac14
$  on $E_{{\tau}}$
which are all non-degenerate.

(4). By Theorem \ref{thm-LW}, the Green function $G(z;\frac{\tau}{2})$ (i.e. the Green function on the torus $E_{\frac{\tau}{2}}=\mathbb{C}/(\mathbb Z+\mathbb Z\frac{\tau}{2})$) has exactly $3$ critical points $\frac12, \frac{\tau}{4}, \frac{1}{2}+\frac{\tau}{4}$ that are all non-degenerate. Then a similar argument as (3) implies that $G_{\frac12+\frac{\tau}{4}}(z)$ has exactly $6$ critical points $ 0,\frac12,\frac{\tau}{2},\frac{1+\tau}{2},\pm\frac\tau4
$ that are all non-degenerate.

(5) Let $b\in (0, 2b_0)\cup(2b_1,+\infty)$. Then $\tau':=\frac{1+\tau}{2}=\frac12+i \frac{b}{2}$ with $\frac{b}{2}\in (0, b_0)\cup (b_1, +\infty)$, so it follows from Theorem \ref{thm-LW} that the Green function $G(z; \tau')$ has exactly $5$ critical points on $E_{\tau'}$, which are all non-degenerate by using Theorem \ref{thm-A0}. On the other hand, it is easy to see that $G_{\frac{1+\tau}{4}}(z)=G_{\frac{\tau'}{2}}(z)$ is doubly periodic with periods $1$ and $\tau'$, so $G_{\frac{1+\tau}{4}}(z)$ is well-defined on $E_{\tau'}$. Consequently, a similar argument as (3) implies that $G_{\frac{1+\tau}{4}}(z)$ has exactly $10$ critical points on $E_{\tau}$, which are all non-degenerate. Then by Lemma \ref{lemma50-0}, there exists $\varepsilon>0$ small such that for any $|p-\frac{1+\tau}{4}|<\varepsilon$, $G_{p}(z)$ has exactly $10$ critical points that are all non-degenerate. The proof is complete.
\end{proof}

\begin{lemma}\label{lemma3-22} Let $\tau=ib$ with $b>0$. Then $\partial\mathcal{B}_k\cap\mathbb{R}=\{d_{k,1}, d_{k,2}\}$ with $d_{k,1}<d_{k,2}$ and
\begin{align}
\label{efc-1}&d_{0,1}=-\eta_1,\qquad\qquad\qquad\qquad\;\; d_{0,2}=\frac{2\pi}{b}-\eta_1,\\
&d_{k,1}=e_k+\frac{3e_k^2-\frac{g_2}{4}}{\frac{2\pi}{b}-(e_k+\eta_1)},\quad\;\;\; d_{k,2}=e_k-\frac{3e_k^2-\frac{g_2}{4}}{e_k+\eta_1},\qquad k=1,2,3.\nonumber
\end{align}
Furthermore, if $b=1$, then
\begin{equation}\label{ee11ee}
d_{1,1}<d_{3,1}<d_{0,1}<d_{1,2}<d_{2,1}<d_{0,2}<d_{3,2}<d_{2,2}.
\end{equation}
\end{lemma}

\begin{proof} Since $\tau=ib$ with $b>0$, we have $e_1, e_2, e_3, g_2, \eta_1\in\mathbb{R}$ and $e_2<e_3<e_1$. Recall \eqref{B00}-\eqref{alphak1} that (Recall from Section 1 that $\mathcal{B}_k$'s are all open disks for $\tau=ib$ with $b>0$)
\[
\mathcal{B}_0=\Big\{z\in\mathbb{C}\; :\; \Big|z-\Big(\frac{\pi}{b}-\eta_1\Big)\Big|<\frac{\pi}{b}\Big\},
\]
and for $k\in\{1,2,3\}$,
\begin{equation}\label{al0phak1}
\mathcal{B}_k=\bigg\{z\in\mathbb{C}\; :\; \bigg|z-e_k-\frac{\alpha_k}{\alpha_k^2-\beta_k^2}\bigg|<\frac{\beta_k}{\left|\alpha_k^2-\beta_k^2\right|}\bigg\}\quad\text{since }|\alpha_k|\neq \beta_k,
\end{equation}
where
$$
\alpha_k=\frac{\frac{\pi}{b}-(\eta_1+e_k)}{3e_k^2-\frac{g_2}{4}}\in\mathbb R,\quad \beta_k=\frac{\pi}{b|3e_k^2-\frac{g_2}{4}|}>0.
$$
This proves $\partial\mathcal{B}_0\cap\mathbb{R}=\{d_{0,1}, d_{0,2}\}$ and \eqref{efc-1}. For later usage, we recall Part I \cite[Remark 3.2]{CFL} that
\begin{equation}\label{eq330}\det D^2 G_p(0)>0\;\Leftrightarrow\;\wp(p)\in \mathcal{B}_0,\end{equation}
$$\det D^2 G_p(0)<0\;\Leftrightarrow\;\wp(p)\in \mathbb{C}\setminus\overline{\mathcal{B}_0},$$
$$\det D^2 G_p(0)=0\;\Leftrightarrow\;\wp(p)\in \partial\mathcal{B}_0.$$

Since $-\frac{g_2}{4}=e_1e_2+e_1e_3+e_2e_3$ and $e_1+e_2+e_3=0$, by denoting $\{i,j,k\}=\{1,2,3\}$, we easily obtain
\[
3e_k^2-\frac{g_2}{4}=2e_k^2+e_ie_j=(e_k-e_i)(e_k-e_j)\begin{cases}>0\quad\text{if }k=1,2\\
<0\quad\text{if }k=3.\end{cases}
\]

{\bf Case 1.} We consider $k=1,2$.

Then Theorem \ref{thm-LW} says that $\frac{\omega_k}{2}$ is a non-degenerate saddle point of $G(z)$, so it follows from Part I \cite[Lemma 3.3]{CFL} that $|\alpha_k|>\beta_k$ and
\begin{equation}\label{eq331}
\det D^2 G_p\Big(\frac{\omega_k}{2}\Big)>0\;\Leftrightarrow\;\wp(p)\in \mathcal{B}_k,
\end{equation}
$$\det D^2 G_p\Big(\frac{\omega_k}{2}\Big)<0\;\Leftrightarrow\;\wp(p)\in \mathbb C\setminus \overline{\mathcal{B}_k},$$
$$\det D^2 G_p\Big(\frac{\omega_k}{2}\Big)=0\;\Leftrightarrow\;\wp(p)\in \partial\mathcal{B}_k.$$
Consequently, $\beta_k=\frac{\pi}{b(3e_k^2-\frac{g_2}{4})}$ and
$$\mathcal{B}_k=\bigg\{z\in\mathbb{C}\; :\; \bigg|z-e_k-\frac{\alpha_k}{\alpha_k^2-\beta_k^2}\bigg|<\frac{\beta_k}{\alpha_k^2-\beta_k^2}\bigg\}.$$
This implies $\partial\mathcal{B}_k\cap\mathbb{R}=\{d_{k,1}, d_{k,2}\}$ with $d_{k,1}<d_{k,2}$ and
$$d_{k,1}=e_k+\frac{\alpha_k-\beta_k}{\alpha_k^2-\beta_k^2}=e_k+\frac{3e_k^2-\frac{g_2}{4}}{\frac{2\pi}{b}-(e_k+\eta_1)},$$
$$d_{k,2}=e_k+\frac{\alpha_k+\beta_k}{\alpha_k^2-\beta_k^2}=e_k-\frac{3e_k^2-\frac{g_2}{4}}{e_k+\eta_1}.$$

{\bf Case 2.} We consider $k=3$.

Then Theorem \ref{thm-LW} says that $\frac{\omega_3}{2}$ is a non-degenerate minimal point of $G(z)$, so it follows from Part I \cite[Lemma 3.3]{CFL} that $|\alpha_3|<\beta_3$ and
\begin{equation}\label{eq332}
\det D^2 G_p\Big(\frac{\omega_3}{2}\Big)>0\;\Leftrightarrow\;\wp(p)\in \mathbb C\setminus \overline{\mathcal{B}_3},
\end{equation}
$$\det D^2 G_p\Big(\frac{\omega_3}{2}\Big)<0\;\Leftrightarrow\;\wp(p)\in  {\mathcal{B}_3},$$
$$\det D^2 G_p\Big(\frac{\omega_3}{2}\Big)=0\;\Leftrightarrow\;\wp(p)\in \partial\mathcal{B}_3.$$
Consequently, $\beta_3=-\frac{\pi}{b(3e_3^2-\frac{g_2}{4})}$ and
$$\mathcal{B}_3=\bigg\{z\in\mathbb{C}\; :\; \bigg|z-e_3-\frac{\alpha_3}{\alpha_3^2-\beta_3^2}\bigg|<\frac{-\beta_3}{\alpha_3^2-\beta_3^2}\bigg\}.$$
This implies $\partial\mathcal{B}_3\cap\mathbb{R}=\{d_{3,1}, d_{3,2}\}$ with $d_{3,1}<d_{3,2}$ and
$$d_{3,1}=e_3+\frac{\alpha_3+\beta_3}{\alpha_3^2-\beta_3^2}=e_3+\frac{3e_3^2-\frac{g_2}{4}}{\frac{2\pi}{b}-(e_3+\eta_1)},$$
$$d_{3,2}=e_3+\frac{\alpha_3-\beta_3}{\alpha_3^2-\beta_3^2}=e_3-\frac{3e_3^2-\frac{g_2}{4}}{e_3+\eta_1}.$$

Finally, when $b=1$, it is well known that $\eta_1=\pi$, $e_3=0$, $e_1=-e_2\approx 2.18844\pi$, so $\partial\mathcal{B}_k$'s can be computed numerically; see \cite[Section 5]{CFL} for the details, and the figures of $\partial\mathcal{B}_k$'s is seen in Figure 1. Clearly \eqref{ee11ee} follows directly from Figure 1.
\end{proof}

\begin{lemma}\label{lemma3-22} Let $\tau=ib$ with $b>0$. Then $(\cup_k\partial \mathcal{B}_k)\cap \mathbb{R}$ contains exactly $8$ points $\{d_{k,j} : k=0,1,2,3, j=1,2\}$ with \eqref{ee11ee} holding for all $b>0$, and
\begin{equation}\label{ektau}\Big\{e_1, e_2, e_3, \wp\Big(\frac14\Big), \wp\Big(\frac{\tau}{4}\Big), \wp\Big(\frac{1}{4}+\frac{\tau}2\Big), \wp\Big(\frac{1}{2}+\frac{\tau}{4}\Big)\Big\}\cap(\cup_k\partial \mathcal{B}_k)=\emptyset.\end{equation}
In particular, the relative positions of the four circles $\partial\mathcal{B}_k$'s are the same as Figure 1 for all $b>0$.
\end{lemma}

\begin{proof}
First, we prove \eqref{ektau}. Note from Lemma \ref{lemma3-10} and Theorem \ref{main-thm-01} that $\wp(\frac14), \wp(\frac{\tau}{4})$,  $\wp(\frac{1}{4}+\frac{\tau}2)$, $\wp(\frac{1}{2}+\frac{\tau}{4})\notin\cup_k\partial\mathcal{B}_k$.
Now we prove
$e_k\notin\cup_l\partial\mathcal{B}_l$.
It is known from Theorem \ref{main-thm-01} that $\frac{\omega_k}{2}$ is a degenerate critical point of $G(z)=G_0(z)$ if and only if $e_k\in\partial \mathcal{B}_0$. Since $\tau=ib$ with $b>0$, we see from Theorem \ref{thm-LW} that $e_k\notin\partial \mathcal{B}_0$ for all $k$.
Fix any $k\in\{1,2,3\}$ and note from the definition \eqref{alphak1} of $\mathcal{B}_k$ that $e_k\notin \partial \mathcal{B}_k$. If $e_k\in \partial \mathcal{B}_l$ for some $l\in\{1,2,3\}\setminus\{k\}$, then Theorem \ref{main-thm-01} implies $e_{l'}=\wp(\frac{\omega_k-\omega_l}{2})\in \partial \mathcal{B}_0$ (where $\{l',l,k\}=\{1,2,3\}$), a contradiction. This proves that $e_k\notin\cup_l\partial\mathcal{B}_l$, so \eqref{ektau} holds.

By Lemma \ref{lemma3-22}, to prove $(\cup_k\partial \mathcal{B}_k)\cap \mathbb{R}$ contains exactly $8$ points $\{d_{k,j} : k=0,1,2,3, j=1,2\}$ is equivalent to prove $d_{k,j}\neq d_{l,j'}$ for any $k\neq l$ and $j,j'$. Hence, it suffices to prove that $\partial\mathcal{B}_k\cap\partial\mathcal{B}_l\cap\mathbb{R}=\emptyset$ for any $k\neq l$.

Suppose there are $k\neq l$ and $d\in\mathbb{R}$ such that $d\in\partial\mathcal{B}_k\cap\partial\mathcal{B}_l$. Then there is a unique $p\in(0,\frac{1}{2})\cup (\frac{1}{2},\frac{1+\tau}{2})\cup (\frac{\tau}{2},\frac{1+\tau}{2})\cup (0,\frac{\tau}{2})$ such that $\wp(p)=d$. Consequently, both $\frac{\omega_k}{2}$ and $\frac{\omega_l}{2}$ are degenerate critical points of $G_p(z)$. Recall \eqref{eqfc-15} that $$0=4\pi^2\det D^2G_p\left(\frac{\omega_{k'}}{2}\right)=\frac{\pi^2}{(\operatorname{Im}\tau)^2}-\left|{\wp\left(p-\frac{\omega_{k'}}{2}\right)+\eta_1}-\frac{\pi}{\operatorname{Im}\tau}\right|^2,\; k'=k,l.$$
Since $\wp(p-\frac{\omega_{k'}}{2}), \eta_1\in\mathbb{R}$, we obtain
\begin{align}\label{eqfc-45}\wp\left(p-\frac{\omega_{k'}}{2}\right)+\eta_1=0 \text{ or }\frac{2\pi}{b},\quad\text{for }\; k'=k,l.\end{align}

{\bf Case 1.} $p\in (0, \frac12)\cup (\frac{\tau}{2}, \frac{1+\tau}2)$.

Then Lemmas \ref{Lemma3-5}-\ref{Lemma3-6} imply that $\sigma_2$ has no cusps, so 
we see from Lemma \ref{lemma2-10} and $\tau=ib$ that $\wp(p-\frac{\omega_{k'}}{2})+\eta_1\neq\frac{2\pi i}{\tau}=\frac{2\pi}{b}$ for any $k'$. Then \eqref{eqfc-45} implies $\wp(p-\frac{\omega_{k'}}{2})+\eta_1=0$ for $k'\in\{k,l\}$, from which and Lemma \ref{lemma2-10} we see that $A_k$ and $A_l$ are both cusps of $\sigma_1$,
a contradiction with Corollary \ref{coro-s4-1}.

{\bf Case 2.} $p\in (0, \frac{\tau}2)\cup (\frac{1}{2}, \frac{1+\tau}2)$.

Then Lemmas \ref{Lemma3-7}-\ref{Lemma3-8} imply that $\sigma_1$ has no cusps, so 
we see from Lemma \ref{lemma2-10} that $\wp(p-\frac{\omega_{k'}}{2})+\eta_1\neq0$ for any $k'$. Then \eqref{eqfc-45} implies $\wp(p-\frac{\omega_{k'}}{2})+\eta_1=\frac{2\pi}{b}=\frac{2\pi i}{\tau}$ for $k'\in\{k,l\}$, so Lemma \ref{lemma2-10} implies that $A_k$ and $A_l$ are both cusps of $\sigma_2$, a contradiction with Corollary \ref{coro-s4-1}.

This proves $\partial\mathcal{B}_k\cap\partial\mathcal{B}_l\cap\mathbb{R}=\emptyset$ for any $k\neq l$, so $d_{k,j}\neq d_{l,j'}$ for any $k\neq l$ and $j,j'$. Together with the fact that \eqref{ee11ee} holds for $b=1$, we see from the continuity that \eqref{ee11ee} holds for all $b>0$.
\end{proof}

Now we are ready to prove Theorems \ref{main-thm-b-2} and \ref{main-thm-b-22}.

\begin{proof}[Proof of Theorems \ref{main-thm-b-2} and \ref{main-thm-b-22}] Note from Lemma \ref{Lemma3-1} that
\begin{equation}\label{eqfc-20}\wp\Big(\frac\tau4\Big)<e_2<\wp\Big(\frac{1}{4}+\frac{\tau}2\Big)<e_3
<\wp\Big(\frac{1}{2}+\frac{\tau}{4}\Big)<e_1<\wp\Big(\frac14\Big).\end{equation}
Recall Lemma \ref{lemma3-22} that $(\cup_k\partial \mathcal{B}_k)\cap \mathbb{R}$ contains exactly $8$ points, denoted by
$$d_1<d_2<\cdots<d_8,$$
and $\{e_1, e_2, e_3, \wp(\frac14), \wp(\frac{\tau}{4}), \wp(\frac{1}{4}+\frac{\tau}2), \wp(\frac{1}{2}+\frac{\tau}{4})\}\cap\{d_j\}_{j=1}^8=\emptyset$.
Then \eqref{ee11ee} yields
$$d_1=d_{1,1},\quad d_2=d_{3,1},\quad d_3=d_{0,1},\quad d_4=d_{1,2},$$
$$d_5=d_{2,1},\quad d_6=d_{0,2},\quad d_7=d_{3,2},\quad d_8=d_{2,2}.$$
This together with Lemma \ref{lemma3-22} implies \eqref{fineq-1}-\eqref{fineq-4}.

{\bf Step 1.} Denote $d_0=-\infty$ and $d_{9}=+\infty$. We prove that for each $1\leq j\leq 9$, the number of pairs of nontrivial critical points of $G_p(z)$ is the same constant $n_j\in\{0,1\}$ for any $\wp(p)\in (d_{j-1}, d_j)$.

Define
$$\Omega_{j,0}:=\Big\{t\in (d_{j-1}, d_j) \;:\;\text{$G_p(z)$ has no nontrivial critical points for $\wp(p)=t$}\Big\},$$
$$\Omega_{j,1}:=\bigg\{t\in (d_{j-1}, d_j) \;:\;\begin{array}{l}\text{$G_p(z)$ has a unique pair of nontrivial }\\\text{critical points for $\wp(p)=t$}\end{array}\bigg\}.$$
Recall the following facts:
\begin{enumerate}
\item[(a)]
Theorem \ref{thm-LW} says that $G_{\frac{\omega_k}{2}}(z)=G(z-\frac{\omega_k}{2})$ has no nontrivial critical points and all trivial critical points of $G_{\frac{\omega_k}{2}}(z)$ are non-degenerate.
\item[(b)] Lemma \ref{lemma3-10} says that $G_p(z)$ has a unique pair of nontrivial critical points for $p\in\{\frac14, \frac{\tau}{4}, \frac{1}{4}+\frac{\tau}2, \frac{1}{2}+\frac{\tau}4\}$.
\item[(c)] For any $\wp(p)\in (d_{j-1}, d_j)$, it follows from $\wp(p)\notin\cup_l\partial\mathcal{B}_l$ that all trivial critical points of $G_{p}(z)$ are non-degenerate.
\item[(d)] For any $\wp(p)\in\mathbb{R}\setminus\{e_1,e_2,e_3\}$, Theorem \ref{thm3-10} says that $G_p(z)$ has at most one pair of nontrivial critical points, which are non-degenerate if exists.
\end{enumerate}
Then $\Omega_{j,0}\cup \Omega_{j,1}=(d_{j-1},d_{j})$. Furthermore, together with Theorem \ref{thm-B} and Lemma \ref{lemma50-0}, we see that both $\Omega_{j,0}$ and $\Omega_{j,1}$ are open, so either $\Omega_{j,0}=(d_{j-1}, d_j)$ or $\Omega_{j,1}=(d_{j-1}, d_j)$.
This completes the proof of Step 1.

{\bf Step 2.} Fix $1\leq j\leq 8$ and let $\wp(p_j)=d_j$. We show that $G_{p_j}(z)$ has no nontrivial critical points.

Assume by contradiction that $G_{p_j}(z)$ has a unique pair of nontrivial critical points $\pm q_0$, then they are non-degenerate. So by the implicit function theorem,  $G_{p}(z)$ has a pair of nontrivial critical points $\pm q_p$ in a small neighborhood of $\pm q_0$ for any $p$ close to $p_j$.

On the other hand, note that $\wp(p_j)=d_j\in\partial\mathcal{B}_{k_j}$ for some $k_j\in\{0,1,2,3\}$, so ${\omega_{k_j}}/{2}$ is a degenerate critical point of $G_{p_j}(z)$. Recall \eqref{eq330}, \eqref{eq331} and \eqref{eq332} that $\det D^2G_p({\omega_{k_j}}/{2})$ changes sign when $\wp(p)$ crosses $\partial\mathcal{B}_{k_j}$, so the Brouwer degree of $\nabla G_p(z)$ contributed by the critical point ${\omega_{k_j}}/{2}$ changes between $1$ and $-1$ when  $\wp(p)$ crosses $\partial\mathcal{B}_{k_j}$. By the invariance of the local degree and (d), we conclude that the number of the pairs of nontrivial critical points in a small neighborhood of ${\omega_{k_j}}/{2}$ changes between $0$ and $1$ when $\wp(p)$ crosses $d_j\in\partial\mathcal{B}_{k_j}$ along $\mathbb{R}$. That is, there is $p$ close to $p_j$ satisfying $\wp(p)\in\mathbb{R}$ such that $G_p(z)$ has a pair of nontrivial critical points in a small neighborhood of  ${\omega_{k_j}}/{2}$ and so has two pairs of nontrivial critical points, a contradiction with (d). This proves that $G_{p_j}(z)$ has no nontrivial critical points.

{\bf Step 3.} We complete the proof of Theorem \ref{main-thm-b-22}.

By Steps 1-2 and Theorem \ref{thm-B}, we know that when $\wp(p)\in (-\infty, d_1]\cup [d_8,+\infty)$, $G_p(z)$ has no nontrivial critical points. Then \eqref{eqfc-20} and (b) imply that $d_1<\wp(\frac{\tau}{4})$ and $\wp(\frac14)<d_8$. Let $I_j=(d_{j-1}, d_j)$ be the finite open intervals for $2\leq j\leq 8$. Then Step 1 says that the number of the pair of nontrivial critical points is the same number $n_j\in\{0,1\}$ for any $\wp(p)\in I_j$. From here, \eqref{eqfc-20}, (a) and (b), we easily see that each $I_j$ contains exactly one element of $\{e_1, e_2, e_3, \wp(\frac14), \wp(\frac{\tau}{4}), \wp(\frac{1}{4}+\frac{\tau}2), \wp(\frac{1}{2}+\frac{\tau}{4})\}$. 
 This proves \eqref{eqfc-21} and so completes the proof of Theorem \ref{main-thm-b-22}.
 
 {\bf Step 4.} We complete the proof of Theorem \ref{main-thm-b-2}.
 
It follows from \eqref{eqfc-21}, Steps 1-2, (a) and (b) that $G_p(z)$ has no nontrivial critical points for
$$\wp(p)\in (-\infty, d_1]\cup [d_2, d_3]\cup [d_4, d_5]\cup [d_6, d_7]\cup [d_8,+\infty),$$
but $G_p(z)$ has a unique pair of nontrivial critical points $\pm a_0$ that are always non-degenerate for
$$\wp(p)\in (d_1, d_2)\cup (d_3, d_4)\cup (d_5, d_6)\cup (d_7, d_8).$$

Furthermore, for $\wp(p)\in (d_1, d_2)$, we have $p\in (0, \frac{\tau}{2})$ and it follows from Theorem \ref{thm3-10} that the nontrivial critical point $a_0=r_0+s_0\tau$ of $G_p(z)$ satisfies $r_0\in\frac12\mathbb Z$. Since $\wp(\frac{\tau}{4})\in (d_1, d_2)$  and $G_{\frac{\tau}{4}}(z)$ has a unique pair of nontrivial critical points $\pm(\frac12+\frac{\tau}{4})$, i.e. $r_0=\frac12$ for $p=\frac{\tau}{4}$, it follows from the continuity that $r_0\equiv \frac12$ for any $\wp(p)\in (d_1, d_2)$. This proves the assertion (2-1). The other (2-2), (2-3) and (2-4) can be proved similarly.

We have already proved that $\pm a_0$ are non-degenerate, i.e. $\det D^2G_p(\pm a_0)\neq 0$.
It remains to prove that $\pm a_0$ are saddle points. Let us take $\wp(p)\in (d_1, d_2)$ for example. Recall Lemma \ref{lemma3-22} that the relative positions of the four circles $\partial\mathcal{B}_k$'s are the same as Figure 1 for all $b>0$. It follows from $d_1=d_{1,1}$, $d_2=d_{3,1}$ and Figure 1 that $\wp(p)\in (d_1, d_2)\subset \mathcal{B}_1\setminus \cup_{k\neq 1} \overline{\mathcal{B}_k}$. From here, \eqref{eq330}, \eqref{eq331} and \eqref{eq332}, we obtain
$$\det D^2G_p\Big(\frac{\omega_k}{2}\Big)>0\quad\text{for }k=1, 3,$$
$$\det D^2G_p\Big(\frac{\omega_k}{2}\Big)<0\quad\text{for }k=0, 2.$$
Thus, we conclude from \eqref{deg-count} that $\det D^2G_p(\pm a_0)<0$, namely $\pm a_0$ are non-degenerate saddle points. The cases $\wp(p)\in (d_3, d_4)\cup (d_5, d_6)\cup (d_7, d_8)$ can be proved similarly.
 The proof is complete.
\end{proof}

\section{Proofs of several theorems}

\label{section7}

This section is devoted to the proofs of Theorems \ref{main-thm-b-23-0}, \ref{main-thm-8},  \ref{thm-section5-5} and \ref{main-thm-b-23} by using Theorems \ref{main-thm-b-2} and \ref{main-thm-b-22}. First, we prove Theorems \ref{main-thm-b-23-0} and \ref{main-thm-b-23}.

\begin{proof}[Proof of Theorems \ref{main-thm-b-23-0} and \ref{main-thm-b-23}] Let $\tau=ib$ with $b>0$. First, we prove Theorem \ref{main-thm-b-23}.
Let $(r,s)\in [-\frac12, \frac12]\times [0,\frac12]\setminus\frac12\mathbb{Z}^2$ and 
recall \eqref{513-1} that
$$
\wp(p_{r,s}(\tau))=\wp (r+s\tau)+\frac{\wp ^{\prime }(r+s\tau)}{%
2(\zeta(r+s\tau)-r\eta_1-s\eta_2)},
$$
or equivalently, $\pm(r+s\tau)$ is a pair of nontrivial critical points of $G_{p_{r,s}(\tau)}(z)$. Since Theorem \ref{thm-LW}-(1) says that $G_0(z)=G(z)$ has no nontrivial critical points, we have $p_{r,s}(\tau)\neq 0$ in $E_{\tau}$, namely $\wp(p_{r,s}(\tau))\neq \infty$ and so $\wp(p_{r,s}(\tau))\in \mathbb C$.

(1) Let $(r,s)\in \partial I\cup \partial II\setminus\frac12\mathbb{Z}^2$, then $\overline{r+s\tau}=\pm (r+s\tau)$ in $E_{\tau}$, so it follows from \eqref{barwp}, $\eta_1\in\mathbb R$ and $\eta_2\in i\mathbb R$ that
\begin{align*}\overline{\wp(p_{r,s}(\tau))}&=\wp (\overline{r+s\tau})+\frac{\wp ^{\prime }(\overline{r+s\tau})}{%
2(\zeta(\overline{r+s\tau})-r\eta_1+s\eta_2)}\\
&=\wp (r+s\tau)+\frac{\wp ^{\prime }(r+s\tau)}{%
2(\zeta(r+s\tau)-r\eta_1-s\eta_2)}=\wp(p_{r,s}(\tau)),\end{align*}
namely $\wp(p_{r,s}(\tau))\in\mathbb{R}$.

(2)-(3) Let $(r,s)\in I^\circ\cup II^{\circ}$. Since $\pm(r+s\tau)$ is a pair of nontrivial critical points of $G_{p_{r,s}(\tau)}(z)$, we see from Theorem \ref{main-thm-b-2} that
\begin{equation}\label{wpp}
\wp(p_{r,s}(\tau))\in\mathbb C\setminus\mathbb R,\quad\forall (r,s)\in I^\circ\cup II^{\circ}.
\end{equation}
Take $(r,s)=(Cs,s)$ with $C\in\mathbb{R}\setminus\{0\}$ and $s>0$. Then it follows from \cite[Theorem 4.5]{CKL-PAMQ} that
$$\lim\limits_{s\to 0+}\wp(p_{Cs,s}(\tau))=-\frac{C\eta_1+\eta_2}{C+\tau}=\frac{2\pi i}{C+\tau}-\eta_1=\frac{2\pi i(C-ib)}{|C+\tau|^2}-\eta_1,$$
namely
$$\lim\limits_{s\to 0+}\operatorname{Im}\wp(p_{Cs,s}(\tau))=\frac{2\pi C}{|C+\tau|^2}.$$
Thus for $C>0$, we have $(Cs, s)\in I^\circ$ and $\operatorname{Im}\wp(p_{Cs,s}(\tau))>0$ for $s>0$ small. Then by \eqref{wpp} and the continuity, we obtain $\operatorname{Im}\wp(p_{r,s}(\tau))>0$ for all $(r,s)\in I^\circ$.

Similarly, for $C<0$, we have $(Cs, s)\in II^\circ$ and $\operatorname{Im}\wp(p_{Cs,s}(\tau))<0$ for $s>0$ small, so $\operatorname{Im}\wp(p_{r,s}(\tau))<0$ for all $(r,s)\in II^\circ$. 
This proves Theorem \ref{main-thm-b-23}.

It suffices to prove the equivalence between Theorem \ref{main-thm-b-23-0} and Theorem \ref{main-thm-b-23}. Suppose $\pm a$ is a pair of nontrivial critical point of $G_p(z)$. By replacing $a$ with $-a$ if necessary, we may assume $a=r+s\tau$ with $(r,s)\in [-\frac12, \frac12]\times [0,\frac12]\setminus\frac12\mathbb{Z}^2$. Then $\wp(p)=\wp(p_{r,s}(\tau))$.

By \eqref{sdsd}, we see that $\wp(a)\in\mathbb{R}$ if and only if $(r,s)\in \partial I\cup \partial II\setminus\frac12\mathbb{Z}^2$. Thus, 
$$(r,s)\in I^\circ\cup II^\circ\quad\text{if and only if}\quad \operatorname{Im}\wp(a)\neq 0.$$
When $I^\circ\cup II^\circ\ni (r,s)\to (0,0)$, we have
$$\wp(a)=\frac{1}{a^2}(1+O(|a|^4))=\frac{r^2-s^2b^2-2rsbi}{|r+s\tau|^4}(1+O(|a|^4)).$$
Thus, $\operatorname{Im}\wp(a)<0$ when $I^\circ\ni (r,s)\to (0,0)$, and $\operatorname{Im}\wp(a)>0$ when $II^\circ\ni (r,s)\to (0,0)$. Then by the continuity, we obtain
 $$(r,s)\in I^\circ\quad\text{if and only if}\quad \operatorname{Im}\wp(a)< 0,$$
 $$(r,s)\in II^\circ\quad\text{if and only if}\quad \operatorname{Im}\wp(a)> 0.$$
This proves the equivalence between Theorem \ref{main-thm-b-23-0} and Theorem \ref{main-thm-b-23}, so Theorem \ref{main-thm-b-23-0} holds.
\end{proof}

To prove Theorem \ref{main-thm-8}, we recall \cite[Theorem 1.10]{CFL}.

\begin{theorem}\cite{CFL}\label{thm-section5}
Fix $\tau$. Let $\Xi$ be a connected component of the open set $\mathbb{C}\setminus (\{e_1,e_2,e_3\}\cup\cup_{k=0}^3\partial \mathcal{B}_k)$, and define
$$m(\Xi):=\#\Big\{\frac{\omega_k}{2}\;: \;0\leq k\leq 3,\; \det D^2G_p\Big(\frac{\omega_k}{2}\Big)>0\;\text{for }p\in\wp^{-1}(\Xi)\Big\},$$ which is a constant independent of $p\in \wp^{-1}(\Xi)$ by Theorem \ref{main-thm-01}. Then $m(\Xi)\leq 2$, and
\begin{itemize}
\item[(1)] if $m(\Xi)=0$, then $G_p(z)$ has at least $6$ critical points for any $\wp(p)\in\Xi$, and $G_p(z)$ has exactly either $6$ or $10$ critical points that are all non-degenerate for almost all $\wp(p)\in \Xi$. 
\item[(2)] if $m(\Xi)=1$,  then $G_p(z)$ has exactly either $4$ or $8$ critical points that are all non-degenerate for almost all $\wp(p)\in \Xi$.
\item[(3)] if $m(\Xi)=2$, then $G_p(z)$ has exactly $6$ critical points for any $\wp(p)\in\Xi$, and critical points are all non-degenerate for almost all $\wp(p)\in \Xi$.
\end{itemize}
\end{theorem}

Theorem \ref{thm-section5} was proved by using the degree counting formula \eqref{deg-count} in Part I \cite{CFL}, where 
we also proved that Theorem \ref{thm-section5}-(1) is sharp (i.e. there are $\tau$ and a connected component $\Xi$ satisfying $m(\Xi)=0$ such that the number of critical points of $G_p(z)$ is exactly $6$ for some $\wp(p)\in \Xi$, while the number is exactly $10$ for some other $\wp(p)\in \Xi$). However, whether Theorem \ref{thm-section5}-(2) is sharp or not remained open there.
Now Theorem \ref{thm-section5-5} of this paper shows that Theorem \ref{thm-section5}-(2) is also sharp.

\begin{proof}[Proof of Theorem \ref{main-thm-8}]
Let $\tau=ib$ with $b>0$. Then Theorem \ref{main-thm-b-22} shows that 
the relative positions of the four circles $\partial\mathcal{B}_k$'s are the same as Figure 1 for all $b>0$. Consequently, we see from Figure 1 that all the sets $\Xi_j$'s defined in Theorem \ref{main-thm-8} are nonempty open subsets of  $\mathbb{C}\setminus (\{e_1,e_2,e_3\}\cup\cup_{k=0}^3\partial \mathcal{B}_k)$. Moreover, \eqref{eq330}, \eqref{eq331} and \eqref{eq332} together imply that
$$m(\Xi_1)=m(\Xi_2)=m(\Xi_3)=m(\Xi_4)=2,$$
$$m(\Xi_5)=m(\Xi_6)=m(\Xi_7)=m(\Xi_8)=1, \quad m(\Xi_9)=0.$$
Thus, Theorem \ref{main-thm-8} follows directly from Theorem \ref{thm-section5}.
\end{proof}

\begin{proof}[Proof of Theorem \ref{thm-section5-5}]
Let $\tau=ib$ with $b>0$, and recall Theorem \ref{main-thm-8} that
$$\Xi_1= \mathcal{B}_1\setminus \overline{\mathcal{B}_3},\quad \Xi_2=\mathcal{B}_0\cap\mathcal{B}_1.$$
Then $\Xi_1, \Xi_2\neq \emptyset$ are both open, and $G_p(z)$ has a unique pair of nontrivial critical points for any $\wp(p)\in \Xi_1\cup\Xi_2$.

Recalling  Theorem \ref{main-thm-b-23}, 
we consider the analytic map $f: I^\circ=(0,\frac12)^2\to \mathbb{C}^+:=\{z\in\mathbb{C} : \operatorname{Im} z>0\}$ defined by 
$$f(r,s):=\wp(p_{r,s}(\tau))=\wp (r+s\tau)+\frac{\wp ^{\prime }(r+s\tau)}{%
2(\zeta(r+s\tau)-r\eta_1-s\eta_2)}.$$
Take $\wp(p_j)\in \Xi_j\cap\mathbb{C}^+$ for $j=1,2$. Then there is a unique $(r_j, s_j)\in I^\circ$ such that $\pm(r_j+s_j\tau)$ is the unique pair of nontrivial critical points of $G_{p_j}(z)$, i.e. $f(r_j, s_j)=\wp(p_j)$ for $j=1,2$.

Denote $\{\wp(p_0)\}:=\partial \mathcal{B}_1\cap \partial \mathcal{B}_3\cap\mathbb{C}^+\subset\partial \Xi_1$, then $f^{-1}(\wp(p_0))\cap I^\circ$ consists of at most $3$ points because $G_{p_0}(z)$ has at most $3$ pairs of nontrivial critical points. Therefore, we can take a path-connected loop $\ell\subset I^\circ\setminus f^{-1}(\wp(p_0))$ connecting $(r_1, s_1)$ and $(r_2, s_2)$. Then $f(\ell)\subset \mathbb{C}^+\setminus\{\wp(p_0)\}$ is a path-connected loop connecting $\wp(p_1)\in \Xi_1$ and $\wp(p_2)\in \Xi_2\subset \mathbb C\setminus\overline{\Xi_1}$. 
Consequently, it follows from Figure 1 that there exists $(r_3, s_3)\in \ell$ such that $\wp(p_3):=f(r_3, s_3)$ satisfies 
$$\wp(p_3)\in \Xi_5=\mathbb{C}\setminus\cup_{k=0}^3 \overline{\mathcal{B}_k}\quad\text{or }\;\wp(p_3)\in\Xi_7=\mathcal{B}_1\cap\mathcal{B}_3\setminus(\{e_2\}\cup\overline{\mathcal{B}_0}).$$

{\bf Case 1.} $\wp(p_3)\in \Xi_5=\mathbb{C}\setminus\cup_{k=0}^3 \overline{\mathcal{B}_k}$.

In this case, we define
$$\Xi:=\Xi_5=\mathbb{C}\setminus\cup_{k=0}^3 \overline{\mathcal{B}_k}.$$
It follows from Theorem \ref{main-thm-b-22} and Figure 1 that $\Xi$ is a connected component of $\mathbb{C}\setminus (\{e_1,e_2,e_3\}\cup\cup_{k=0}^3\partial \mathcal{B}_k)$, and
$$(-\infty, d_1)\cup(d_8,+\infty)\subset \Xi.$$
Recalling 
\[\Omega_N:=\bigg\{\wp(p)\in\Xi\,:\begin{array}{l}\text{$G_p(z)$ has exactly $N$ critical points}\\\text{that are all non-degenerate}\end{array}\bigg\}\;\text{for }N=4,8,\]
we see from Theorem \ref{main-thm-b-2} that $(-\infty, d_1)\cup(d_8,+\infty)\subset\Omega_4$, so $\Omega_4\neq\emptyset$ is an open subset of $\Xi$ by using Lemma \ref{lemma50-0}.

To prove $\Omega_8\neq \emptyset$, we note from $\wp(p_3)=f(r_3, s_3)\in \Xi$ that there is a small $\delta>0$ such that $f(B_{\delta})\subset \Xi$, where $$B_{\delta}:=\big\{(r,s)\in I^\circ : |(r,s)-(r_3, s_3)|<\delta\big\}.$$ Then for any $\wp(p)\in f(B_{\delta})\subset \Xi$, there is $(r,s)\in B_{\delta}\subset I^\circ$ such that $\wp(p)=f(r,s)$, namely $\pm(r+s\tau)$ is a pair of nontrivial critical points of $G_p(z)$, so $G_p(z)$ has at least $6$ critical points for any  $\wp(p)\in f(B_{\delta})\subset \Xi$. Since $f(B_{\delta})$ is of positive Lebegue measure, we see from Theorem \ref{main-thm-8}-(2) that $\Omega_8\neq\emptyset$, and so $\Omega_8$ is an open subset of $\Xi$ by using Lemma \ref{lemma50-0}. 

Furthermore,  
it follows from Theorem \ref{main-thm-8}-(2) that $\Xi\subset\overline{\Omega_4}\cup\overline{\Omega_{8}}$. Since $\Xi$ is connected, we see that $\overline{\Omega_4}\cap\overline{\Omega_{8}}\cap\Xi\neq\emptyset$, which implies $\partial\Omega_4\cap\partial\Omega_{8}\cap \Xi\neq \emptyset$. Finally, take any $\wp(p)\in (\partial\Omega_4\cup\partial\Omega_{8})\cap\Xi$. If all critical points of $G_p(z)$ are non-degenerate, then Lemma \ref{lemma50-0} implies $\wp(p)\in\Omega_4\cup\Omega_{8}$, a contradiction. Thus $G_p(z)$ has degenerate nontrivial critical points.

{\bf Case 2.} $\wp(p_3)\in\Xi_7=\mathcal{B}_1\cap\mathcal{B}_3\setminus(\{e_2\}\cup\overline{\mathcal{B}_0})$.

In this case, we define
$$\Xi:=\Xi_7=\mathcal{B}_1\cap\mathcal{B}_3\setminus(\{e_2\}\cup\overline{\mathcal{B}_0}).$$
It follows from Theorem \ref{main-thm-b-22} and Figure 1 that $\Xi$ is a connected component of $\mathbb{C}\setminus (\{e_1,e_2,e_3\}\cup\cup_{k=0}^3\partial \mathcal{B}_k)$, and
$$(d_2, d_3)\setminus\{e_2\}\subset \Xi.$$
Again, it follows from Theorem \ref{main-thm-b-2} and Lemma \ref{lemma50-0} that $(d_2, d_3)\setminus\{e_2\}\subset\Omega_4$, i.e. $\Omega_4\neq\emptyset$ is an open subset of $\Xi$.
The rest proof is similar to Case 1 and is omitted here. 
The proof is complete.
\end{proof}

\section{Other applications}

\label{section6}

As mentioned in Section 1.2, it was shown by Hitchin \cite{Hit1} that the case with one of $(r,s)$ real and the other purely imaginary has important applications to Einstein metrics. Thus, we are also interested in the following question: Given $(\tau, p)$, how many $A$ are there such that the correpsonding $(r,s)\in \mathbb{R}\times (i\mathbb{R}\setminus\{0\})$ or $(r,s)\in (i\mathbb{R}\setminus\{0\})\times \mathbb{R}$?  Notice that if $s\in i\mathbb{R}\setminus\{0\}$, then
$$\triangle_1(A)=\frac12(e^{2\pi is}+e^{-2\pi is})>1,$$
and if $r\in i\mathbb{R}\setminus\{0\}$, then
$$\triangle_2(A)=\frac12(e^{2\pi ir}+e^{-2\pi ir})>1.$$
This motivates us to define
$$\sigma_j^*:=\triangle_j^{-1}((1,+\infty)),\quad j=1,2.$$
Then 
\begin{align*}&\text{$(r,s)\in \mathbb{R}\times (i\mathbb{R}\setminus\{0\})$ if and only if $A\in\sigma_1^*\cap \sigma_2$,}\\
&\text{$(r,s)\in (i\mathbb{R}\setminus\{0\})\times\mathbb{R}$ if and only if $A\in\sigma_1\cap \sigma_2^*$. }\end{align*}
We will see from Lemmas \ref{lemma2-7-2}-\ref{lemma2-8-2} that $\sigma_j^*$ consists of countably many analyitc arcs, so $\sigma_1^*\cap \sigma_{2}$ (resp. $\sigma_1\cap \sigma_2^*$) contains at most countably many points for generic $(\tau, p)$. 
The main result of this section is to prove that for the special case $\tau=ib$ and $\wp(p)\in\mathbb{R}$, $\sigma_1^*\cap \sigma_{2}$ (resp. $\sigma_1\cap \sigma_2^*$) contains open intervals.

\begin{theorem}\label{thm4-1}
Let $\tau=ib$ with $b>0$. 
\begin{itemize}
\item[(a)] If $p\in (0,\frac12)$, then
\begin{equation}\label{12-sigma1}\sigma_1^*\cap\sigma_2=(-\infty, A_1)\cup (A_0, +\infty)\setminus
\{\text{at most one point}\},\end{equation}
\begin{equation}\label{12-sigma2}\sigma_1\cap\sigma_2^*\cap\mathbb{R}=(A_2, A_0)\setminus\{\text{at most one point}\},\; \sigma_1\cap\sigma_2^*\setminus\mathbb{R}\;\text{is at most finite}.\end{equation}
\item[(b)] If $p\in (\frac\tau2,\frac{1+\tau}2)$, then
$$\sigma_1^*\cap\sigma_2=(A_1, A_0)\setminus\{\text{at most one point}\},$$
$$\sigma_1\cap\sigma_2^*\cap\mathbb{R}=(A_0, A_2)\setminus\{\text{at most one point}\},\; \sigma_1\cap\sigma_2^*\setminus\mathbb{R}\;\text{is at most finite}.$$
\item[(c)] If $p\in (0,\frac\tau2)$, then
\begin{equation}\label{123-sigma1}\sigma_1\cap\sigma_2^*=(-i\infty, A_0)\cup(A_2,+i\infty)\setminus\{\text{at most one point}\},\end{equation}
\begin{equation}\label{123-sigma2}\sigma_1^*\cap\sigma_2\cap i\mathbb{R}=(A_0, A_1)\setminus\{\text{at most one point}\},\; \sigma_1^*\cap\sigma_2\setminus i\mathbb{R}\;\text{is at most finite}.\end{equation}
\item[(d)] If $p\in (\frac12,\frac{1+\tau}2)$, then
$$\sigma_1\cap\sigma_2^*=(A_0, A_2)\setminus\{\text{at most one point}\},$$
$$\sigma_1^*\cap\sigma_2\cap i\mathbb{R}=(A_1, A_0)\setminus\{\text{at most one point}\},\; \sigma_1^*\cap\sigma_2\setminus i\mathbb{R}\;\text{is at most finite}.$$
\end{itemize}
\end{theorem}

The rest of this section is to prove Theorem \ref{thm4-1}.
For convenience, we also define
$$\overline{\sigma}_j^*:=\triangle_j^{-1}([1,+\infty)),\quad j=1,2.$$
Then $\overline{\sigma}_j^*\setminus\sigma_j^*=\triangle_j^{-1}(1)$ is a countable set.
Recall \eqref{eqfc-vjk1} that \begin{equation}\label{eqfc-aoo}\triangle_1(A_0)=\triangle_1(A_1)=1,\quad\triangle_2(A_0)=\triangle_2(A_2)=1.\end{equation} 

\begin{lemma}\label{lemma2-7-2} 
\begin{itemize}
\item[(1)]
 $\{A_0, A_1\}$ are precisely the set of endpoints of $\overline{\sigma}_1^*$, and  $\{A_0, A_2\}$ are precisely the set of endpoints of $\overline{\sigma}_2^*$.
 \item[(2)] For $j\in \{1,2\}$, any connected component of $\mathbb{C}\setminus\overline{\sigma}_j^*$ is unbounded.
 \end{itemize}
\end{lemma}

\begin{proof}
The proof is the same as that of Lemma \ref{lemma2-7}.
\end{proof}

\begin{lemma}\label{lemma2-8-2}
Denote $B_R:=\{A\in\mathbb{C} : |A|<R\}$. Then for $R>0$ large, the following statements hold.
\begin{itemize}
\item[(1)] $\overline{\sigma}_1^*\setminus{B_R}$ consists of countably many disjoint analytic arcs that can be parametrized by 
\begin{equation}\label{sigma1-2}
A=t-\frac{2p\eta_1-\zeta(2p)}{2}+(2m+\epsilon_1)\pi i+O(t^{-1})
,\; \end{equation}
with $t\in(-\infty, -t_{1,m}]\cup[t_{2,m},+\infty)$ for some $t_{j,m}>0$ large, where $m\in\mathbb{Z}$ can be arbitrary, and $\epsilon_1\in\{0,1\}$ is independent of $m$.

\item[(2)] $\overline{\sigma}_2^*\setminus{B_R}$ consists of countably many disjoint analytic arcs that can be parametrized by 
\begin{equation}\label{sigma2-2}
A=\frac{t}{\tau}-\frac{2p\eta_2-\tau\zeta(2p)}{2\tau}+\frac{(2m+\epsilon_2)\pi i}{\tau}+O(t^{-1})
,\end{equation}
with $t\in(-\infty, -t_{3,m}]\cup[t_{4,m},+\infty)$ for some $t_{j,m}>0$ large, where $m\in\mathbb{Z}$ can be arbitrary, and $\epsilon_2\in\{0,1\}$ is independent of $m$.
\end{itemize}

\end{lemma}

\begin{proof}
Recall the proof of Lemma \ref{lemma2-8} that
\begin{align*}e^{-2\pi i s}=\exp\left(\varepsilon_3\left(A\left(1+\frac{2p\eta_1-\zeta(2p)}{2}A^{-1}+O(A^{-2})\right)+\frac{1-\varepsilon_4}{2}\pi i\right)\right),
\end{align*}
\begin{align*}e^{2\pi i r}
=\exp\left(\varepsilon_3\left(A\left(\tau+\frac{2p\eta_2-\tau\zeta(2p)}{2}A^{-1}+O(A^{-2})\right)+\frac{1-\varepsilon_5}{2}\pi i\right)\right).
\end{align*}
where $\varepsilon_j\in\{\pm 1\}$. Hence, $\triangle_1(A)=\frac12(e^{-2\pi i s}+e^{2\pi i s})\geq 1$ if and only if $e^{-2\pi i s}>0$, if and only if
$$A\left(1+\frac{2p\eta_1-\zeta(2p)}{2}A^{-1}+O(A^{-2})\right)+\frac{1-\varepsilon_4}{2}\pi i=t+2m\pi i\quad\text{with}\;t\in\mathbb{R},$$
which is equivalent to \eqref{sigma1-2}.

Similarly, $\triangle_2(A)=\frac12(e^{2\pi i r}+e^{-2\pi i r})\geq 1$ if and only if $e^{-2\pi i r}>0$, if and only if
$$A\left(\tau+\frac{2p\eta_2-\tau\zeta(2p)}{2}A^{-1}+O(A^{-2})\right)+\frac{1-\varepsilon_5}{2}\pi i=t+2m\pi i\quad\text{with}\;t\in\mathbb{R},$$
which is equivalent to \eqref{sigma2-2}. The rest proof is similar to that of Lemma \ref{lemma2-8}.
\end{proof}

\begin{lemma}\label{lemma3-4-2}
Let $\tau=ib$ with $b>0$.
\begin{itemize}
\item[(1)] If $p\in (0, \frac12)\cup (\frac{\tau}{2}, \frac{1+\tau}2)$, then $\overline{\sigma}^*_j$ is symmetric with respect to the real axis for each $j=1,2$.
\item[(2)] If $p\in  (0, \frac{\tau}{2})\cup(\frac12, \frac{1+\tau}{2})$, then $\overline{\sigma}^*_j$ is symmetric with respect to the imaginary axis for each $j=1,2$.
\end{itemize}

\end{lemma}

\begin{proof}
The proof is the same as that of Lemma \ref{lemma3-4}.
\end{proof}

Now we are ready to prove Theorem \ref{thm4-1}.

\begin{proof}[Proof of Theorem \ref{thm4-1}] 
(a). Let $p\in (0, \frac12)$.
Recall
\eqref{eqfc-8} that $$\triangle_j(A)\in \mathbb{R},\quad\forall A\in \mathbb{R}.$$
Then it follows from \eqref{eqfc-aoo} and Lemma \ref{Lemma3-5} that
\begin{equation}\label{eqfc-31}(-\infty, A_1]\cup [A_0, +\infty)=\triangle_1^{-1}([1,+\infty))\cap\mathbb{R}=\overline{\sigma}_1^*\cap\mathbb{R},\end{equation}
\begin{equation}\label{eqfc-32}[A_2, A_0]=\triangle_2^{-1}([1,+\infty))\cap\mathbb{R}=\overline{\sigma}_2^*\cap\mathbb{R},\end{equation}
$$[A_3, A_2]=\triangle_1^{-1}((-\infty,-1])\cap\mathbb{R},\quad [A_1, A_3]=\triangle_2^{-1}((-\infty,-1])\cap\mathbb{R}.$$
Remark that \eqref{eqfc-31} implies that $\epsilon_1=0$ in Lemma \ref{lemma2-8-2}. 

By $\overline{\sigma}_j^*\setminus\sigma_j^*=\triangle_j^{-1}(1)$, \eqref{eqfc-31} and $\sigma_2=(-\infty, A_1]\cup [A_3, A_2]\cup [A_0, +\infty)$, we obtain
$$\sigma_1^*\cap\sigma_2=(-\infty, A_1)\cup (A_0, +\infty)\setminus\triangle_1^{-1}(1).$$
Since $(-\infty, A_1)\cup (A_0, +\infty)\cap\triangle_1^{-1}(1)\subset \sigma_1\cap\sigma_2\setminus\{A_k\}_k$ contains at most one point, we see that \eqref{12-sigma1} holds. A similar argument implies
$$\sigma_1\cap\sigma_2^*\cap\mathbb{R}=(A_2, A_0)\setminus\{\text{at most one point}\}.$$
Note from Lemma \ref{Lemma3-5}-(1) that
$$\sigma_1\cap\sigma_2^*\setminus\mathbb{R}\subset \sigma_{1,\infty}\cap \overline{\sigma}_2^*\setminus\mathbb{R},$$
where $\sigma_{1,\infty}$ is an unbounded simple curve tending to $\pm i\infty-\frac{2p\eta_1-\zeta(2p)}{2}$.
By Lemma \ref{lemma2-7-2}-(2) and Lemma \ref{lemma2-8-2}-(2), we know that $\overline{\sigma}_2^*\setminus\mathbb{R}$ contains countably many unbounded arcs tending to $\pm i\infty-\frac{2p\eta_2-\tau\zeta(2p)}{2\tau}+\frac{(2m+\epsilon_2)\pi i}{\tau}$, and any two of them can not have two intersection points. Furthermore, both $\sigma_{1,\infty}$ and $\overline{\sigma}_2^*\setminus\mathbb{R}$ are symmetric with respect to the real axis. Since $p\in (0, \frac12)$ implies
$$-\frac{2p\eta_1-\zeta(2p)}{2}\neq -\frac{2p\eta_2-\tau\zeta(2p)}{2\tau}+\frac{(2m+\epsilon_2)\pi i}{\tau},\quad\forall m\in\mathbb Z,$$
it is easy to see that $\sigma_{1,\infty}\cap \overline{\sigma}_2^*\setminus\mathbb R$ contains at most finite points, so $\sigma_1\cap\sigma_2^*\setminus\mathbb{R}$ is at most a finite set. This proves \eqref{12-sigma2}.

(b). The proof is similar to (a) and is omitted here.

(c). Recall
\eqref{eqfc-9} that $$\triangle_j(A)\in \mathbb{R},\quad\forall A\in i\mathbb{R}.$$
Then it follows from \eqref{eqfc-aoo} and Lemma \ref{Lemma3-7} that
\begin{equation}\label{eqfc-33}[A_0, A_1]=\triangle_1^{-1}([1,+\infty))\cap i\mathbb{R}=\overline{\sigma}_1^*\cap i\mathbb{R},\end{equation}
\begin{equation}\label{eqfc-34}(-i\infty, A_0]\cup[A_2,+i\infty)=\triangle_2^{-1}([1,+\infty))\cap i\mathbb{R}=\overline{\sigma}_2^*\cap i\mathbb{R},\end{equation}
$$[A_3, A_2]=\triangle_1^{-1}((-\infty,-1])\cap i\mathbb{R},\quad [A_1, A_3]=\triangle_2^{-1}((-\infty,-1])\cap i\mathbb{R}.$$
Remark that \eqref{eqfc-34} implies that $\epsilon_2=0$ in Lemma \ref{lemma2-8-2}. 
Then a similar argument as (a) implies \eqref{123-sigma1}-\eqref{123-sigma2}.

(d). The proof is similar to (c) and is omitted here. The proof is complete.
\end{proof}

\subsection*{Acknowledgements} Z. Chen was supported by National Key R\&D Program of China (No. 2023YFA1010002) and  NSFC (No. 12222109). 
E. Fu was supported by NSFC (No. 12401188) and BIMSA Start-up Research Fund.

\end{document}